\RequirePackage{fix-cm}
\documentclass[smallextended]{svjour3}       
\smartqed  

\usepackage{amsmath, amsfonts, amsthm, amssymb,amscd,amscd}
\usepackage{float,epsf}
\usepackage[english]{babel}
\usepackage{enumerate}
\usepackage{tikz}
\usepackage[normalem]{ulem}
\usepackage{graphicx}

\usepackage{srcltx}

\usepackage{bm}           
\usepackage{mathrsfs} 
\usepackage{mathabx}
\usepackage{amsxtra}
\usepackage[numbers,sort&compress]{natbib} 
\usepackage[perpage,symbol*]{footmisc}  

\usepackage[titletoc,title]{appendix}

\def\bbm{\begin{bmatrix}}
	\def\ebm{\end{bmatrix}}
\def\bpm{\begin{pmatrix}}
	\def\epm{\end{pmatrix}}
\def\bvm{\begin{vmatrix}}
	\def\evm{\end{vmatrix}}
\def\bin{\begin{enumerate}}
	\def\ein{\end{enumerate}}
\def\bit{\begin{itemize}}
	\def\eit{\end{itemize}}
\def\bid{\begin{description}}
	\def\eid{\end{description}}

\newcommand*\xbar[1]{%
	\hbox{%
		\vbox{%
			\hrule height 0.5pt 
			\kern0.5ex
			\hbox{%
				\kern-0.1em
				\ensuremath{#1}%
				\kern-0.1em
			}%
		}%
	}%
}

\newcommand{\VERTiii}[1]{{\left\vert\kern-0.3ex\left\vert\kern-0.3ex\left\vert #1
		\right\vert\kern-0.3ex\right\vert\kern-0.3ex\right\vert}}
\newcommand{\VERT}{\vert\kern-0.3ex\vert\kern-0.3ex\vert}
\newcommand{\VERTl}{\left\vert\kern-0.3ex\left\vert\kern-0.3ex\left\vert}
\newcommand{\VERTr}{\right\vert\kern-0.3ex\right\vert\kern-0.3ex\right\vert}
\newcommand{\VERTbig}{\big\vert\kern-0.3ex\big\vert\kern-0.3ex\big\vert}
\newcommand{\VERTBig}{\Big\vert\kern-0.3ex\Big\vert\kern-0.3ex\Big\vert}
\newcommand*{\medcap}{\mathbin{\scalebox{1.5}{\ensuremath{\cap}}}}%
\makeatletter
\DeclareFontFamily{OMX}{MnSymbolE}{}
\DeclareSymbolFont{MnLargeSymbols}{OMX}{MnSymbolE}{m}{n}
\SetSymbolFont{MnLargeSymbols}{bold}{OMX}{MnSymbolE}{b}{n}
\DeclareFontShape{OMX}{MnSymbolE}{m}{n}{
	<-6>  MnSymbolE5 <6-7>  MnSymbolE6 <7-8>  MnSymbolE7 <8-9>  MnSymbolE8 <9-10> MnSymbolE9 <10-12> MnSymbolE10 <12->   MnSymbolE12
}{}
\DeclareFontShape{OMX}{MnSymbolE}{b}{n}{
	<-6>  MnSymbolE-Bold5 <6-7>  MnSymbolE-Bold6 <7-8>  MnSymbolE-Bold7 <8-9>  MnSymbolE-Bold8 <9-10> MnSymbolE-Bold9 <10-12> MnSymbolE-Bold10 <12->   MnSymbolE-Bold12
}{}

\let\llangle\@undefined
\let\rrangle\@undefined
\DeclareMathDelimiter{\llangle}{\mathopen}%
{MnLargeSymbols}{'164}{MnLargeSymbols}{'164}
\DeclareMathDelimiter{\rrangle}{\mathclose}%
{MnLargeSymbols}{'171}{MnLargeSymbols}{'171}
\makeatother

\newcommand{\md}{\mathrm{d}}
\renewcommand{\div}{\mathrm{div}}

\newcommand{\mr}{\mathbb{R}}
\newcommand{\mR}{\mathbb{R}}

\newcommand{\deq}{:=}
\newcommand{\bt}{\boldsymbol{\tau}}
\newcommand{\bn}{\boldsymbol{\nu}}
\newcommand{\bh}{\boldsymbol{h}}
\newcommand{\xxi}{\boldsymbol{\xi}}
\newcommand{\be}{\mathbf{e}}

\newcommand{\p}{\partial}
\newcommand{\gr}{\nabla}

\newcommand{\cxi}{{\xi_1}}
\newcommand{\ceta}{{\xi_2}}

\renewcommand{\r}{\rho}
\newcommand{\g}{\gamma}

\newcommand{\eps}{\varepsilon}

\newcommand{\z}{\zeta}
\renewcommand{\O}{\Omega}

\newcommand{\E}{G}

\newcommand{\xx}{\mathbf{x}}
\newcommand{\vv}{\mathbf{v}}
\newcommand{\ee}{\mathbf{e}}
\newcommand{\pp}{\mathbf{p}}
\newcommand{\rr}{\mathbf{r}}

\newcommand{\Gw}{\Gamma_{\text{\rm wedge}}}

\newcommand{\Gso}{\Gamma_{\text{\rm sonic}}}

\newcommand{\Gsh}{\Gamma_{\text{\rm shock}}}

\newcommand{\dist}{{\rm dist}}

\newcommand{\mS}{{\mathcal S}}
\newcommand{\mO}{{\mathcal O}}
\newcommand{\mN}{{\mathcal N}}

\newcommand{\rc}{r}

\newcommand{\PtLwR}{{P_3}}

\newcommand{\Pd}{P_{\rm d}}

\newtheorem{thm}{Theorem}[section]
\newtheorem{lem}{Lemma}[section]
\newtheorem{prop}[lem]{Proposition}
\newtheorem{cor}[lem]{Corollary}
\newtheorem{defi}[lem]{Definition}
\newtheorem{rem}[lem]{Remark}
\newtheorem{claim}[lem]{Claim}
\numberwithin{equation}{section}
\numberwithin{figure}{section}

\numberwithin{equation}{section}
\numberwithin{thm}{section}
\numberwithin{lem}{section}

 \journalname{Archive for Rational Mechanics and Analysis, }
\begin{document}

\title{Convexity of
	Self-Similar Transonic Shocks\\ and Free Boundaries for the Euler Equations \\ for Potential Flow\thanks{The research of
		Gui-Qiang G. Chen was supported in part by
		the UK
		Engineering and Physical Sciences Research Council Award
		EP/E035027/1 and
		EP/L015811/1, and the Royal Society--Wolfson Research Merit Award (UK).
		The research of Mikhail Feldman was
		supported in part by the National Science Foundation under Grants DMS-1401490 and DMS-1764278,
		and the Van Vleck
		Professorship Research Award  by the University of Wisconsin-Madison.
		The research of Wei Xiang was supported in part by the UK EPSRC Science and Innovation
		Award to the Oxford Centre for Nonlinear PDE (EP/E035027/1),
		the CityU Start-Up Grant for New Faculty 7200429(MA),
		the Research Grants Council of the HKSAR,
		China (Project No. CityU 21305215, Project No. CityU 11332916, Project No. CityU 11304817,
		and Project No. CityU 11303518),
		and partly by
		the National Science Foundation Grant
		DMS-1101260 while visiting the University of Wisconsin-Madison.}
}

\titlerunning{Convexity of Self-Similar Transonic Shocks}        

\author{Gui-Qiang G. Chen \and \\
        Mikhail Feldman \and \\
        Wei Xiang
}

\institute{Gui-Qiang G. Chen,
               Mathematical Institute, University of Oxford,
              Oxford, OX2 6GG, UK\\
              \email{chengq@maths.ox.ac.uk}\\{}\\           
                 Mikhail Feldman,
              Department of Mathematics, University of Wisconsin-Madison,
              Madison, WI 53706-1388, USA\\
              \email{feldman@math.wisc.edu}\\{}\\
             Wei Xiang,
           Department of Mathematics, City University of Hong Kong, Kowloon, Hong Kong, China  \\
           \email{weixiang@cityu.edu.hk}
}
\date{}

\maketitle

\begin{abstract}
We are concerned with geometric properties of transonic shocks as free boundaries
in two-dimensional self-similar coordinates for
compressible fluid flows,
which are not only important for the understanding of geometric structure and stability of fluid motions in
continuum mechanics but also fundamental in the mathematical theory of multidimensional
conservation laws.
A transonic shock for the Euler equations for self-similar potential flow
separates elliptic (subsonic) and hyperbolic (supersonic) phases
of the self-similar solution of the corresponding
nonlinear partial differential equation
in a domain under consideration,
in which the location of the transonic shock
is {\it apriori} unknown.
We first develop a general framework
under which self-similar transonic shocks,
as free boundaries, are proved to be uniformly convex,
and then apply this framework to prove the uniform convexity
of transonic shocks in the two longstanding
fundamental shock problems -- the shock reflection-diffraction by wedges
and the Prandtl-Meyer reflection
for supersonic flows past solid ramps.
To achieve this, our approach is to exploit underlying nonlocal properties of the solution and
the free boundary for the potential flow equation.

\keywords{Transonic shock \and free boundary \and strict convexity \and uniform convexity \and	Euler equations \and compressible flow \and potential flow \and self-similar \and	conservation laws \and PDE \and shock reflection-diffraction \and	Prandtl-Meyer reflection \and nonlinear \and global approach \and techniques \and	geometric shapes \and fine properties}
\subclass{Primary: 35R35\and 35M12 \and
	35C06 \and 35L65 \and 35L70 \and 35J70 \and 76H05 \and 35L67 \and
	35B45 \and 35B35 \and	35B40 \and 35B36 \and	35B38;
	Secondary: 35L15 \and 35L20 \and 35J67 \and 76N10 \and 76L05 \and 76J20 \and 76N20 \and 76G25}
\end{abstract}

\section{\, Introduction}\label{sec:introduction}
We are concerned with geometric properties of transonic shocks
as free boundaries in two-dimensional self-similar coordinates for compressible fluid flows,
which are not only important for the understanding of geometric structure and stability of fluid motions in
continuum mechanics but also fundamental in the mathematical theory of multidimensional
conservation laws (see \cite{bers-book1958SubsonicTransonicGas,cf-book2014shockreflection,Da}).
Mathematically, a transonic shock for the Euler equations for potential flow
separates elliptic (subsonic) and hyperbolic (supersonic) phases
of the self-similar solution of the corresponding
nonlinear partial differential equation (PDE)
in a domain under consideration, in which the location
of the transonic shock is {\it apriori} unknown.
The Rankine-Hugoniot conditions on the shock, together
with the nonlinear PDE in the
elliptic and hyperbolic regions, provide
the sufficient overdeterminancy for finding the shock location.
This enforces a restriction to the shock
and yields its fine properties such as
its possible geometric shapes,
which is the main theme of this paper.
For this purpose, we formulate the transonic shock problem
as a one-phase free boundary problem for the nonlinear elliptic PDE
in a domain with a part of the boundary fixed,
as illustrated in Fig. \ref{figure:free boundary problems}.
More precisely, we first develop a general framework
under which self-similar transonic shock
waves, as the free boundaries in the one-phase problem,
are proved to be uniformly convex,
and then apply this framework to prove the uniform convexity
of transonic shocks in the two longstanding
fundamental shock problems -- the shock reflection-diffraction by wedges
and the Prandtl-Meyer reflection for supersonic flows past solid ramps.
In particular, the convexity of transonic shocks
is consistent with the geometric configurations of shocks observed in
physical experiments and numerical simulations; see
\textit{e.g.} \cite{BD,Chapman,cdx-1,GlimmMajda},
\cite{DP,DG,Hindman-Kutler-Anderson,Kutler-Shankar,Schneyer,Shankar-Kutler-Anderson},
\cite{Glaz-Colella1,Glaz-Colella2,GWGH,IVF,WC}, and the references cited therein.
Also see \cite{CCY2,CCY3,KTa,LaxLiu,LL2,SCG,Serre} for the geometric structure
of numerical Riemann solutions involving transonic shocks for
the Euler equations for compressible fluids.

One of our key observations in this paper is that
the convexity of transonic shocks is not a local property.
In fact, for the regular shock reflection-diffraction problem as described in \S 7.1,
the uniform convexity is a result of the interaction between the cornered wedge
and the incident shock, since the reflected shock remains flat when the wedge is a flat wall.
Therefore, any local argument is not sufficient to lead to a proof of the uniform convexity.
In this paper, we develop a global approach
by exploiting some nonlocal properties of transonic shocks
in self-similar coordinates
and employ it to prove that the transonic shocks must
be convex.
Our approach is based on two
features related to the global and nonlinear phenomena.
One is that the convexity of transonic shocks is closely related to the monotonicity properties
of the solution, which is derived from the global structure in the applications.
These properties are also crucial in the proof of the existence of the two shock problems
in \cite{Bae-Chen-Feldman-2,cf-book2014shockreflection}.
The other is that the Rankine-Hugoniot conditions, combined with the monotonicity properties,
enforce
the nonlocal dependence between the values of the velocity at the points of the transonic shock,
as well as the nonlocal dependence between the velocity and
the geometric shape of the shock.
Moreover, for this problem, it seems to be difficult to apply directly the methods
as in \cite{CS1, CS2, DM},
owing to the difference and more complicated structure of the boundary conditions.

The convexity of shock waves is not only an important geometric property
observed frequently in physical experiments and numerical simulations,
but also crucial in the analysis of multidimensional shock waves.
For example, the convexity property of transonic shocks plays
an essential role
in the proof of the uniqueness and stability of shock waves
with large curvature in \cite{CFX2017}.
Therefore, our approach can be useful for
other nonlinear problems involving transonic shocks, especially for the problems
that cannot be handled by the perturbation methods.

In particular, as an application
of our general framework for the convexity of shocks,
we prove the uniform convexity of transonic shocks in
the two longstanding fundamental shock problems.
The first is the problem of shock reflection-diffraction
by concave cornered wedges as analyzed in \S 7.1. It has been
analyzed
in Chen-Feldman \cite{cf-annofmath20101067regularreflection,cf-book2014shockreflection},
in which von Neumann's sonic and detachment conjectures for the existence of regular
shock reflection-diffraction configurations have been solved all the way up to
the detachment wedge-angle for potential flow.
The second is the Prandtl-Meyer reflection problem
for supersonic flow past a solid ramp as analyzed in \S 7.2.
Elling-Liu \cite{ellingliu-CPAMPrandtlMeyerReflection20081347} made a first rigorous
analysis of the problem for which the steady supersonic weak shock solution is a large-time asymptotic limit
of an unsteady flow
under certain assumptions for an important class of wedge angles and potential fluids.
Recently, in Bae-Chen-Feldman \cite{baechenfeldman-prandtlmayerReflection,Bae-Chen-Feldman-2},
the existence theorem for the general case all the way up to the detachment wedge-angle
has been established
via new techniques
based on those developed in Chen-Feldman \cite{cf-book2014shockreflection}.
For both problems, we apply the general framework
developed in this paper to prove the uniform convexity of the transonic shocks involved.

The study of geometric properties of free boundaries, such as the convexity of free boundaries and the monotonicity properties of the corresponding solutions under consideration,
is fundamental in the mathematical theory of free boundary problems;
see \cite{CJK,CS1,CS2,DM,EvansSpruck,Friedman_book,Toland} and the
references cited therein.
Furthermore, as mentioned earlier, the convexity of free boundaries has played an essential
role in the analysis of the uniqueness and stability of solutions of the free boundary problems,
as shown in \cite{CFX2017}.

\smallskip
The organization of this paper is as follows:
In \S \ref{sec:potential flow wquation and free boundary problem},
we introduce the potential flow equation and
the Rankine-Hugoniot conditions on the shock, and
set up a framework as a general free boundary problem on which we focus in this paper,
and then we present the main theorem for this free boundary problem.
In \S \ref{sec:proof of main thoerem}, we show some useful lemmas.
Then we develop our
approach to prove first
the strict
convexity of the shock,
\emph{i.e.}, Theorem \ref{thm:main theorem}
in \S \ref{sec:proof of the main theorem},
and to prove further the uniform convexity
of the shock
on compact subsets of its relative interior,
\emph{i.e.}, Theorem \ref{thm:main theorem-strictUnif}
in \S \ref{sec:proof of the main theorem-1}.
In \S 6, we establish the relation between the strict convexity of the transonic shock
and the monotonicity properties of the solution, \emph{i.e.}, Theorem \ref{theorem:1x}.
Finally, in \S \ref{sec:application}, we apply the main theorems to prove the uniform convexity
of transonic shocks in the two
shock problems -- the shock reflection-diffraction by wedges
and the Prandtl-Meyer reflection
for supersonic flows past solid ramps.

A note regarding terminology for simplicity:
Since our main concern is the convexity of the elliptic (subsonic) region
for which the transonic shock as a free boundary
is a part of the boundary of the region throughout this paper,
we use the term $-$ convexity $-$ for the free boundary, even though it corresponds to the
concavity of the shock location function in a natural coordinate system.
Moreover, we use the term $-$ {\it uniform convexity}
$-$ for a transonic shock to represent that the transonic shock is
of non-vanishing curvature on any compact subset of its relative interior.

\section{\, The Potential Flow Equation and Free Boundary
	Problems}\label{sec:potential flow wquation and free boundary problem}

\subsection{\, The potential flow equation}\label{subsec:potential flow rh condition}

As in \cite{bcf-invent2009174505optimalregularity,cf-annofmath20101067regularreflection},
the Euler equations for potential flow consist of the conservation
law of mass for the density and the Bernoulli law for
the velocity potential $\Psi$:
\begin{eqnarray}
&&\partial_t\rho+\gr_{\mathbf{x}}\cdot(\rho\nabla_{\xx}\Psi)=0,  \label{equ:1}\\
&&\p_t\Psi+\frac{1}{2}|\gr_{\mathbf{x}}\Psi|^2+i(\r)=B_0,  \label{equ:2}
\end{eqnarray}
where $B_0$ is the Bernoulli constant determined by the incoming
flow and/or boundary conditions,
$\mathbf{x}=(x_1,x_2)\in\mr^2$,
$
i(\r)=\int^{\rho}_1\frac{p'(\tau)}{\tau}d\tau
$
for the pressure function $p=p(\rho)$,
and $\mathbf{v}=\nabla \Psi$ is the velocity.

For polytropic gas, by scaling,
$$
p(\r)=\frac{\r^{\g}}{\g},\quad c^2(\r)=\r^{\g-1},\quad
i(\r)=\frac{\r^{\g-1}-1}{\g-1} \qquad \mbox{for $\g>1$},
$$
where $c(\r)$ is the sound speed.

If the initial-boundary value problem is invariant under
the self-similar scaling:
$$
(\mathbf{x},t)\rightarrow(\alpha \mathbf{x},\alpha t),\quad
(\rho,\Psi)\rightarrow(\rho,\frac{\Psi}{\alpha})\qquad\,\,\,\,
\mbox{for $\alpha\neq0$},
$$
then we can seek self-similar solutions with the form:
$$
\rho(\mathbf{x},t)=\rho(\xxi),\quad
\Psi(\mathbf{x},t)=t\big(\varphi(\xxi)+\frac{1}{2}|\xxi|^2\big)
\qquad\,\,\mbox{for $\xxi=(\xi_1,\xi_2)=\frac{\mathbf{x}}{t}$},
$$
where $\varphi$ is called a pseudo-velocity potential that satisfies
$D\varphi:=(\varphi_{\xi_1}, \varphi_{\xi_2})=\mathbf{v}-\xxi$, which
is called a pseudo-velocity.
The pseudo--potential function
$\varphi$ satisfies the following potential flow equation in the self-similar
coordinates:
\begin{equation}\label{1.1}
\mbox{div}(\rho D\varphi)+2\rho=0,
\end{equation}
where the density function $\rho=\rho(|D\varphi|^2,\varphi)$ is determined by
\begin{equation}\label{1.2}
\rho(|D\varphi|^2,\varphi)
=\big(\rho_0^{\gamma-1}-(\gamma-1)(\varphi+\frac{1}{2}|D\varphi|^2)\big)^{\frac{1}{\gamma-1}},
\end{equation}
with constant $\rho_0>0$, and the {\it divergence}  div and {\it gradient} $D$ are with respect to the
self-similar variables $\xxi$.

From \eqref{1.1}--\eqref{1.2}, we see that the potential function
$\varphi$ is governed by the following potential flow equation of
second order:
\begin{equation}\label{equ:potential flow equation}
\div\big(\r(|D\varphi|^2,\varphi)D\varphi\big)+2\r(|D\varphi|^2,\varphi)=0.
\end{equation}
Equation (\ref{equ:potential flow equation}) written in
the non-divergence form is
\begin{equation}\label{nondivMainEq}
(c^2-\varphi_{\cxi}^2)\varphi_{\cxi\cxi}-2\varphi_\cxi\varphi_\ceta\varphi_{\cxi\ceta}
+(c^2-\varphi_{\ceta}^2)\varphi_{\ceta\ceta}+2c^2-|D\varphi|^2=0,
\end{equation}
where the sound speed $c=c(|D\varphi|^2,\varphi,\rho_0)$ is determined by
\begin{equation}\label{c-through-density-function}
c^2(|D\varphi|^2,\varphi,\rho_0)=
\rho^{\gamma-1}(|D\varphi|^2,\varphi,\rho_0^{\gamma-1})
=\rho_0^{\gamma-1}-(\gamma-1)\big(\frac{1}{2}|D\varphi|^2+\varphi\big).
\end{equation}
Equation \eqref{equ:potential flow equation} is a second-order
equation of mixed hyperbolic-elliptic type, as it can be seen
from (\ref{nondivMainEq}):  It is elliptic if
and only if
\begin{equation}\label{criterion:ellipticity-a}
|D\varphi|<c(|D\varphi|^2,\varphi, \rho_0),
\end{equation}
which is equivalent to
\begin{equation}\label{criterion:ellipticity}
|D\varphi|<c_{\star}(\varphi,\rho_0)\deq\sqrt{\frac{2}{\g+1}\big(\rho_0^{\gamma-1}-(\gamma-1)\varphi\big)}.
\end{equation}
Moreover,  from  (\ref{nondivMainEq})--(\ref{c-through-density-function}),
equation \eqref{equ:potential flow equation} satisfies the Galilean invariance property:
If $\varphi(\xxi)$ is a solution, then its shift $\varphi(\xxi-\xxi_0)$ for any constant vector $\xxi_0$ is also a solution.
Furthermore,
$\varphi(\xxi)+const.$ is a solution of  \eqref{equ:potential flow equation}
with adjusted constant $\rho_0$ correspondingly in \eqref{1.2}.

One class of solutions of (\ref{equ:potential flow equation})
is that of {\em constant states} that are the solutions
with constant velocity $\mathbf{v}=(u, v)$.
This implies that the pseudo-potential $\varphi$ of a constant state satisfies
$D\varphi=\mathbf{v}-\xxi$ so that
\begin{equation}\label{constantStatesForm}
\varphi(\xxi)=-\frac 12|\xxi|^2+\mathbf{v}\cdot\xxi +C,
\end{equation}
where $C$ is a constant. For such $\varphi$, the expressions
in \eqref{1.2}--\eqref{c-through-density-function}
imply that the density and sonic
speed are positive constants $\rho$ and $c$, {\it i.e.},
independent of $\xxi$.
Then,
from \eqref{criterion:ellipticity-a} and \eqref{constantStatesForm},
the ellipticity condition for the constant state is
$$
|\xxi -\mathbf{v}|<c.
$$
Thus, for a constant state $\mathbf{v}$,
equation (\ref{equ:potential flow equation})
is elliptic inside the {\em sonic circle},
with center $\mathbf{v}$ and radius $c$.

\subsection{\, Weak solutions and the Rankine-Hugoniot conditions}

Since the problem involves transonic shocks, we define the notion
of weak solutions of equation \eqref{equ:potential flow equation},
which admits shocks.
As in \cite{cf-annofmath20101067regularreflection},
it is defined in the distributional sense.

\begin{defi}\label{def:weak solution}
	A function $\varphi\in W^{1,1}_{\rm loc}(\O)$ is called a weak solution of
	\eqref{equ:potential flow equation} if
	\begin{enumerate}[\rm (i)]
		\item  \label{def:weak solution-i1}
		$\rho_0^{\gamma-1}-(\gamma-1)(\varphi+\frac{1}{2}|D\varphi|^2)\geq0\quad\text{a.e. in }\O${\rm ;}
		
		\smallskip
		\item \label{def:weak solution-i2}
		$(\r(|D\varphi|^2,\varphi),\r(|D\varphi|^2,\varphi)|D\varphi|)\in(L^1_{\rm loc}(\O))^2${\rm ;}
		
		\smallskip
		\item \label{def:weak solution-i3}
		For every $\z\in C^{\infty}_{c}(\O)$,
		\begin{equation}\label{def:weak solution-i3-Eqn}
		\int_{\O}\big(\r(|D\varphi|^2,\varphi)D\varphi\cdot
		D\z-2\r(|D\varphi|^2,\varphi)\z\big)\mathrm{d}\xxi=0.
		\end{equation}
	\end{enumerate}
\end{defi}

A piecewise $C^2$ solution $\varphi$ in $\Omega$, which is $C^2$ away from and $C^1$ up
to the $C^1$--shock curve $S$,
satisfies the conditions of
Definition \ref{def:weak solution} if and only if it is a $C^2$--solution of
\eqref{equ:potential flow equation} in each subregion
and satisfies the following
Rankine-Hugoniot conditions across curve $S$:
\begin{eqnarray}
&&[\r(|D\varphi|^2,\varphi)D\varphi\cdot\bn]_{S}=0, \label{RH conditon involve normal direction}\\[1mm]
&&[\varphi]_{S}=0, \label{RH conditon continous}
\end{eqnarray}
where the square bracket $[\,\cdot\,]_{S}$ denotes the jump across $S$,
and $\bn$ is the unit normal vector to $S$.
Condition  (\ref{RH conditon continous})
follows from the requirement: $\varphi\in W^{1,1}_{\rm loc}(\O)$
for piecewise-smooth $\varphi$,
and condition (\ref{RH conditon involve normal direction})
is obtained from (\ref{def:weak solution-i3-Eqn})
via integration by parts and by using (\ref{RH conditon continous}) and
the piecewise-smoothness of $\varphi$.
Physically, condition (\ref{RH conditon involve normal direction})
is owing to the conservation of mass across
the shock, and (\ref{RH conditon continous})
is owing to the irrotationality. From now on,
we denote $D\varphi\cdot\bn=\partial_{\bn}\varphi=\varphi_{\bn}$
when no confusion arises.

It is well known that there are fairly many weak solutions to conservation
laws \eqref{equ:potential flow equation}.
In order to single out the physically relevant solutions, the entropy condition
is required.
A discontinuity of $D\varphi$ satisfying the Rankine-Hugoniot
conditions
\eqref{RH conditon involve normal direction}--\eqref{RH conditon continous}
is called a shock if it satisfies the following
physical entropy condition:
\begin{equation}\label{2.13e}
\begin{split}
&\mbox{\it The density function $\r$ increases across
   	the discontinuity}\\
&\mbox{\it in the pseudo-flow direction}.
\end{split}
\end{equation}

\smallskip
\noindent
From \eqref{RH conditon involve normal direction},
the entropy condition indicates that the normal
derivative function $\varphi_{\bn}$ on a shock
always decreases across
the shock in the pseudo-flow direction.
That is, when the pseudo-flow direction
and the unit normal vector $\nu$ are both from state $(0)$ to $(1)$,
then $\r_1>\r_0$ and $\varphi_{1\bn}<\varphi_{0\bn}$.

\subsection{\, General framework and free boundary problems}
Now we develop a general framework for the transonic shocks as free boundaries,
on which we will focus our analysis in this paper.

As in Fig. \ref{figure:free boundary problems},
let $\Omega$ be a bounded, open, and connected set,
and $\partial \Omega = \Gsh\cup\Gamma_1\cup\Gamma_2$,
where the closed curve segment
$\Gsh$ is a transonic shock that separates a pseudo-supersonic
constant state $(0)$ outside $\Omega$ from a
pseudo-subsonic (non-constant) state $(1)$ inside $\Omega$,
and $\Gamma_1\cup\Gamma_2$ is a fixed boundary whose structure
will be specified later.
The dashed ball $B_{c_0}(O_0)$ is the sonic circle of state $(0)$ with
center $O_0=(u_0,v_0)$ and radius $c_0$.
Note that $\Gsh$ is outside of $B_{c_0}(O_0)$
because state (0) is pseudo-supersonic on $\Gsh$.
$A$ and $B$ are the endpoints of the free boundary $\Gsh$,
while $\bt_A$ and $\bt_B$ are the unit tangent vectors pointing into the interior of
$\Gsh$ at $A$ and $B$, respectively.
\begin{figure}[!ht]
	\centering
	\includegraphics[width=0.40\textwidth]{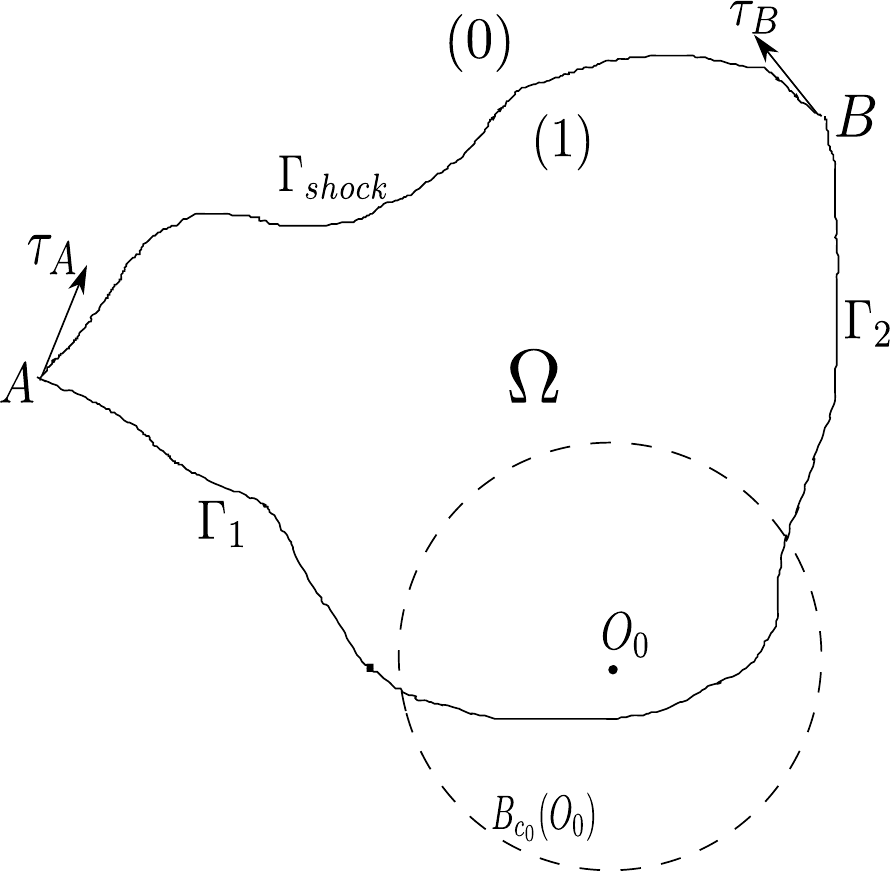}
	\hspace{0.9in}\caption{Free boundary problems}
	\label{figure:free boundary problems}
\end{figure}

Denote $\mathbf{v}_0=(u_0,v_0)$. Then the pseudo-potential of constant state $(0)$ with density $\rho_0>0$
has the form:
\begin{equation}\label{state0-Nonzero}
\varphi_0
=-\frac{1}{2}(\xxi-\mathbf{v}_0)^2.
\end{equation}
Let
$$
\phi:=\varphi-\varphi_0.
$$
Then we see from (\ref{nondivMainEq}) that $\phi=\varphi-\varphi_0$ satisfies the following equation in $\Omega$:
\begin{equation}\label{equ:study}
(c^2-\varphi_{\xi_1}^2)\phi_{\xi_1\xi_1}-2\varphi_{\xi_1}\varphi_{\xi_2}\phi_{\xi_1\xi_2}+(c^2-\varphi_{\xi_2}^2)\phi_{\xi_2\xi_2}=0,
\end{equation}
where $c=c(|D\varphi|^2, \varphi, \rho_0)$ is the sound speed, determined by (\ref{c-through-density-function}).
Along the shock curve $\Gsh$ that separates the constant state (0)
with pseudo-potential $\varphi_0$
from the non-constant state $\varphi$ in $\Omega$, the boundary conditions for $\phi$ are:
\begin{equation}\label{equ:boundary}
\phi=0,\,\,\, \r(|D\phi+D\varphi_0|^2,\phi+\varphi_0)D(\phi+\varphi_0)\cdot\bn=\rho_0D\varphi_0\cdot\bn\quad\, \mbox{on $\Gsh$},
\end{equation}
from the Rankine-Hugoniot conditions
\eqref{RH conditon involve normal direction}--\eqref{RH conditon continous}.

Now we state the main results of this paper.
We first state the following structural framework for domain $\Omega$ under consideration.

From now on, $\Gamma^0$ denotes the relative interior of a curve segment $\Gamma$.
In particular, $\Gsh^0$ is the relative interior of $\Gsh$.

\bigskip
\noindent
{\bf Framework (A)} -- The structural framework for domain $\Omega$:

\begin{enumerate}
	\item[(i)] Domain $\Omega$ is bounded.
	Its boundary $\partial\Omega$ is a continuous closed curve without self-intersections,
	piecewise $C^{1,\alpha}$  up to the endpoints of
	each smooth part for some $\alpha\in(0,1)$, and the number of smooth parts
	is finite.
	
	\item[(ii)]
	At each corner point of $\partial\Omega$,
	angle $\theta$ between the arcs meeting at that point from
	the interior of $\Omega$ satisfies $\theta\in (0, \pi)$.
	
	\item[(iii)]
	$\partial\Omega = \Gsh\cup\Gamma_1\cup\Gamma_2$,
	where $\Gsh$, $\Gamma_1$, and $\Gamma_2$ are connected and
	disjoint, and both $\Gsh^0$ and $\Gamma_1\cup\Gamma_2$ are non-empty.
	Moreover, if $\Gamma_i\ne \emptyset$ for some $i\in\{1,2\}$, then its
	relative interior is nonempty, \emph{i.e.}, $\Gamma_i^0\ne \emptyset$.

	\item[(iv)] $\Gsh$ includes its endpoints $A$ and $B$
      with corresponding unit tangent vectors $\bt_A$ and $\bt_B$  pointing into the interior of
      $\Gsh$ respectively.
	If $\Gamma_1\ne \emptyset$, then
	$A$ is a common endpoint of $\Gsh$ and $\Gamma_1$.
	If $\Gamma_2\ne \emptyset$, then
	$B$ is a common endpoint of $\Gsh$ and $\Gamma_2$.
\end{enumerate}

If $\bt_A\ne \pm\bt_B$, define the cone:
$$
Con:=\{r\bt_A+s\bt_B\; : \; r\,,s\in(0,\infty)\}.
$$
Then we have
\begin{thm}\label{thm:main theorem}
	Assume that domain $\Omega$ satisfies Framework {\rm (A)}.
	Assume that $\phi\in C^1(\overline{\Omega})\cap C^{2}(\Omega\cup\Gsh^0)\cap C^{3}(\Omega)$ is
	a solution of \eqref{equ:study}--\eqref{equ:boundary}, which is not a constant state in $\Omega$.
	Moreover, let $\phi$ satisfy the following conditions{\rm :}
	\smallskip
	\begin{enumerate}
		\item[\rm (A1)]\label{Assumpt-H1}
		The entropy condition holds across $\Gsh${\rm :}
		$\rho(|D\varphi|^2, \varphi)>\rho_0$ and
		$\phi_{\bn}<0$
		along $\Gsh$, where $\bn$ is the interior normal vector to $\Gsh$,
		i.e., pointing into $\Omega${\rm ;}
		
		\smallskip
		\item[\rm (A2)]\label{Assumpt-H2}
		There exist constants $C_1>0$ and $\alpha_1\in(0,1)$ such that
		$\|\phi\|_{1+\alpha_1,\overline{\Omega}}\leq C_1${\rm ;}
		
		\smallskip
		\item[\rm (A3)]\label{Assumpt-H3}
		In $\Omega\cup\Gsh^0$, equation \eqref{equ:study} is
		strictly
		elliptic{\rm :} $c^2-|D(\phi+\varphi_0)|^2>0${\rm ;}
		
		\smallskip
		\item[\rm (A4)]\label{Assumpt-H4-a}
		$\Gsh$ is $C^2$ in its relative interior{\rm ;}
		
		\smallskip
		\item[\rm (A5)]\label{Assumpt-H4}
		$\bt_A\ne\pm\bt_B$, and $\{P+Con\}\cap\Omega=\emptyset$
		for any point $P\in\overline{\Gsh}${\rm ;}
		
		\smallskip
		\item[\rm (A6)]\label{Assumpt-H5}
		There exists a vector $\mathbf{e}\in Con$ such that
		one of the following conditions holds{\rm :}
		\begin{enumerate}[\rm (i)]
			\item\label{H5Cases-1}
			$\Gamma_1\ne\emptyset$, and
			the directional derivative
			$\phi_{\mathbf{e}}$
			cannot have a local maximum point on $\Gamma_1^0\cup\{A\}$ and a local minimum point on $\Gamma_2^0$,
			\item\label{H5Cases-2}
			$\Gamma_2\ne\emptyset$, and $\phi_{\mathbf{e}}$ cannot have a local minimum point on
			$\Gamma_1^0$ and a local maximum point on $\Gamma_2^0\cup\{B\}$,
			\item\label{H5Cases-3}
			$\phi_{\mathbf{e}}$ cannot have a local minimum point on $\Gamma_1\cup\Gamma_2$,
		\end{enumerate}
		where all the local maximum or minimum points are relative to $\overline\Omega$.
	\end{enumerate}
	Then the free boundary $\Gsh$ is a convex graph{\rm .} That is, there exists a concave
	function $f\in C^{1,\alpha}(\mR)$ in some orthonormal coordinate system $(S,T)$ in $\mR^2$ such that
	\begin{equation}\label{shock-graph-inMainThm}
	\begin{split}
    &\Gsh =\{(S,T):\; S=f(T), \; T_A<T<T_B\},\\
    &\Omega\cap\{T_A<T<T_B\} \subset\{S<f(T)\}
    \end{split}	
\end{equation}
	with $f\in C^\infty((T_A, T_B))$,
	and shock $\Gsh$ is strictly convex in its relative interior
	in the sense that,
	if $P=(S,T)\in\Gsh^0$ and $f''(T)=0$, then
	there exists an integer $k>1$, independent of the choice of the coordinate system $(S,T)$, such that
	\begin{equation}\label{strictConvexity-degenerate-graph}
	f^{(n)}(T)=0 \;\;\mbox{ for $n=2, \dots, 2k-1$},
	\qquad\,\,\, f^{(2k)}(T)<0.
	\end{equation}
	The number of the points at which $f''(T)=0$ is at most finite on each compact subset of $\Gsh^0$.
	In particular, the free boundary $\Gsh$ cannot contain any straight segment.
\end{thm}

\begin{rem}
	Conditions {\rm (A2)} and {\rm (A5)}--{\rm (A6)} of Theorem {\rm \ref{thm:main theorem}}
	are the requirements on the global behavior of solutions.
	In fact, {\rm (A5)} ensures that there is a coordinate system in which the
	shock is
	a Lipschitz graph globally.
\end{rem}

\begin{rem}
	Condition {\rm (A6)} allows us to deal with three different kinds of boundary conditions.
	Moreover, at each of the endpoints of $\Gsh$, the ellipticity can be either
	uniform or degenerate.
	Some applications to each case can be found in \S {\rm \ref{sec:application}}.
\end{rem}

\begin{rem} The assumption that $\phi$ is not a constant state means that $\phi$ cannot
	be of the form{\rm :} $\phi=a_1+(a_2,a_3)\cdot\xxi$ in $\Omega$, where $a_j, j=1,2,3$, are constants.
	In fact, this assumption can be guaranteed by the boundary conditions
	assigned along $\Gamma_1\cup\Gamma_2$ in the applications in \S {\rm \ref{sec:application}}.
\end{rem}

In the next theorem, we show that, if assumptions (A1)--(A4) and (A6) hold,
then a monotonicity
condition for $\phi$ near $\Gsh^0$, which is slightly stronger than condition ${\rm (A5)}$,
is the necessary and sufficient condition for the strict convexity of
shock $\Gsh$.

\begin{thm}\label{theorem:1x}
	Let $\Omega$ and $\phi$ be as in Theorem {\rm \ref{thm:main theorem}} except
	condition ${\rm (A5)}$.
	Then the fact that the free boundary $\Gsh$ is a strictly convex graph
	in the sense of \eqref{shock-graph-inMainThm}--\eqref{strictConvexity-degenerate-graph}
	in Theorem {\rm \ref{thm:main theorem}} is the necessary and sufficient
	condition for the monotonicity property that
	$\phi_{\mathbf{e}}>0$ on $\Gsh^{0}$ for any unit vector
	$\mathbf{e}\in \overline{Con}$,
	where
	$\Gsh^0$ is the relative interior of
	$\Gsh$.
\end{thm}

\begin{rem}\label{theorem:1xRem}
	Let $\Omega$ and $\phi$ be as in Theorem {\rm \ref{theorem:1x}}, including that
	the monotonicity property {\rm (}or equivalently, the strict convexity of $\Gsh${\rm )}
	holds.
	In addition, assume that,
	for any unit vector $\mathbf{e}\in\overline{Con}$
	and any point $\xxi$  in the fixed boundary part $\Gamma_1\cup\Gamma_2$,
	$\phi_{\mathbf{e}}$ satisfies that either $\phi_{\mathbf{e}}(\xxi)\geq0$ or
	$\phi_{\mathbf{e}}$
	cannot attain its local minimum at $\xxi$ with respect to $\Omega$. Then  $\phi_{\mathbf{e}}>0$ in $\Omega\cup\Gsh^{0}$
	for any unit vector $\mathbf{e}\in \overline{Con}$.
\end{rem}

The proof of Remark \ref{theorem:1xRem} is given after the proof of Theorem {\rm \ref{theorem:1x}}
in \S \ref{EquivalenceSect}.
Moreover, the assumptions of
Remark \ref{theorem:1xRem} can be justified for
the two applications{\rm :} the regular shock reflection problem and the Prandtl-Meyer
reflection problem{\rm ;} see \S \ref{sec:application}.

Furthermore, under some additional assumptions
that are satisfied in the two applications,
the shock curve is uniformly convex
in its relative interior in the sense defined in the following theorem:

\begin{thm}\label{thm:main theorem-strictUnif}
	Let $\Omega$ and $\phi$ be as in Theorem {\rm \ref{thm:main theorem}}.
	Furthermore, assume that, for any unit vector $\mathbf{e}\in\mathbb{R}^2$,
	the boundary part $\Gamma_1\cup\Gamma_2$ can be further decomposed so that
	\begin{enumerate}[\rm (A1)]
		\item[\rm (A7)]\label{Assumpt-H6}
		$\Gamma_1\cup\Gamma_2=
		\hat \Gamma_0\cup\hat{\Gamma}_1\cup\hat{\Gamma}_2\cup\hat \Gamma_3$,
		where some of $\hat\Gamma_i$ may be empty,
		$\hat\Gamma_i$ is connected for each $i=0,1,2,3$,
		and all curves $\hat \Gamma_i$ are located along $\partial\Omega$ in the
		order of their indices,
		i.e., non-empty sets $\hat\Gamma_j$ and $\hat\Gamma_k$, $k>j$,
		have a common endpoint if and only if either $k=j+1$ or $\Gamma_i=\emptyset$
		for all $i=j+1, \dots, k-1$.
		Also, the non-empty set $\hat\Gamma_i$ with the smallest {\rm (}resp. largest{\rm )} index
		has the
		common endpoint $A$ {\rm (}resp. $B${\rm )} with $\Gsh$.
		Moreover, if $\hat\Gamma_i\ne \emptyset$ for some $i\in\{0,1,2,3\}$, then its
		relative interior is nonempty{\rm :} $\hat\Gamma_i^0\ne \emptyset${\rm ;}
		
		\smallskip
		\item[\rm (A8)]\label{Assumpt-H7}
		$\phi_{\mathbf{e}}$ is constant along $\hat\Gamma_0$ and $\hat\Gamma_3${\rm ;}
		
		\smallskip
		\item[\rm (A9)]\label{Assumpt-H8}
		For $i=1,2$, if $\phi_\ee$ attains its local
		minimum or maximum relative to $\overline\Omega$ on $\hat\Gamma_i^0$,
		then $\phi_{\mathbf{e}}$ is constant along $\hat\Gamma_i${\rm ;}
		
		\smallskip
		\item[\rm (A10)]\label{Assumpt-H9} One of the following two conditions holds{\rm :}
		\begin{enumerate}
 			\item[\rm (i)] Either  $\hat{\Gamma}_1=\emptyset$ or $\hat{\Gamma}_2=\emptyset${\rm ;}
			\item[\rm (ii)] Both $\hat{\Gamma}_1$ and $\hat{\Gamma}_2$
			are non-empty, and $\hat{\Gamma}_3=\emptyset$,
			so that $\hat\Gamma_2$ has the common
			endpoint $B$ with $\Gsh$.
			At point $B$, the following conditions hold{\rm :}
			
			\smallskip
			\begin{itemize}
				\item
				If $\bn_{\rm sh}(B)\cdot {\mathbf{e}}<0$,
				then $\phi_{\mathbf{e}}$ cannot attain  its
				local maximum
				relative to $\overline\Omega$ at $B${\rm ,}
				\item
				If $\bn_{\rm sh}(B)\cdot {\mathbf{e}}=0$, then $\phi_\ee(B)=\phi_\ee(Q^*)$
				for  the common endpoint $Q^*$ of $\hat\Gamma_1$ and $\hat\Gamma_2$,
			\end{itemize}
			
			\smallskip
			\noindent
			where $\bn_{\rm sh}(B):=\lim\limits_{\Gsh^0\ni P\to B}\bn(P)$, which exists since $\Gsh$ is $C^1$
			up to $B$.
		\end{enumerate}
	\end{enumerate}
	Then
	the shock function $f(T)$ in \eqref{shock-graph-inMainThm}
	satisfies that $f''(T)<0$ for all $T\in (T_A, T_B)${\rm ;} that is, $\Gsh$ is uniformly
	convex on closed subsets of its relative interior.
\end{thm}

\begin{rem}\label{FB-levelSet-rmk}
	By \eqref{equ:boundary} and condition {\rm (A1)} of Theorem {\rm \ref{thm:main theorem}},
	it follows that $\phi<0$ in $\Omega$ near $\Gsh$.
	Since $\Gsh$ is the zero level set of $\phi$, then the following statements hold $($see also Lemma {\rm \ref{lem:shock graph}(v)}$)${\rm :}
	\begin{enumerate}[\rm (i)]
		\item\label{FB-levelSet-rmk-i1}
		The convexity of $\Gsh$ is equivalent to
		the fact that
		$\phi_{\bt\bt}\ge 0$ on $\Gsh$. Moreover,
		by \eqref{strictConvexity-degenerate-graph},
		if $\phi_{\bt\bt}=0$ at some $P\in \Gsh$,
		then there exists an integer $k>1$ such that
		\begin{equation}\label{strictConvexity-degenerate}
		\partial^n_{\bt}\phi=0 \quad\mbox{for $n=2, \dots, 2k-1$},
		\qquad\,\,\,\,\,
		\partial_{\bt}^{(2k)}\phi>0 \quad\mbox{at $P$},
		\end{equation}
		where $k$ is the same as in \eqref{strictConvexity-degenerate-graph}.
		In particular, this implies that $k$ is independent of the choice of the coordinate system $(S,T)$
		used in \eqref{shock-graph-inMainThm}{\rm ;}
		
		\item\label{FB-levelSet-rmk-i2}
		The conclusion of Theorem {\rm \ref{thm:main theorem-strictUnif}} is equivalent
		to the following{\rm :}
		$\phi_{\bt\bt}>0$ along $\Gsh^0$, where $\Gsh^0$ is the interior points of $\Gsh$.
	\end{enumerate}
\end{rem}

\begin{rem}
	If the conclusion of Theorem  {\rm \ref{thm:main theorem-strictUnif}}
	holds, then the  curvature of $\Gsh${\rm :}
	$$
	\kappa=-\frac{f''(T)}{\big(1+(f'(T))^2\big)^{3/2}}
	$$
	has a positive lower bound
	on any closed subset of $(T_A, T_B)$.
\end{rem}

\begin{rem}
	The definition of $\hat{\Gamma}_0$ and $\hat{\Gamma}_3$ is motivated by the observation
	that $\phi_{\mathbf{e}}$ is constant along
	the sonic arcs in the two shock problems{\rm ;} see the applications in \S {\rm \ref{sec:application}}
	for more details.
\end{rem}

\begin{rem}
	We can simplify \eqref{state0-Nonzero} as follows{\rm :}
	By the Galilean invariance of the potential flow equation \eqref{equ:study}
	{\rm (}{\it i.e.}, invariance with respect to the shift of coordinates{\rm )},
	we assume without loss of generality that $\mathbf{v}_0=(0,0)${\rm ;}
	indeed, this can be achieved by introducing the new coordinates $\xxi'=(\xi_1-u_0,\xi_2-v_0)$.
	Furthermore, we choose constant $\rho_0$ in \eqref{1.2}
	to be the density of state $(0)$.
	Then the pseudo-potential of state $(0)$ is
	\begin{equation}\label{state0-zero}
	\varphi_0=-\frac{1}{2}|\xxi|^2.
	\end{equation}
	We will use this form in the proof of the main theorems.
\end{rem}

\begin{rem}
	Rewrite the condition{\rm :} $\phi_{\bn}<0$ in {\rm (A1)}, as
	$D\varphi\cdot\bn<D\varphi_0\cdot\bn$. Then, replacing
	$\phi+\varphi_0$ by  $\varphi$ in the second equality
	in \eqref{equ:boundary} and using that $\rho>\rho_0$
	by  {\rm (A1)} for $\rho_0>0$, we have
	\begin{equation}\label{ps-potentials-on-shock}
	D\varphi_0\cdot\bn>D\varphi\cdot\bn>0\qquad\,
	\,\mbox{on $\Gsh$}.
	\end{equation}
\end{rem}

The theorems stated above are proved in \S 3--\S 6.
In \S \ref{generalPropertiesFB-sect}, we first prove some general properties
of the free boundary $\Gsh$,
and then derive some additional properties from the assumptions in the theorems.
In \S \ref{sec:proof of the main theorem}--\S 6,
we employ all of these
properties to prove Theorems
\ref{thm:main theorem}--\ref{thm:main theorem-strictUnif}.
Specifically, we prove Theorem
\ref{thm:main theorem} in
\S \ref{sec:proof of the main theorem},
Theorem \ref{thm:main theorem-strictUnif} in \S \ref{sec:proof of the main theorem-1},
and Theorem \ref{theorem:1x} in \S 6.
Then, in \S \ref{sec:application}, we apply the general framework
to show the convexity results for the two shock problems: the shock reflection-diffraction problem
and the Prandtl-Meyer reflection problem.
In the appendix,
we construct paths in $\Omega$ satisfying certain properties -- these paths are used
in the proof of the main results.

In the rest of the paper, we use the following terminology: A statement that a function
attains a local extremum at $P\in\partial\Omega$ means that the local extremum is
relative to $\overline\Omega$.
In the case when the local extremum is along (or relative to) $\partial\Omega$,
we always state that explicitly.

\section{\, Basic Properties of Solutions}\label{sec:proof of main thoerem}
\label{generalPropertiesFB-sect}
In this section, we list several lemmas for the solutions
of the self-similar potential flow equation \eqref{equ:study},
which will be used in the subsequent development.
Some of them have been proved in Chen-Feldman \cite{cf-book2014shockreflection}
for a specific geometric situation for the
shock reflection-diffraction problem.
Here we list these facts under the general
conditions of Theorem \ref{thm:main theorem}
and present them in the form
convenient for
the use in the general situation considered here.
For many of them, the proofs are similar to the
arguments in \cite{cf-book2014shockreflection},
in which cases
we omit or sketch them only below for the sake of brevity.

\subsection{\, Additional properties from {\rm (A1)}--{\rm (A5)}}
Let
$\phi\in C(\overline{\Omega})\cap C^{2}(\Omega\cup\Gsh^0)
\cap C^{3}(\Omega)$ be
a solution of \eqref{equ:study}--\eqref{equ:boundary}.
In this subsection, we use the
results
of \cite[Lemma 6.1.4]{cf-book2014shockreflection}
to show some properties
as the consequences of conditions {\rm (A1)}--{\rm (A5)}
of Theorem \ref{thm:main theorem}.
First, for a given unit constant vector $\ee\in \mathbb{R}^2$,
we derive the equation and the boundary conditions for $\phi_\ee$.

Let $\ee^{\perp}$ be the unit vector orthogonal to $\ee$,
and let $(S,T)$ be the coordinates with basis $\{\ee, \ee^{\perp}\}$.
Then equation \eqref{equ:study} in the $(S,T)$--coordinates is
\begin{equation}\label{equ:in ST coordinate}
(c^2-\varphi_S^2)\phi_{SS}-2\varphi_S\varphi_T\phi_{ST}+(c^2-\varphi_T^2)\phi_{TT}=0.
\end{equation}
Differentiating \eqref{equ:in ST coordinate} with respect to $S$ and using the Bernoulli law:
$$
\partial_Sc^2=-(\gamma-1)(\varphi_S\phi_{SS}+\varphi_T\phi_{ST}),
$$
we obtain the following equation for $w=\partial_S\phi=\partial_\ee\phi$:
\begin{equation}\label{equ:phi e}
\begin{split}
&(c^2-\varphi_S^2)w_{SS}
-2\varphi_S\varphi_Tw_{ST}+\big(c^2-\varphi_T^2\big)w_{TT}\\
&+
\big(\partial_S(c^2-\varphi_S^2)-(\gamma-1)\varphi_S\phi_{TT}\big)w_S
\\
&-\big(2\partial_S(\varphi_S\varphi_T)-2\varphi_T\phi_{TT}+
(\gamma-1)\varphi_T\phi_{TT}\big)w_T=0.
\end{split}
\end{equation}
Since the coefficients of the second-order
terms of \eqref{equ:phi e}
are the same as the ones of \eqref{equ:in ST coordinate},
we find that \eqref{equ:phi e} is
strictly
elliptic in
$\Omega\cup\Gsh^{0}$.
Using the regularity of $\phi$ above, we find that
the coefficients of (\ref{equ:phi e}) are continuous on
$\Omega\cup\Gsh^{0}$. Thus,  \eqref{equ:phi e} is uniformly
elliptic on compact subsets of $\Omega\cup\Gsh^{0}$.

For the boundary conditions along $\Gsh$, we first have
$$
\phi=0\qquad\, \mbox{along $\Gsh$}.
$$
Thus, the unit normal vector $\boldsymbol{\nu}$ and the tangent vector $\boldsymbol{\tau}$ of $\Gsh$ are
\begin{equation}\label{equ:normal tangential}
\boldsymbol{\nu}=(\nu_1,\nu_2)=\frac{D\phi}{|D\phi|},
\qquad \boldsymbol{\tau}=(\tau_1,\tau_2)=
\frac{(-\partial_{\xi_2}\phi,\partial_{\xi_1}\phi)}{|D\phi|}.
\end{equation}
Notice that, from the entropy condition -- condition (A1) of Theorem \ref{thm:main theorem},
we have
$$
D\phi\neq0, \quad\rho>\rho_0 \qquad\,\, \mbox{on $\Gsh$},
$$
so that \eqref{equ:normal tangential} is well defined.
Taking the tangential  derivative of the second equality in
\eqref{equ:boundary} along $\Gsh$ and using \eqref{equ:normal tangential},
we have
$$
(-\partial_{\xi_2}\phi\,\partial_{\cxi}+\partial_{\xi_1}\phi\,\partial_{\ceta})
\big((\rho D\varphi-\rho_0 D\varphi_0)\cdot\boldsymbol D\phi\big)=0
\qquad\,\, \mbox{on $\Gsh$}.
$$
From this, after a careful calculation by using
equation \eqref{equ:study} (see \cite[Sect. 5.1.3]{cf-book2014shockreflection} for details),
we have
\begin{equation}\label{equ:boundary for phi e}
D^2\phi[\bt,\boldsymbol{h}]=0\qquad\,\, \mbox{on $\Gsh$},
\end{equation}
where $D^2\phi[a,b]
:=\sum_{i,j=1}^2a_ib_j\partial_{ij}\phi$
and
\begin{equation}\label{Expression-h}
\boldsymbol{h}=-\frac{\rho-\rho_0}{\rho_0 c^2}\big(\rho(c^2-\varphi_{\bn}^2)\varphi_{\bn}\bn -(\rho\varphi_{\bn}^2+\rho_0c^2)\varphi_{\bt}\bt\big).
\end{equation}
Using (\ref{ps-potentials-on-shock}) and conditions (A1) and (A3) of Theorem \ref{thm:main theorem},
we obtain from
(\ref{Expression-h}) that
\begin{equation}\label{equ:27 oblique of h}
\bh\cdot\bn=-\frac{\rho-\rho_0}{\rho_0c^2}\rho(c^2-\varphi_{\bn}^2)\varphi_{\bn}<0\qquad\,\,\mbox{along $\Gsh^0$}.
\end{equation}

Based on equation \eqref{equ:phi e} and the boundary condition \eqref{equ:boundary for phi e}, we have the following lemma.

\begin{lem}\label{lem:sign of second derivative}
	Let $\Omega$ be a domain with piecewise $C^1$ boundary,
	and let $\Gsh\subset\partial\Omega$ be $C^2$ in its relative interior. Let
	$\phi\in C^{2}(\Omega\cup\Gsh^{0})\cap C^{3}(\Omega)$ be a solution of
	\eqref{equ:study} in $\Omega$ and satisfy
	\eqref{equ:boundary} on $\Gsh$, and let $\phi$ be not a constant state in $\Omega$.
	Assume also that $\phi$ satisfies conditions {\rm (A1)}--{\rm (A3)} of Theorem {\rm \ref{thm:main theorem}}.
	For a fixed unit vector $\ee\in\mr^2$ with $\bn\cdot \ee<0$,
	if a local minimum or maximum of $w{\rm :}=\partial_\ee\phi$ in $\Omega$ is attained at $P\in\Gsh^{0}$,
	then $\phi_{\bt\bt}>0$ or $\phi_{\bt\bt}<0$, respectively,
	where $\bn$ denotes the interior unit normal vector to $\Gsh$ pointing into $\Omega$.
\end{lem}

\begin{proof}
	First, we note that the proof of \cite[Lemma 8.2.4]{cf-book2014shockreflection} applies to the present case
	so that the conclusion of that lemma holds:
	$$
	\bh(P)=k\ee\qquad  \mbox{at $P$ for some $k\in\mr$}.
	$$
	Since $\bn\cdot \ee<0$, we follow the proof of \cite[Lemma 8.2.15]{cf-book2014shockreflection} to obtain
	that $k>0$ and
	$$
	w_{\bn}=\frac{c^2}{k\rho\varphi_{\bn}(c^2-\varphi_{\bn}^2)} \left(\rho^2\varphi_{\bn}^2(c^2-|D\varphi|^2)+\rho_1^2c^2\varphi_{\bt}^2\right)\phi_{\bt\bt}
	\qquad\,\, \mbox{at $P$}.
	$$
	Thus, by ellipticity and (\ref{ps-potentials-on-shock}), $\phi_{\bt\bt}$ has
	the same sign as $w_{\bn}$. Also, $w$ satisfies equation (\ref{equ:phi e}), which is
	strictly  elliptic in $\Omega\cup\Gsh^0$.
	Then,
	from Hopf's lemma, $w_{\bn}(P)<0$ if $w$ attains its local maximum at $P$,
	while $w_{\bn}(P)>0$ if $w$ attains its local minimum at $P$.
	Then $\phi_{\bt\bt}(P)<0$ if $w$ attains its local maximum at $P$,
	while $\phi_{\bt\bt}(P)>0$ if $w$ attains its local minimum at $P$.
\end{proof}

Next we consider the geometric shape of $\Gsh$ under the conditions
listed in Theorem \ref{thm:main theorem}.

\begin{lem}\label{lem:shock graph}
	Let $\Omega$ be a domain with piecewise $C^1$ boundary,
	and let $\Gsh\subset\partial\Omega$ be $C^2$ in its relative interior.
	Let $\phi\in C(\overline{\Omega})\cap C^{2}(\Omega\cup\Gsh^0)\cap C^{3}(\Omega)$ be
	a solution of \eqref{equ:study}--\eqref{equ:boundary}.
	Assume also that conditions {\rm (A1)}--{\rm (A5)} of Theorem {\rm \ref{thm:main theorem}}
	are satisfied.
	For a unit vector $\ee\in Con$, which is defined in Theorem {\rm \ref{thm:main theorem}(A5)},
	let $\ee^{\perp}$ be the orthogonal unit vector
	to $\ee$ with $\ee^{\perp}\cdot\bt_A>0$.
	Let $(S,T)$ be the coordinates with respect to basis $\{\ee, \ee^{\perp}\}$,
	and let $(S_P,T_P)$ be the coordinates of point $P$ in the $(S,T)$--coordinates.
	Note that
	$T_B>T_A$ since $\ee^{\perp}\cdot\bt_A>0$. Then there
	exists $f_\ee\in C^{1,\alpha}(\mr)$ such that
	\begin{enumerate}[\rm (i)]
		\item \label{lem:shock graph-i1}
		$\Gsh=\{S=f_\ee(T)\,:\,T_A<T<T_B\}$, $\Omega\subset\{S<f_\ee(T) :\,T\in\mr\}$,
		$A=(f_\ee(T_A),T_A)$, $B=(f_\ee(T_B),T_B)$, and $f\in C^2((T_A,T_B))${\rm ;}
		
		\smallskip
		\item \label{lem:shock graph-i2}
		The directions of the tangent lines to $\Gsh$ lie between $\bt_A$ and $\bt_B${\rm ;}
		that is, in the $(S,T)$--coordinates,
		$$
		-\infty<\frac{\bt_B\cdot \ee}{\bt_B\cdot \ee^{\perp}}=f_\ee'(T_B)\leq
		f'_\ee(T)\leq f'_\ee(T_A)
		=\frac{\bt_A\cdot \ee}{\bt_A\cdot \ee^{\perp}}<\infty
        $$
		for any $T\in(T_A,T_B)$;
		
        \smallskip
        \item \label{lem:shock graph-i2-00}
		$\bn(P)\cdot \ee<0$ for any $P\in\Gsh${\rm ;}
		
		\smallskip
		\item \label{lem:shock graph-i2-01}
		$\phi_\ee>0$ on $\Gsh${\rm ;}
		
		\smallskip
		\item \label{lem:shock graph-i3}
		For any $T\in(T_A,T_B)$,
		$$
		\phi_{\bt\bt}(f_\ee(T),T)<0\quad \Longleftrightarrow
		\quad f''_\ee(T)>0,
		$$
		while
		$$
		\phi_{\bt\bt}(f_\ee(T),T)>0\quad  \Longleftrightarrow
		\quad f''_\ee(T)<0.
		$$
	\end{enumerate}
\end{lem}

\begin{proof}
	By the first condition in (\ref{equ:boundary}) and
	the entropy condition {\rm(A1)},
	\begin{equation}\label{phiOnShock}
	\phi=0, \quad \phi_{\bn}<0 \qquad\,\,\mbox{ on $\Gsh$}.
	\end{equation}
	From this,  we have the following two facts:
	\begin{enumerate}
		\item[\rm (a)] $D\phi\neq(0,0)$ on $\Gsh$;
		\item[\rm (b)] Combining (\ref{phiOnShock})
		with  assumption {\rm(A5)}, $D\phi\cdot\ee\geq 0$ on $\Gsh$ for each $\ee\in Con$.
	\end{enumerate}
	Using facts (a)--(b) and recalling that $Con$ denotes the open cone,
	we conclude that $D\phi\cdot\ee>0$ on $\Gsh$ for any $\ee\in Con$.
	Then the implicit function theorem ensures the existence of $f_\ee$
	such that property {\rm (\ref{lem:shock graph-i1})} holds.
	
	For property {\rm (\ref{lem:shock graph-i2})},
	from the definition that $\ee^{\perp}\cdot\bt_A>0$
	and the fact that $\{P+Con\}\cap\Omega=\emptyset$,
	we find that, in the $(S,T)$--coordinates,
	for any given $T\in(T_A,T_B)$ and small $\sigma>0$,
	$$
	f_\ee(T)+\frac{\bt_A\cdot \ee}{\bt_A\cdot \ee^{\perp}}\sigma
	\geq f_\ee(T+\sigma)\geq f_\ee(T)+\frac{\bt_B\cdot \ee}{\bt_B\cdot \ee^{\perp}}\sigma.
	$$
	From this, noting that
	$f'_\ee(T_A)
	=\frac{\bt_A\cdot \ee}{\bt_A\cdot \ee^{\perp}}$ and the similar expression for
	$f'_\ee(T_B)$ follow from the definition of $f'_\ee$, we obtain  {\rm (ii)}.
	
	Next we show  (\ref{lem:shock graph-i2-00}). From
	(\ref{lem:shock graph-i1}),
	$\bn=\frac{(f'_\ee(T), -1)}{\sqrt{1+(f'_\ee(T))^2}}$,
	$\bt_A=\frac{(1, f'_\ee(T_A))}{\sqrt{1+(f'_\ee(T_A))^2}}$,
	and $\bt_B=-\frac{(1, f'_\ee(T_B))}{\sqrt{1+(f'_\ee(T_B))^2}}$.
	Since $\ee\in Con$, then
	$\ee=s_1 (1, f'_\ee(T_A))-s_2 (1, f'_\ee(T_B))$ for some $s_1, s_2>0$.
	Also, the condition that $\bt_A\ne -\bt_B$ in (A5) implies that
	$f'_\ee(T_A)\ne f'_\ee(T_B)$.
	Then
	$$
	\bn\cdot\ee=\frac{1}{\sqrt{1+(f'_\ee(T))^2}}
	\big(
	s_1(f'_\ee(T)-f'_\ee(T_A)) +s_2(f'_\ee(T_B)-f'_\ee(T))
	\big)<0,
	$$
	where we have used {\rm (\ref{lem:shock graph-i2})}
	and the fact that $f'_\ee(T_A)\ne f'_\ee(T_B)$ to obtain the last inequality.
	Now (\ref{lem:shock graph-i2-00}) is proved.
	
	To show property (\ref{lem:shock graph-i2-01}),
	we notice that, along $\Gsh$, $\phi_{\bt}=0$, $\phi_{\bn}<0$ by assumption (A1) of Theorem \ref{thm:main theorem},
	and $\bn\cdot \ee<0$ by (\ref{lem:shock graph-i2-00}).
	Therefore, $\phi_\ee=(\bn\cdot\ee)\phi_{\bn}>0$,
	which is (\ref{lem:shock graph-i2-01}).
	
	Finally,  property {\rm (\ref{lem:shock graph-i3})}
	follows
	from the boundary conditions along $\Gsh$.
	More precisely, in the $(S,T)$--coordinates, differentiating twice with respect to $T$
	in the equation: $\phi(f_\ee(T),T)=0$, and using
	that $\phi_{\bt}=0$ and $\phi_\ee\ne 0$  along $\Gsh$
	by property (\ref{lem:shock graph-i2-01}), we have
	\begin{equation}\label{tangDerivOnShock}
	f''_\ee(T)=-\frac{D^2\phi[D^{\perp}\phi,D^{\perp}\phi]}{(\phi_{\ee})^3}(f_\ee(T),T)
	=-\frac{\phi_{\bn}^2\phi_{\bt\bt}}{\phi_\ee^3}(f_\ee(T),T).
	\end{equation}
	Now property {\rm (\ref{lem:shock graph-i3})} directly follows
	from (\ref{tangDerivOnShock}) and properties
	(\ref{lem:shock graph-i2-00})--(\ref{lem:shock graph-i2-01}).
	This completes the proof.
\end{proof}

In order to show Lemma \ref{distToSonicCenter-lemma} below,  we
first note the following property of solutions of the potential flow equation:

\begin{lem}[\cite{cf-book2014shockreflection}, Lemma 6.1.4]
	\label{lem:interesting lemma}
	Let $\Omega\subset\mathbb{R}^2$ be open, and let $\Omega$ be divided by a smooth curve $S$ into two
	subdomains $\Omega^+$ and $\Omega^-$.
	Let $\varphi\in C^{0,1}(\Omega)$ be a weak solution in $\Omega$ as defined in Definition {\rm \ref{def:weak solution}}
	such that $\varphi\in C^2(\Omega^{\pm})\cap C^1(\Omega^{\pm}\cup S)$.
	Denote $\varphi^{\pm}:=\varphi\big|_{\Omega^{\pm}}$. Suppose that $\varphi$ is a constant state
	in $\Omega^-$ with density $\rho_-$ and sound speed $c_-$, that is,
	$$
	\varphi^-(\xxi)=-\frac{1}{2}|\xxi|^2+\vv_-\cdot\xxi +A^-,
	$$
	where $\vv_-$ is a constant vector and $A^-$ is a constant.
	Let $P_k\in S$, for $k=1,2$, be such that
	\begin{itemize}
		\item[\rm (i)] $\varphi^-$ is supersonic at $P_k${\rm :}
		$|D\varphi^-|>c_-:=c(|D\varphi^-|^2,\varphi^-, \rho_0)$ at $P_k${\rm ;}
		
		\smallskip
		\item[\rm (ii)] $D\varphi^-\cdot\boldsymbol{\nu}>D\varphi^+\cdot\boldsymbol{\nu}>0$ at $P_k$,
		where $\boldsymbol{\nu}$ is the unit normal vector to $S$ oriented
		from $\Omega^-$ to $\Omega^+${\rm ;}
		
		\smallskip
		\item[\rm (iii)] For  the tangent line $L_{P_k}$ to $S$ at $P_k$, $k=1,2$,
		$L_{P_1}$ is parallel to $L_{P_2}$ with
		$\boldsymbol{\nu}(P_1)=\boldsymbol{\nu}(P_2)${\rm ;}
		
		\smallskip
		\item[\rm (iv)] $d(P_1)>d(P_2)$, where $d(P_k)$ is the distance between line $L_{P_k}$ and center $O^-=\vv_-$ of
		the sonic circle of state $\varphi^-$ for each $k=1,2$.
	\end{itemize}
	Then
	$$
	\phi^+_{\boldsymbol{\nu}}(P_1)<\phi^+_{\boldsymbol{\nu}}(P_2),
	$$
	where $\phi^+(\xxi)=\frac{1}{2}|\xxi|^2+\varphi^+(\xxi)$.
\end{lem}

Now we prove a technical fact used in the main argument of the paper.

\begin{lem}\label{distToSonicCenter-lemma}
	Let $\Omega$, $\Gsh$, and $\phi$ be as in Lemma {\rm \ref{lem:shock graph}}.
	For the unit vector $\ee\in Con$, let $(S,T)$ be the coordinates defined in
	Lemma {\rm \ref{lem:shock graph}}, and let $f_\ee$ be the function
	from Lemma {\rm \ref{lem:shock graph}(\ref{lem:shock graph-i1})}.
	Assume that, for two different points $P=(T,f_\ee(T))$
	and $P_1=(T_1,f_\ee(T_1))$
	on $\Gsh$,
	$$
	f_\ee(T)>f_\ee(T_1)+f'_\ee(T)(T-T_1),
	\qquad f'_\ee(T)=f'_\ee(T_1).
	$$
	Then
	\begin{enumerate}[\rm (i)]
		\item \label{distToSonicCenter-lemma-i1}
		$d(P):=\mbox{\rm dist}(O_0,L_{P})>\mbox{\rm dist}(O_0,L_{P_1})=:d(P_1)$,
		where $O_0$ is the center of sonic circle of state $(0)$, and $L_{P}$ and $L_{P_1}$ are
		the tangent lines of $\Gsh$ at $P$ and $P_1$, respectively.
		
		\smallskip
		\item  \label{distToSonicCenter-lemma-i2}
		$\phi_\ee(P)>\phi_\ee(P_1).$
	\end{enumerate}
\end{lem}

\begin{proof}
	First, since $f'_\ee(T)=f'_\ee(T_1)$, denote $\bn:=\bn(P)=\bn(P_1)$ and $\bt:=\bt(P)=\bt(P_1)$. In addition,
	$$
	d(P)=\mbox{dist}(O_0,L_P)=PO_0\cdot\bn,\qquad
	d(P_1)=\mbox{dist}(O_0,L_{P_1})=P_1O_0\cdot \bn.
	$$
	Therefore, it suffices to find the expression of vector $PO_0$ in terms of vector $P_1O_0$.
	
	\smallskip
	From the definition of the $(S,T)$--coordinates and the shock function $f_\ee$ in the previous lemmas, we have
	$$
	(T,f_\ee(T))=(T_1,f_\ee(T_1))+\big(f_\ee(T)-f_\ee(T_1)\big)\ee+(T-T_1)\ee^{\perp},
	$$
	so that
	\begin{align}
	(T,f_\ee(T))=&\, (T_1,f_\ee(T_1))+\big(f_\ee(T)-f_\ee(T_1)-f_\ee'(T_1)(T-T_1)\big)\ee\nonumber\\
	         &\,\, +(T-T_1)\big(\ee^{\perp}+f_\ee'(T_1)\ee\big).
	\end{align}
	Since $\big(\ee^{\perp}+f_\ee'(T_1)\ee\big)\cdot\bn=0$,
	$$
	PO_0\cdot\bn=\big(O_0-(T,f_\ee(T))\big)\cdot\bn
	=P_1O_0\cdot\bn
	-\big(f_\ee(T)-f_\ee(T_1)-f_\ee'(T_1)(T-T_1)\big)\ee\cdot\bn.
	$$
	From Lemma \ref{lem:shock graph}(\ref{lem:shock graph-i2-00}) and
	the fact that $f_\ee(T)>f_\ee(T_1)+f'_\ee(T_1)(T-T_1)$,
	we conclude that $PO_0\cdot\bn >P_1O_0\cdot\bn$.
	This implies
	$$
	d(P)=\mbox{dist}(O_0,L_{P})>\mbox{dist}(O_0,L_{P_1})=d(P_1).
	$$
	Then (\ref{distToSonicCenter-lemma-i1}) is proved.
	
	\medskip
	Now we prove (\ref{distToSonicCenter-lemma-i2}). By (\ref{distToSonicCenter-lemma-i1})
	and Lemma \ref{lem:interesting lemma},
	$$
	\phi_{\bn}(P)<\phi_{\bn}(P_1).
	$$
	Also,  $\partial_{\bt}\phi=0$ on $\Gsh$
	by the first condition in (\ref{equ:boundary}).
	Thus,
	$\partial_{\bt}\phi(P)=\partial_{\bt}\phi(P_1)=0$.
	Then, using $\ee\cdot\bn<0$, we obtain
	$$
	D\phi(P)\cdot\ee =\partial_{\bn}\phi(P)\,\bn\cdot\ee
	>\partial_{\bn}\phi(P_1)\,\bn\cdot\ee
	=D\phi(P_1)\cdot\ee,
	$$
	which is (\ref{distToSonicCenter-lemma-i2}).
\end{proof}

\subsection{\, Real analyticity of the shock and related properties}
In this subsection, we show that the shock, $\Gsh^0$, is real analytic and
$\phi$ is real analytic in $\Omega\cup\Gsh^0$.
To see that, we note that the free boundary
problem (\ref{equ:potential flow equation}) and
(\ref{RH conditon involve normal direction})--(\ref{RH conditon continous})
can be written in terms of $\phi=\varphi-\varphi_0$ with  $\bn=\frac{D\phi}{|D\phi|}$ in the form:
\begin{align}
&N(D^2\phi, D\phi, \phi, \xxi)=0 &&\mbox{in $\Omega$}, \label{eqnForFbp}\\
&M(D\phi, \phi, \xxi)=0 &&\mbox{on $\Gsh$}, \label{derivBCForFbp}\\
&\phi=0 &&\mbox{on $\Gsh$}, \label{zeroBCForFbp}
\end{align}
where, for $(\rr, \pp, z, \xxi)\in S^{2\times 2}\times \mr^2\times \mr\times\overline\Omega$
with $S^{2\times 2}$ as the
set of symmetric $2\times 2$ matrices,
\begin{align}
&N(\rr, \pp, z, \xxi):=
\big(c^2-(p_1-\xi_1)^2\big)r_{11}-2(p_1-\xi_1)(p_2-\xi_2)r_{12}\nonumber\\
&\qquad\qquad\qquad\,\,\,\, +\big(c^2-(p_2-\xi_2)^2\big)r_{22}, \label{func-eqnForFbp}\\
&M(\pp, z, \xxi):=\big(\rho(\pp, z, \xxi) (\pp+D\varphi_0)-\rho_0 D\varphi_0\big)\cdot
\frac{\pp}{|\pp|} \label{func-derivBCForFbp}
\end{align}
with
\begin{eqnarray*}
	c^2(\pp, z, \xxi)=\rho_0^{\gamma-1}-(\gamma-1)\big(z-\xxi\cdot \pp+\frac{1}{2}|\pp|^2\big),\quad\,
	\rho(\pp, z, \xxi)=c(\pp, z, \xxi)^{\frac{2}{\gamma-1}}.
\end{eqnarray*}

Equation \eqref{eqnForFbp} is quasilinear, so that its ellipticity depends only on $(\pp, z, \xxi)$.
By assumption, the equation is strictly elliptic on solution $\phi$, {\it i.e.},
for $(\pp, z, \xxi) = (D\phi(P), \phi(P), P)$ for all
$P\in \Omega\cup\Gsh^0$.

Furthermore, it is easy to check by an explicit calculation that the ellipticity of the equation
and the fact that $\bn=\frac{D\phi}{|D\phi|}$ on $\Gsh^0$
imply the obliqueness of the boundary condition (\ref{derivBCForFbp})
on $\Gsh^0$ for solution $\phi$:
$$
D_\pp M(D\phi, \phi, \xxi)\cdot \bn>0 \qquad \;\mbox{ on $\Gsh^0$}.
$$

Moreover, from the explicit expressions, $N(\rr, \pp, z, \xxi)$ is real analytic on
$S^{2\times 2}\times \mr^2\times \mr\times\overline\Omega$, and
$M(\pp, z, \xxi)$ is real analytic on
$$
\{(\pp, z, \xxi)\;: \; \rho_0^{\gamma-1}-(\gamma-1)\big(z-\xxi\cdot \pp+\frac{1}{2}|\pp|^2\big)>0\}.
$$
Since $\varphi_0$ is pseudo-supersonic,  $\varphi$ is pseudo-subsonic on $\Gsh$, and conditions
(\ref{RH conditon involve normal direction})--(\ref{RH conditon continous}) hold, we have
$$
\rho(D\phi, \phi, \xxi)>\rho_0 \qquad \mbox{ for all $\xxi\in\Gsh$},
$$
so that
$$
\rho_0^{\gamma-1}-(\gamma-1)\big(z-\xxi\cdot \pp+\frac{1}{2}|\pp|^2\big)>\rho_0^{\gamma-1}
$$
for all $(\pp, z, \xxi)=(D\phi(\xxi), \phi(\xxi), \xxi)$ with $\xxi\in\Gsh$.
That is,
$M(\pp, z, \xxi)$ is real analytic in an open set containing $(\pp, z, \xxi)=(D\phi(\xxi), \phi(\xxi), \xxi)$
for all $\xxi\in\Gsh$.

Then, by Theorem 2 in Kinderlehrer-Nirenberg \cite{KinderlehrerNirenberg-ASNSPCS1977373},
we have the following lemma:

\begin{lem}\label{lem:analyticity}
	Let $\Omega$, $\Gsh$, and  $\phi$ be as in Lemma {\rm \ref{lem:shock graph}}.
	Then $\Gsh^0$ is real analytic in its relative interior{\rm ;}
	in particular,
	$f_\ee$ is real analytic on $(T_A, T_B)$ for any $\ee\in Con$.
	Moreover,  $\phi$ is real analytic in $\Omega$ up to $\Gsh^0$.
\end{lem}

We remark here that the assertion on the analyticity of the solution up to the free boundary
is not listed in the formulation of Theorem 2 in
\cite{KinderlehrerNirenberg-ASNSPCS1977373}, but is shown in its proof.

\smallskip
Now we show the following fact that will be repeatedly used for subsequent development.
\begin{lem}\label{noZeroesInfiniteOrderLemma}
	Let $\Omega$, $\Gsh$, and  $\phi$ be as in Lemma {\rm \ref{lem:shock graph}}.
	Assume that $\phi$ is not a constant state in $\Omega$.
	Let $\ee\in Con$, and let $T_A, T_B$, and $f_\ee$ be from
	Lemma {\rm \ref{lem:shock graph}(\ref{lem:shock graph-i1})}.
	Then, for any $T_P\in(T_A, T_B)$, there exists an integer $k\ge 2$
	such that $f_\ee^{(k)}(T_P)\ne 0$.
\end{lem}
\begin{proof}
	In this proof, we use equation
	\eqref{equ:boundary for phi e}
	in the $(S,T)$--coordinates with basis $\{\bn,\bt\}=\{\bn(P),\bt(P)\}$ (constant vectors).
	
	We argue by a contradiction.
	Assume that $P=(f_\ee(T_P),T_P)\in\Gsh^0$ is such that $f_\ee^{(i)}(T_P)=0$ for all $i>1$.
	From (\ref{tangDerivOnShock}) and
	its derivatives with respect to $T$,
	we use assumption (A1) of Theorem \ref{thm:main theorem} to obtain
	$$
	\partial_{\bt}^i\phi(P)=0\qquad \mbox{ for all $i>1$}.
	$$
	
	Writing \eqref{equ:boundary for phi e} in the coordinates with the basis of the normal vector $\bn$
	and tangent vector $\bt$
	on $\Gsh$ at $P$,
	and writing vector $\bh$ given in \eqref{Expression-h} as $\bh=h_{\bn} \bn+h_{\bt}\bt$,
	we have
	\begin{equation}\label{DerivRH-nu-tau}
	h_{\bt} \phi_{\bt\bt}+h_{\bn}\phi_{\bn\bt}=0 \qquad\,\, \mbox{ at $P$}.
	\end{equation}
	From  \eqref{equ:27 oblique of h}, $h_{\bn}=\bh\cdot\bn<0$
	at $P$ so that $\phi_{\bt\bt}=0$ implies that $\phi_{\bn\bt}=0$.
	Now, from equation \eqref{equ:in ST coordinate}
	and assumption (A3) of Theorem \ref{thm:main theorem},
	we obtain that $\phi_{\bn\bn}=0$, so that
	\begin{equation}\label{all2ndDerivZero}
	\phi_{\bt\bt}=\phi_{\bn\bt}=\phi_{\bn\bn}=0\qquad\,\,\mbox{ at $P$}.
	\end{equation}
	
	Continuing inductively with respect to order $k$ of differentiation,
	we fix $k>2$,
	and assume that $D^j\phi(P)=0$ for $j=2,\dots, k-1$. With this,
	taking the $(k-1)$-th tangential
	derivative of (\ref{equ:boundary for phi e}), we obtain
	$$
	h_{\bt}\partial^k_{\bt}\phi+h_{\bn}\partial^{k-1}_{\bt}\partial_{\bn}\phi=0 \,\qquad \mbox{ at $P$}.
	$$
	Thus, from $\partial^k_{\bt}\phi(P)=0$, we have
	$$
	\partial^{k-1}_{\bt}\partial_{\bn}\phi=0\;\qquad \mbox{ at $P$}.
	$$
	Then, using the $\partial^{k-2}_T$--derivative of equation  \eqref{equ:in ST coordinate},
	we see that $\partial^{k-2}_{\bt}\partial^2_{\bn}\phi(P)=0$. Furthermore,
	using the $\partial^{k-3}_T\partial_S$--derivative of equation \eqref{equ:in ST coordinate},
	we see that $\partial^{k-3}_{\bt}\partial^3_{\bn}\phi(P)=0$, {\it etc}.
	Thus, we obtain that
	all the derivatives of $\phi$ of order two and higher are zero at $P$.
	Now, from the analyticity of $\phi$ up to $\Gsh^0\ni P$,
	we conclude that $\phi$ is linear in the whole domain $\Omega$,
	which is a contradiction to the condition of Theorem {\rm \ref{thm:main theorem}}
	that $\varphi$
	is not a constant state.
\end{proof}

\subsection{\, Minimal and maximal chains: Existence and properties}
\label{ExistMinMaxChainsSection}
In this subsection, we assume that $\Omega\subset \mR^2$  is open, bounded, and connected,
and that $\partial\Omega$ is a continuous curve, piecewise $C^{1,\alpha}$
up to the endpoints of each smooth part
and has a finite number of smooth parts.
Moreover, at each corner point of $\partial\Omega$,
angle $\theta$ between the arcs meeting at that point from the interior of
$\Omega$ satisfies $\theta\in (0, \pi)$.
Note that Theorem \ref{thm:main theorem} requires all these conditions.

Let $\phi\in C(\overline{\Omega})\cap C^2(\Omega\cup\Gsh^0)\cap C^3(\Omega)$
be a solution of equation (\ref{equ:study}) in $\Omega$ satisfying
conditions (A2)--(A3) of Theorem \ref{thm:main theorem}.
Let ${\mathbf e}\in\mathbb{R}^2$ be a unit vector.

\begin{defi}\label{def:minimal chain}
	Let $E_1, E_2\in \partial\Omega$.
	We say that points $E_1$ and $E_2$ are connected by a minimal {\rm (}resp.  maximal{\rm )}
	chain with radius $r$ if
	there exist $r>0$, integer $k_1\ge 1$, and a chain of balls $\{B_{r}(C^i)\}_{i=0}^{k_1}$ such that
	\begin{enumerate}[\rm (a)]
		\item \label{def:minimal chain-Ia}
		$C^0=E_1$, $C^{k_1}=E_2$, and $C^i\in\overline\Omega$  for $i=0,\dots, k_1${\rm ;}
		
		\smallskip
		\item \label{def:minimal chain-Ib}
		$C^{i+1}\in \overline {B_{r}(C^i)\cap\Omega}$ for $i=0,\dots, k_1-1${\rm ;}
		
		\smallskip
		\item \label{def:minimal chain-Ic}
		$\phi_\ee(C^{i+1})={\displaystyle \min_{\overline{B_{r}(C^i)\cap\Omega}}\phi_\ee<\phi_\ee(C^i)}$
		{\rm (}resp. $\phi_\ee(C^{i+1})={\displaystyle\max_{\overline{B_{r}(C^i)\cap\Omega}}\phi_\ee>\phi_\ee(C^i)}${\rm )}
		for $i=0,\dots, k_1-1${\rm ;}
		
		\smallskip
		\item \label{def:minimal chain-Id}
		$\phi_\ee(C^{k_1})={\displaystyle\min_{\overline{B_{r}(C^{k_1})\cap\Omega}}\phi_\ee}$
		{\rm (}resp. $\phi_\ee(C^{k_1})={\displaystyle\max_{\overline{B_{r}(C^{k_1})\cap\Omega}}\phi_\ee}${\rm )}.
	\end{enumerate}
	For such a chain, we also use the following terminology{\rm :} The chain starts at $E_1$ and ends at $E_2$, or
	the chain is from $E_1$ to $E_2$.
\end{defi}

\begin{rem}
	This definition does not rule out the possibility that $B_r(C^i)\cap\partial\Omega\ne\emptyset$,
	or even $C^i\in\partial\Omega$, for some or all $i=0,\dots, k_1-1$.
\end{rem}

\begin{rem}
	Radius $r$ is a parameter in the definition of minimal or maximal chains.
	We do not fix $r$ at this point.
	In the proof of Theorems {\rm \ref{thm:main theorem}}--{\rm \ref{thm:main theorem-strictUnif}},
	the radii will be determined for various chains in such a way that
	Lemmas {\rm \ref{lem:ExistMinMaxChain}}--{\rm \ref{lem:chainsDoNotIntersect-MinMin}} below
	can be applied.
\end{rem}

We now consider the minimal and maximal chains for $\phi_\ee$ in $\Omega$.
In the results of these subsections, all the constants depend on the parameters
in the conditions
listed above, {\it i.e.}, the $C^{1,\alpha}$--norm of the smooth parts of $\partial \Omega$,
the angles at the corner points,
and $\|\phi\|_{C^{1,\alpha}(\overline\Omega)}$, in addition to the further parameters
listed in the statements.

We first show that the chains with sufficiently small radius are connected sets.

\begin{lem}\label{lem:connectBall}
	There exists $r^*>0$, depending only on the $C^{1,\alpha}$--norms of
	the smooth parts of $\partial\Omega$
	and angles $\theta\in (0, \pi)$ in the corner points, such that,
	for any $E\in\overline\Omega$
	and $r\in (0, r^*]$,
	\begin{enumerate}[\rm (i)]
		\item\label{lem:connectBall-i1}
		$B_r(E)\cap\Omega$ is connected{\rm ;}
	
	\smallskip
		\item\label{lem:connectBall-i2}
		For any $G\in\overline{B_r(E)\cap\Omega}$,
		$B_r(E)\cap B_r(G)\cap\Omega$ is nonempty.
	\end{enumerate}
\end{lem}

\begin{proof}
	We only sketch the argument,
	since the details are standard.
	
	\smallskip
	We first prove (\ref{lem:connectBall-i1}).
	Denote $Q_{\rc}:=(-L\rc, L\rc)\times(-\rc, \rc)$.
	The conditions on $\partial\Omega$ imply that there exist $L, N>4$
	such that, for any sufficiently small $\rc>0$, the following facts hold:
	
	\begin{enumerate}[\rm (a)]
		\item \label{domainSmooth-case}
		If $P\in\partial\Omega$ has  the distance at least
		$N r$ from the corner points of $\partial\Omega$,
		then, in some orthonormal coordinate system in $\mR^2$ with the origin at $P$,
		\begin{equation}\label{domainSmoothCase}
        \begin{split}
		&\Omega\cap Q_{2\rc}=\{(s,t)\in Q_{2\rc}\;:\; s>g(t)\},\\
		&\partial\Omega\cap Q_{2\rc}=\{(s,t)\in Q_{2\rc}\;:\; s=g(t)\}
	    \end{split}
     	\end{equation}
		for some $g\in C^{1,\alpha}(\mR)$ with $g(0)=g'(0)=0$;
		
		\smallskip
		\item\label{domainNearCorner-case}
		If $P\in\partial\Omega$ is a corner point, then, in some
		orthonormal coordinate system in $\mR^2$ with the origin at $P$,
		\begin{equation}\label{domainNearCorner}
		\begin{split}
		&\Omega\cap Q_{4N \rc}=\{(s,t)\in Q_{4N \rc}\;:\; s>\max(g_1(t), g_2(t))\},\\[1mm]
		&\partial\Omega\cap Q_{4N \rc}=\{(s,t)\in Q_{4N\rc}\;:\; s=\max(g_1(t), g_2(t))\}
		\end{split}
		\end{equation}
		for some $g_1$ and $g_2$ satisfying
		\begin{equation}\label{domainNearCorner-funct}
		\begin{split}
		&g_1, g_2\in C^{1,\alpha}(\mR), \;\; g_1(0)=g_2(0)=0,\;\;g_1'(0)< 0,\;\;g_2'(0)> 0,\\
		& g_1(t)>g_2(t)\;\mbox{  for $t<0$},\qquad\;
		g_1(t)<g_2(t)\;\mbox{  for $t>0$}.
		\end{split}
		\end{equation}
		Note that, in order to obtain
		(\ref{domainNearCorner})--(\ref{domainNearCorner-funct}),
		we use the condition that angle $\theta$ at $P$ satisfies $\theta \in(0, \pi)$.
	\end{enumerate}
	
	\smallskip
	Let $E\in \overline\Omega$.  Without loss of generality, we assume that $\dist(E, \partial\Omega)<r$;
	otherwise, (\ref{lem:connectBall-i1}) already holds.
	
	\smallskip
	The first case is that the distance from $E$ to the corner points is at least $2Nr$.
	Then, denoting by $P$ the nearest point on $\partial\Omega$ to $E$,
	it follows that $P$ satisfies the condition for Case (a) above,
	so that $P$ is the unique nearest point on $\partial\Omega$ to $E$,
	and $E=(s^*, 0)$ with $s^*\in [0, r)$ in the coordinate system
	described in (\ref{domainSmooth-case}) above.
	Then, denoting
	$f^\pm(t):=s^*\pm\sqrt{r^2-t^2}$ on $[-r, r]$,
	and using that $|g'(t)|\le Ct^\alpha$ and
	$|g(t)|\le Ct^{1+\alpha}$ on $[-r, r]$ for $C$ depending on
	the $C^{1,\alpha}$--norm of the smooth
	parts of $\partial\Omega$, we obtain that, if $r$ is small,
	there exist $t^+\in (\frac 9{10} r, r]$
	and $t^-\in [-r, -\frac 9{10} r)$
	such that
	\begin{equation}\label{connectedRnbhd-smooth-0}
	f^+>g \,\,\,\mbox{on $(t^-, t^+)$}, \qquad
	\, f^+<g \,\,\, \mbox{on $[-r, r]\setminus [t^-, t^+]$},
	\end{equation}
	where the last set is empty if $t^\pm=\pm r$, and
	\begin{equation}\label{connectedRnbhd-smooth}
	\Omega\cap B_{r}(E)=\{(s,t)\;:\;
	\max(f^-(t), g(t))<s<f^+(t),\; t^-<t< t^+\},
	\end{equation}
	which is a connected set, by  the first inequality in
	(\ref{connectedRnbhd-smooth-0})
	and the fact that $f^-<f^+$ in $(-r,r)$.
	
	\smallskip
	In the other case, when the distance from $E$
	to the corner points is smaller than $2Nr$,
	we argue similarly by using the coordinates
	described in Case (\ref{domainNearCorner-case}) above,
	related to the corner point $P$ that is the nearest to $E$.
	The existence of such a coordinate system and the fact that $\dist(E,P)<2Nr$
	also imply that the nearest corner
	$P$ is unique for $E$. Then, in these coordinates,
	$$
	E=(s^*, t^*)\in
	\overline{\Omega\cap Q_{2N r}}.
	$$
	Let
	$$
	\Omega^{(k)}:= \{s>g^{(k)}(t),\;\; t\in\mR\},\,\,\,\,	
 \Gamma^{(k)}:= \{s=g^{(k)}(t),\;\; t\in\mR\}\qquad \mbox{ for $k=1,2$}.
	$$
	Then, by (\ref{domainNearCorner}),
	\begin{equation}\label{domainNearCorner-Un}
	\Omega\cap Q_{4N r}=\Omega^{(1)}\cap\Omega^{(2)}\cap Q_{4N r}.
	\end{equation}
	
	If $r$ is sufficiently small, we deduce from (\ref{domainNearCorner-funct})
	that there exists $\lambda\in(0,1)$ such that
	\begin{equation}\label{domainNearCorner-funct-1}
	-\lambda^{-1}\le g_1'(t)\le -\lambda,\quad
	\lambda\le g_2'(t)\le \lambda^{-1}
	\qquad\,\, \mbox{ for all $t\in (-4Nr, 4Nr)$}.
	\end{equation}
	
	Let $P^{(k)}=(s^{(k)},  t^{(k)})$ be the nearest point to $E$ on $\Gamma^{(k)}$.
	Then
	$P^{(k)}\in \Gamma^{(k)}\cap Q_{2N r}$.

	Assume that ${\dist}(E,\Gamma^{(1)}) <r$, which implies that $E\in B_r(P^{(1)})$.
	Using  (\ref{domainNearCorner-funct-1}),
	$g_1'(t^{(1)})<0$.
	Then, reducing $r$ depending on the $C^{1,\alpha}$--norm of
	$g_1$,
	rotating the coordinate system $(s,t)$
	by angle $\arctan(|g_1'(t^{(1)})|)$ clockwise, and shifting the origin
	into $P^{(1)}$, we conclude that, in the resulting coordinate system $(S,T)$,
	\begin{align*}
   	&\Omega^{(1)}\cap Q_{r}=\{(S,T)\in Q_{r}\;:\; S>G(T)\},\\
	&\Gamma^{(1)}\cap Q_{r}=\{(S,T)\in Q_{r}\;:\; S=G(T)\},
	\end{align*}
	for some $G\in C^{1,\alpha}(\mR)$ with $G(0)=G'(0)=0$, which
	is similar to (\ref{domainSmoothCase}). Then, arguing as in Case
	(\ref{domainSmooth-case}), we obtain an expression similar
	to (\ref{connectedRnbhd-smooth}) for $\Omega^{(1)}\cap B_{r}(E)$
	in the $(S,T)$--coordinates.
	Changing
	back to the $(s,t)$--coordinates and possibly further reducing $r$ depending
	on $\lambda$, we obtain the existence of $t^-\in [t^*- r, t^*)$
	such that
	\begin{equation}\label{connectedRnbhd-corner-0-0}
	f^+>g_1 \,\,\mbox{on $(t^-, t^*+r)$}, \quad\,\,
	f^+<g_1 \,\, \mbox{on $[t^*-r, t^*+r]\setminus [t^-, t^*+r]$},
	\end{equation}
	where the last set is empty if $t^-=t^*-r$, and
	\begin{equation}\label{connectedRnbhd-corner-0}
	\Omega^{(1)}\cap B_{r}(E)=\{(s,t)\;:\;
	\max(f^-(t), g_1(t))<s<f^+(t),\; t^-<t<t^*+r\},
	\end{equation}
	where $f^\pm(t):=s^*\pm\sqrt{r^2-(t-t^*)^2}$ on $[t^*-r, t^*+r]$.
	Note that (\ref{connectedRnbhd-corner-0}) also holds if
	${\dist}(E,\Gamma^{(1)})\ge r$: Indeed, in this case,
	$\Omega^{(1)}\cap B_{r}(E)=B_{r}(E)$
	and $g_1(t)\le f^-(t)$ on $[t^*-r, t^*+r]$, so that
	(\ref{connectedRnbhd-corner-0}) holds with $t^-=t^*-r$.
	
	By a similar argument, we show the existence of $t^+\in (t^*, t^*+r]$ such
	that
	\begin{equation}\label{connectedRnbhd-corner-1-0}
	f^+>g_2 \,\,\mbox{on $(t^*-r, t^+)$}, \quad\,\,
	f^+<g_2 \,\, \mbox{on $[t^*-r, t^*+r]\setminus [t^*-r, t^+]$},
	\end{equation}
	where the last set is empty if $t^+=t^*+r$, and
	\begin{equation}\label{connectedRnbhd-corner-1}
	\Omega^{(2)}\cap B_{r}(E)=\{(s,t)\;:\;
	\max(f^-(t), g_2(t))<s<f^+(t),\; t^*-r<t<t^+\}.
	\end{equation}
	From (\ref{domainNearCorner-Un}), (\ref{connectedRnbhd-corner-0}), and
	(\ref{connectedRnbhd-corner-1}), we obtain
	\begin{equation}\label{connectedRnbhd-corner}
	\Omega\cap B_{r}(E)=\{(s,t)\;:\;
	\max(f^-(t), g_1(t), g_2(t))<s<f^+(t),\; t^-<t<t^+\},
	\end{equation}
	which is a connected set, by the first inequalities
	in (\ref{connectedRnbhd-corner-0-0}) and
	(\ref{connectedRnbhd-corner-1-0})
	and the fact that $f^-<f^+$ in $(t^*-r, t^*+r)$.
	
	\medskip
	Now we prove assertion (\ref{lem:connectBall-i2}).
	We can assume that $G\in B_r(E)\cap\partial\Omega$;  otherwise,
	(\ref{lem:connectBall-i2}) already holds.
	Then we again consider two cases, as above, and use
	expressions (\ref{connectedRnbhd-smooth}) and (\ref{connectedRnbhd-corner})
	to conclude the proof.
\end{proof}

\begin{rem}
	The condition that the interior angles $\theta$ at the corner points of $\partial\Omega$
	satisfy $\theta\in(0, \pi)$ is necessary for Lemma {\rm \ref{lem:connectBall}}.
	Indeed, let $\theta\in(\pi, 2\pi)$ at some corner $Q\in\partial\Omega$.
	For simplicity, consider first the case when $\partial\Omega\cap B_{5R}(Q)$
	consists of two straight lines intersecting at $Q$ for some $R>0$.
	Then it is easy to see that, for any $E\in\partial\Omega$ with $d:=\dist(E, Q)\in (0, R]$,
	$B_r(E)\cap\Omega$ is not connected for all $r\in (d\sin(2\pi-\theta),\, d)$.
	With the assumption that $\partial\Omega$ is piecewise $C^{1,\alpha}$ up to the
	corner points {\rm (}without assumption that $\partial\Omega\cap B_{5R}(Q)$
	is piecewise-linear{\rm )}, the same is
	true for all $r\in (d_1, d)$ for some
	$d_1\in (d\sin(2\pi-\theta), d)$
	if $d$ is sufficiently small.
\end{rem}

\begin{lem}\label{lem:connected-chain}
	There exists $r^*>0$ such that
	any chain in Definition {\rm \ref{def:minimal chain}} with $r\in (0, r^*)$ satisfies
	
	\begin{enumerate}[\rm (i)]
		\item \label{lem:connected-chain-i1}
		${\displaystyle \bigcup_{i=0}^{k_1}\left(B_{r}(C^i)\cap \Omega\right)}$
		is connected{\rm ;}
		
		\smallskip
		\item \label{lem:connected-chain-i2}
		There exists a continuous curve
		$\mS$ with endpoints $C^0$ and $C^{k_1}$
		such that
		$$
		\mS^0\subset {\displaystyle \bigcup_{i=0}^{k_1}\left(B_{r}(C^i)\cap \Omega\right)},\,\,\,\,
		\; {\dist}(\mS_r, \partial\Omega)>0 \qquad\;\;
		\mbox{for all $r>0$},
		$$
		where $\mS_r=\mS\setminus\big(B_r(C^0)\cup B_r(C^{k_1})\big)$,
		and $\mS^0$ denotes the open curve that does not include the endpoints.
		More precisely, $\mS=g([0,1])$, where $g\in C([0, 1];\mR^2)$ and is locally Lipschitz
		on $(0,1)$ with $g(0)=C^0$, $g(1)=C^{k_1}$, and
		${\displaystyle g(t)\in\bigcup_{i=0}^{k_1}\left(B_{r}(C^i)\cap \Omega\right)}$ for all $t\in (0,1)$.
	\end{enumerate}
\end{lem}

\begin{proof}
	We use
	$r^*$ in Lemma \ref{lem:connectBall}.
	We prove (\ref{lem:connected-chain-i1}) by induction:
	We first note that $B_{r}(C^1)\cap \Omega$
	is connected by Lemma \ref{lem:connectBall}(\ref{lem:connectBall-i1}).
	Suppose that, for $m\in \{1, 2,\dots, k_1-1\}$,
	${\displaystyle A_m=\bigcup_{i=0}^{m}\left(B_{r}(C^i)\cap \Omega\right)}$
	is connected.
	We note that $A_m$ has a nonempty intersection with
	$B_{r}(C^{m+1})\cap \Omega$ by
	Definition \ref{def:minimal chain}(\ref{def:minimal chain-Ib}) and
	Lemma \ref{lem:connectBall}(\ref{lem:connectBall-i2}).
	Also, $B_{r}(C^{m+1})\cap \Omega$
	is a connected set.
	Then it follows that
	$\displaystyle \bigcup_{i=0}^{m+1}\left(B_{r}(C^i)\cap \Omega\right)$ is connected.
	This proves \eqref{lem:connected-chain-i1}.
	
	Assertion
	\eqref{lem:connected-chain-i2} with reduced $r^*$ follows
	from Lemmas \ref{lem:connected-path} and
	\ref{lem:connected-path-inChain}.
\end{proof}

\begin{rem}\label{rem:connected-chain-i2}
	Lemma {\rm \ref{lem:connected-chain}\eqref{lem:connected-chain-i2}}
	implies that $\mS^0$ lies in the interior of $\Omega$.
\end{rem}

Now we show the existence of minimal (resp. maximal) chains. We use
$r^*$ from Lemma \ref{lem:connected-chain} from now on.

\begin{lem}\label{lem:ExistMinMaxChain}
	If $E_1\in\partial\Omega$ and is not a local minimum point {\rm (}resp. maximum point{\rm )}
	of $\phi_\ee$ with respect to $\overline\Omega$, then, for any $r\in (0, r^*)$,
	there exists a minimal {\rm (}resp. maximal{\rm )} chain $\{\E^i\}_{i=0}^{k_1}$ for $\phi_\ee$ of radius $r$
	in the sense of  Definition {\rm \ref{def:minimal chain}}, starting at $E_1$, i.e., $\E^0=E_1$.
	Moreover, $\E^{k_1}\in\partial\Omega$ is a local minimum {\rm (}resp. maximum{\rm )} point of
	$\phi_\ee$ with respect
	to $\overline\Omega$, and $\phi_\ee(\E^{k_1})<\phi_\ee(E_1)$ {\rm (}resp. $\phi_\ee(\E^{k_1})>\phi_\ee(E_1)${\rm )}.
\end{lem}

\begin{proof}
	We discuss only the case of the minimal chain,
	since the case of the maximal chain can be considered similarly.
	Thus, $E_1$ is not a local minimum point
	of $\phi_\ee$ with respect to $\overline\Omega$.
	
	\smallskip
	Let $\E^0=E_1$. Choose $\E^{i+1}$ to be the point such that
	the minimum of $w=\phi_\ee$ in $\overline{B_{r}(\E^i)}\cap\overline{\Omega}$
	is attained at $\E^{i+1}$, provided that $w(\E^{i+1})<w( \E^i)$;
	otherwise ({\it i.e.}, if the minimum of $w=\phi_\ee$ in $\overline{B_{r_1}(\E^i)}\cap\overline{\Omega}$
	is attained at $\E^i$ itself),
	the process ends and we set $k_1:= i$.
	
	\smallskip
	In order to show that $\{\E^i\}_{i=0}^{k_1}$ is a minimal chain for $r\in (0, r^*)$,
	it suffices to
	show that $\E^{k_1}\in\partial\Omega$ and that $k_1$ is positive and finite.
	These can be seen as follows:
	
	\begin{enumerate}[\rm (i)]
		\item Since $\E^0 = E_1$ is not a local minimum point relative to $\overline\Omega$,
		it follows that $\E^1\ne \E^0$ so that $k_1\ge 1$ and $\phi_\ee(\E^0)<\phi_\ee(\E^1)$.
		
		\smallskip
		\item
		There is only a finite number of $\{\E^i\}$.
		Indeed, on the contrary, since domain $\Omega$ is bounded,
		there exists a subsequence $\{\E^{i_m}\}$
		such that $\E^{i_m}\rightarrow \hat C$ as $m\rightarrow\infty$,
		where $\hat C$ is a point lying in $\overline{\Omega}$.
		Thus, for any $\epsilon<r$, there is a large number $N$ such that, for any $j,\,m>N$,
		$\mbox{dist}\{\E^{i_j},\,\E^{i_m}\}<\epsilon$.
		On the other hand, by construction, for any $j<i-1$, $\E^{i}$ cannot lie in the ball centering at $\E^j$
		with radius $r$ so that $\mbox{dist}\{\E^i,\,\E^j\}\geq r$ for any $j<i-1$. This is a contradiction.
		
		\smallskip
		\item
		$\E^{k_1}\in\partial\Omega$. Otherwise, $\E^{k_1}\in\Omega$ is an interior
		local minimum point of $\phi_\ee$,
		which contradicts
		the strong maximum principle, since $\phi_\ee$ satisfies equation \eqref{equ:phi e}
		that is strictly elliptic
		in $\Omega$, and
		$\phi_\ee$ is not constant in $\Omega$
		by the assumption that $\varphi$ is not a uniform state.
	\end{enumerate}
	
	Therefore, $\{\E^i\}_{i=0}^{k_1}$ is a minimal chain with $\E^{k_1}\in \partial\Omega$.
	Also, from the construction, $\E^{k_1}$ is a local minimum point of $w$ with respect
	to $\overline\Omega$ with $w(\E^{k_1})<w(E_1)$.
\end{proof}

\begin{lem}\label{lem:chainsDoNotIntersect-MinMax}
	For any $\delta>0$,
	there exists $r_1^*\in(0, r^*]$ such that the following holds{\rm :}
	Let ${\mathcal C}\subset\partial\Omega$
	be connected, let $E_1$ and $E_2$ be the endpoints
	of ${\mathcal C}$, and let there  be
	a minimal chain $\{E^i\}_{i=0}^{k_1}$ of radius $r_1\in (0, r_1^*]$
	which starts at $E_1$ and ends at $E_2$,
	and
    $H_1\in {\mathcal C}^0={\mathcal C}\setminus \{E_1,E_2\}$ such that
	$$
	\phi_\ee(H_1)\ge \phi_\ee(E_1)+\delta.
	$$
	Then, for any $r_2\in (0, r_1]$,
	any maximal chain $\{H^j\}_{j=0}^{k_2}$ of radius $r_2$
	starting from $H_1$
	satisfies $H^{k_2}\in  {\mathcal C}^0$,
	where ${\mathcal C}^0$ denotes the relative interior of curve  ${\mathcal C}$ as before.
\end{lem}

\begin{proof}
	Using the bound: $\|\phi\|_{1+\alpha_1, \overline\Omega}\le C$ by condition
	(A2) of Theorem \ref{thm:main theorem},
	we can find a radius $r_1^*\in (0, r^*]$
	small enough such that
	$$
    \underset{B_{r_1^*}(P)\cap\Omega}{\rm osc}\,
	    \phi_\ee \le \frac\delta 4
	\qquad\,\,\, \mbox{for all $P\in\overline\Omega$}.
	$$
	We fix this $r_1^*$ and
	assume that the minimal chain $\{E^i\}_{i=0}^{k_1}$ from $E_1$ to $E_2$ is of
	radius $r_1\in (0, r_1^*]$.
	
	\smallskip
	Recall that,
	from Definition \ref{def:minimal chain} for the minimal and maximal chains,
	$\phi_\ee(E_1)>\phi_\ee(E^i)$ for $i=1, \dots, k_1$, and
	$\phi_\ee(H_1)<\phi_\ee(H^j)$ for $j=1, \dots, k_2$.
	Then,
	for each $i=0,\dots, k_1$, and $j=0,\dots, k_2$,
	\begin{align*}
	\min_{\overline{B_{r_1}(H^j)\cap\Omega}}\phi_\ee
     &>	\phi_\ee(H^j)-\frac\delta 2
	\ge\phi_\ee(H_1)-\frac\delta 2 \\
	&\ge \phi_\ee(E_1)+\frac\delta 2\ge\phi_\ee(E^i)+\frac\delta 2>
	\max_{\overline{B_{r_1}(E^i)\cap\Omega}}\phi_\ee+\frac\delta 4,
	\end{align*}
	where we have used that $E_1=E^0$, $H_1=H^0$,
	and $0<r_1\le r_1^*$.
	Then
	\begin{equation}\label{A-C-chains-disjoint}
	\overline{B_{r_1}(H^j)\cap\Omega}
	\medcap \overline{B_{r_1}(E^i)\cap\Omega}
	=\emptyset\quad\, \mbox{for each $i=0,\dots, k_1$, and $j=0,\dots, k_2$}.
	\end{equation}
	From this, we have
	$$
	B_{r_1}(H^j)\cap\Omega
	\subset \Omega\setminus \Lambda \quad\,\, \mbox{ for each $j=1, \dots,  k_2$},
	$$
	where $\Lambda:=  \bigcup_{i=0}^{k_1}\overline{B_{r_1}(E^i)\cap\Omega}$.
	
	\smallskip
	Since $\overline{B_{r_1}(H^1)\cap\Omega}$
	is a connected set, then one of connected components of  set
	$\displaystyle \overline\Omega\setminus \Lambda$
	contains $\overline{B_{r_1}(H^1)\cap\Omega}$.
	We denote this component
	by $K_1$.
	Since $\Omega$ is a connected set, then it follows from (\ref{A-C-chains-disjoint}) and
	Lemma \ref{lem:connected-chain}(\ref{lem:connected-chain-i1}) applied to
	chain $\{H^j\}$ that
	$$
	\bigcup_{j=0}^{ k_2}\overline{B_{r_2}(H^j)\cap\Omega} \subset K_1.
	$$
	Thus, $H^{k_2}\in \partial K_1\cap\partial\Omega$.
	It remains to show that $\partial K_1\cap\partial\Omega$ lies within
	${\mathcal C}$.
	
	\smallskip
	Notice that $H_1\in \partial K_1\cap{\mathcal C}$ so that
	$\partial K_1\cap{\mathcal C}\ne \emptyset$.
	Also, $K_1$ is a connected set with $K_1\cap\Lambda=\emptyset$.
	From Lemma \ref{lem:connected-chain}(\ref{lem:connected-chain-i2})
	applied to chain $\{E^i\}$,
	we obtain the existence of a continuous curve $\mS\subset\Lambda$ connecting $E_1$ to $E_2$
	with the properties listed in Lemma \ref{lem:connected-chain}(\ref{lem:connected-chain-i2}).
	Combining
	these properties with Remark \ref{rem:connected-chain-i2}, we see that
	$K_1\subset\overline\Omega_1$, where $\Omega_1$ is the open
	region bounded by curves $\mS$ and ${\mathcal C}$.
	Notice that  $\Omega_1\subset\Omega$.
	Thus, $\partial K_1\cap\partial\Omega$ lies within
	$\partial\Omega_1 \cap\partial\Omega= {\mathcal C}$,
	which implies that $H^{k_2}\in {\mathcal C}$.
	Moreover, the definition of minimal and maximal chains and our assumptions in this lemma imply
	$$
	\phi_\ee(H^{k_2})>\phi_\ee(H^1)>\phi_\ee(E_1)>\phi_\ee(E_2).
	$$
	Thus, $H^{k_2}\in {\mathcal C}^0$.
\end{proof}

\begin{rem}
	In Lemma {\rm \ref{lem:chainsDoNotIntersect-MinMax}}, we have not discussed
	the existence of the maximal chain
	$\{H^j\}_{j=0}^{k_2}$ of radius $r_2$
	starting from $H_1$.
	If $H_1$ is not a local maximum point of $\phi_\ee$ with respect
	to $\overline\Omega$, such an existence follows from
	Lemma {\rm \ref{lem:ExistMinMaxChain}}.
\end{rem}

We also have a version of Lemma \ref{lem:chainsDoNotIntersect-MinMax}
in which the roles of minimal and maximal chains are interchanged:

\begin{lem}\label{lem:chainsDoNotIntersect-MaxMin}
	For any $\delta>0$,
	there exists $r_1^*\in(0, r^*]$ such that the following holds{\rm :}
	Let ${\mathcal C}\subset\partial\Omega$
	be connected, let $E_1$ and $E_2$ be the endpoints of ${\mathcal C}$,
	and let there exist a maximal chain $\{E^i\}_{i=0}^{k_1}$ of radius $r_1\in (0, r_1^*]$
	which starts at $E_1$ and ends at $E_2$,
	and $H_1\in {\mathcal C}^0$ such that
	$$
	\phi_\ee(H_1)\le \phi_\ee(E_1)-\delta.
	$$
	Then, for any $r_2\in (0, r_1]$,
	any minimal chain $\{H^j\}_{j=0}^{k_2}$ of radius $r_2$,
	starting from $H_1$,
	satisfies that $H^{k_2}\in  {\mathcal C}^0$.
\end{lem}

The proof
follows the argument of
Lemma \ref{lem:chainsDoNotIntersect-MinMax} with the changes resulting
from switching
between the minimal and
maximal chains and the correspondingly reversed signs in the inequalities.

\begin{lem}\label{lem:chainsDoNotIntersect-MinMin}
	For any $r_1\in (0, r^*]$, there exists $r_2^*=r_2^*(r_1)\in (0, r^*]$
	such that the following holds{\rm :}
	Let ${\mathcal C}\subset\partial\Omega$
	be connected, let $E_1$ and $E_2$ be the endpoints of ${\mathcal C}$,
	let there exist a minimal chain $\{E^i\}_{i=0}^{k_1}$ of radius $r_1\in (0, r^*]$
	which starts at $E_1$ and ends at $E_2$, and let there exist
	$H_1\in {\mathcal C}^0$ such that
	$$
	\phi_\ee(H_1)< \phi_\ee(E_2).
	$$
	Then, for any $r_2\in(0, r_2^*]$,
	any minimal chain $\{H^j\}_{j=0}^{k_2}$ of radius $r_2$,
	starting from $H_1$,
	satisfies that $H^{k_2}\in  {\mathcal C}^0$.
\end{lem}

\begin{proof}
	As in the proof of Lemma \ref{lem:chainsDoNotIntersect-MinMax}, we need to show
	(\ref{A-C-chains-disjoint}). Set $\delta:=\phi_\ee(E_2)- \phi_\ee(H_1)$. Then $\delta>0$.
	
	Using condition (A2) of Theorem \ref{thm:main theorem},
	we can find a radius $r_2^*\in (0, r^*]$
	small enough such that
    $\underset{B_{r_2^*}(P)\cap\Omega}{\rm osc}\, \phi_\ee \le \frac\delta 4$
	for all $P\in\overline\Omega$.
	We fix this $r_2^*$ and
	assume that the minimal chain $\{H^j\}_{j=0}^{k_2}$ starting at $H_1$ is of
	radius $r_2\in (0, r_2^*]$.
	Then, using properties (\ref{def:minimal chain-Ic})--(\ref{def:minimal chain-Id})
	in Definition \ref{def:minimal chain} for the minimal
	chains, we have
	$$
	\phi_\ee(E_2)=
	\phi_\ee(E^{k_1})=\min_{\overline{B_{r_1}(E^{k_1})\cap\Omega}}\phi_\ee<
	\phi_\ee(E^{k_1-1})=\min_{\overline{B_{r_1}(E^{k_1-1})\cap\Omega}}\phi_\ee<
	\cdots,
	$$
	that is,
	$$
	\phi_\ee(E_2)\le\min_{\overline{B_{r_1}(E^{i})\cap\Omega}}\phi_\ee
	\qquad\,\,\mbox{for $i=0, \dots, k_1$}.
	$$
	Then, for $i=0, \dots, k_1$ and $j=0, \dots, k_2$,
	\begin{equation*}
	\max_{\overline{B_{r_2}(H^j)\cap\Omega}}\phi_\ee\le
	\phi_\ee(H^{j})+\frac \delta 2 \le \phi_\ee(H_1)+\frac \delta 2=
	\phi_\ee(E_2)-\frac {\delta} 2 \le
	\min_{\overline{B_{r_1}(E^i)\cap\Omega}}\phi_\ee-\frac {\delta} 2 .
	\end{equation*}
	This implies (\ref{A-C-chains-disjoint}). Then the rest of the proof of
	Lemma \ref{lem:chainsDoNotIntersect-MinMax} applies without changes.
\end{proof}

\section{\, Proof of Theorem \ref{thm:main theorem}}
\label{sec:proof of the main theorem}
In this section, we first prove Theorem \ref{thm:main theorem}, based on the lemmas obtained in \S 3.

We use the $(S,T)$--coordinates from Lemma \ref{lem:shock graph}
for a unit vector $\ee\in Con$ chosen below so that it suffices
to prove that the graph of $f''_\ee$ is concave:
$$
f''_\ee(T)\leq0 \qquad\; \mbox{for all $T\in(T_A,T_B)$},
$$
and satisfies the strict convexity in the sense of Theorem \ref{thm:main theorem}.

In the following, we denote all the points on $\Gsh$ with respect to $T$;
that is,
for any point $P\in\Gsh$, there exists $T_P$ such that
$P=(f_\ee(T_P),T_P)$ in the $(S,T)$--coordinates.

\smallskip
The proof of Theorem \ref{thm:main theorem} consists of the following four steps,
where the non-strict concavity of $f''_\ee$ is shown in Steps 1--3, while
the strict convexity is shown in Step 4:

\begin{enumerate}[\rm Step 1.]
	\item \label{StesOfProof-step1}
	For any fixed $\ee\in Con$,
	if there exists $\hat P\in \Gsh^0$ with $f_\ee''(T_{\hat P})>0$,
	we prove the existence of a point $C\in \Gsh^0$, depending on $\ee$,
	such that $f_\ee''(T_C)\ge 0$, and $C$ is a
	local minimum point of $\phi_\ee$ along $\Gsh$,
	but $C$ is {\em not} a local minimum point
	of $\phi_\ee$ relative to $\overline\Omega$.
	
	\smallskip
	\item  \label{StesOfProof-step2}
	We fix $\ee\in Con$ to be the vector from condition (A6).
	Then we prove the existence of $C_1\in \Gsh^0$ such that
	there exists a  minimal chain with radius $r_1$ from $C$ to $C_1$.
	
	\smallskip
	\item Let  $\ee\in Con$ be the same as in Step \ref{StesOfProof-step2}.
	We show that the existence of points $C$ and $C_1$ described above yields a contradiction,
	from which we conclude that there is no $\hat P\in \Gsh^0$ with $f_\ee''(T_{\hat P})>0$.
	More precisely, it will be proved by showing the following facts:
	
	\smallskip
	\begin{itemize}
		\item Let $A_2$ be a maximum point of $\phi_\ee$ along $\Gsh$ lying between points $C$ and $C_1$.
		Then $A_2$ is a local maximum point of $\phi_\ee$ relative to $\overline\Omega$,
		and there is no point between $C$ and $C_1$ on $\Gsh$ such that
		the tangent line at this point is parallel to the one at $A_2$.
		
		\smallskip
		\item Between $C$ and $A_2$, or between $C_1$ and $A_2$, there exists a local minimum point $C_2$
		of $\phi_\ee$ along $\Gsh$ such that $C_2\ne C$, or $C_2\ne C_1$, and $C_2$ is not a local minimum point
		of $\phi_\ee$ relative to  domain $\overline\Omega$.
		
		\smallskip
		\item Then, by applying the results on the minimal chains obtained in \S 3.3
		and the facts obtained above in this step, and iterating these arguments,
		we can conclude our contradiction argument.
	\end{itemize}
	
	\smallskip
	\item Fix $\ee\in Con$. We show that, for every $P\in \Gsh^0$,
	either $f_\ee''(T_P)<0$ or there exists an even integer $k>2$ such that
	$f_\ee^{(i)}(T_P)=0$ for all $i=2,\dots,k-1$, and $f_\ee^{(k)}(T_P)<0$.
	This
	proves the strict convexity of the shock.
	We also note that $k$ is independent of the choice of $\ee\in Con$,
	since, by Lemma \ref{lem:shock graph},
	the above property is equivalent to the facts that
	$\partial_{\bt}^{i}\phi(P)=0$
	for all $i=2,\dots,k-1$,
	and
	$\partial_{\bt}^k\phi(P)>0$.
\end{enumerate}

Now we follow these steps to prove Theorem \ref{thm:main theorem} in the rest of this section.

\subsection{\, Step 1: Existence of a local minimum point $C\in\Gsh^0$ along $\Gsh$ in the convex part.}

We choose any $\ee \in Con$ and keep it fixed through Step 1.
Assume that
\begin{equation}\label{non-convexity-Assumption}
\mbox{
	There exists a point $\hat{P}\in\Gsh^0$ such that $f_\ee''(T_{\hat{P}})>0$.
}
\end{equation}
Then, in this step, we prove that there exist points $\hat{A}, \hat{B}, C\in \Gsh^0$ such that
$T_C\in (T_{\hat{A}},T_{\hat{B}})$ with $f_\ee''(T_C)\ge 0$,  $f_\ee''(T)< 0$
for all $T\in(T_{\hat{A}},T_{\hat{B}})$ which are sufficiently close to $T_{\hat{A}}$ and $T_{\hat{B}}$,
and
$$
\phi_\ee(C)=\min_{T\in[T_{\hat{A}},T_{\hat{B}}]}\phi_\ee(f_\ee(T),T).
$$
Moreover, the minimum at $C$ is strict in the sense
that
$$
\phi_\ee(f_\ee(T),T)>\phi_\ee(C)
\qquad\mbox{for all $T\in(T_{\hat{A}},T_{\hat{B}})$ with $f_\ee''(T)< 0$}.
$$

\begin{lem}\label{Properties-Iplus}
	Let
	$$
	I^+:=I^+(\hat P)=(T_{A^+}, T_{B^+})
	$$
	be the maximal interval satisfying
	\begin{itemize}
		\item $I^+\subset(T_A,T_B)${\rm ,}
		
		\smallskip
		\item $T_{\hat P} \in I^+${\rm ,}
		
		\smallskip
		\item $f_\ee''(T_P)\ge 0$ for all $T_P\in I^+${\rm ,}
		
		\smallskip
		\item Maximality{\rm :}\,
		If $(T_{P_1},T_{P_2})\subset(T_A,T_B)$ such that $\hat{P}\in (T_{P_1},T_{P_2})$
		and \\ $f_\ee''(T_P)\ge 0$ for all $T_P\in(T_{P_1},T_{P_2})$,
		then $(T_{P_1},T_{P_2})\subset I^+$.
	\end{itemize}
	Note that such $I^+$ exists and is nonempty because $\hat{P}\in\Gsh^0$ and $f_\ee''(T_{\hat{P}})>0$.
	Then
	\begin{enumerate}[\rm (i)]
		\item \label{Properties-Iplus-ia}
		$T_A<T_{A^+}< T_{B^+}<T_B${\rm ,}
		
		\smallskip
		\item \label{Properties-Iplus-ic}
		$f_\ee'(T_{A^+})<f_\ee'(T_{\hat{P}})<f_\ee'(T_{B^+})$ and $f_\ee'(T_{A^+})\le f_\ee'(T)\le f_\ee'(T_{B^+})$ for all $T\in I^+${\rm ,}
		
		\smallskip
		\item \label{Properties-Iplus-ib}
		There exists an open interval $J^+\subset (T_A,T_B)$ such that $[T_{A^+}, T_{B^+}]\subset J^+$
		and
		\begin{equation}\label{Extend-Iplus-Properties}
		f_\ee''(T)<0, \,\,\, f_\ee'(T_{A^+})\le f_\ee'(T)\le f_\ee'(T_{B^+})
		\qquad\mbox{for all $T\in J^+\setminus\overline{I^+}$},
		\end{equation}
		where $J^+\backslash\overline{I^+}$ is non-empty, since
		$\overline{I^+} \subset J^+$ and $J^+$ is open.
	\end{enumerate}
\end{lem}

\begin{proof}
	Assume that $T_{A^+}=T_A$.
	By the definition of $I^+$,  $f_\ee$ is convex on $I^+$.
	From condition (A4) of Theorem \ref{thm:main theorem}, $f_\ee\in C^2((T_A,T_B))\cap
	C^{1,\alpha}([T_A,T_B])$.
	Combining these facts with $f_\ee''(T_{\hat P})>0$,
	we have
	$$
	f_\ee(T_{\hat P})>f_\ee(T_A)+f'_\ee(T_A)(T_{\hat P} - T_A).
	$$
	By Lemma \ref{lem:shock graph}(\ref{lem:shock graph-i1}),
	this implies that $(A+Con)\cap\Omega\ne \emptyset$, which contradicts (A5).
	Then $T_{A^+}>T_A$. Similarly, $T_{B^+}<T_B$. This proves (\ref{Properties-Iplus-ia}).
	
	\smallskip
	Property (\ref{Properties-Iplus-ic}) follows directly from the definition of $I^+$ and
	the fact that $f_\ee''(T_{\hat P})>0$, by combining with regularity $f_\ee\in C^2((T_A,T_B))$.
	
	\smallskip
	It remains to show (\ref{Properties-Iplus-ib}).
	We first show that
	\begin{equation}\label{Properties-Iplus-ib-eqn}
	\mbox{there exists $T_{\hat A_1}\in [T_A, T_{A^+})$ such that $f_\ee''<0$ on
		$(T_{\hat A_1}, T_{A^+})$,}
	\end{equation}
	where $T_A< T_{A^+} $ by (\ref{Properties-Iplus-ia}).
	If \eqref{Properties-Iplus-ib-eqn} is false, then there exists a sequence
	$\{T^+_i\}\subset (T_A, T_{A^+})$ such that $\lim_{i\to\infty}T^+_i=T_{A^+}$ and
	$f_\ee''(T^+_i)\ge 0$ for all $i$. Also, from the maximality part in the definition of $I^+$,
	there exists a sequence
	$\{T^-_i\}\subset (T_A, T_{A^+})$ such that $\lim_{i\to\infty}T^-_i=T_{A^+}$ and
	$f_\ee''(T^-_i)< 0$ for all $i$. From this, using the regularity of $f_\ee$ in
	Lemma \ref{lem:analyticity},
	it is easy to see that $f_\ee^{(k)}(A^+)=0$ for $k=2, 3, \dots$, which contradicts
	Lemma \ref{noZeroesInfiniteOrderLemma}. This proves \eqref{Properties-Iplus-ib-eqn}.
	
	Moreover, by property (\ref{Properties-Iplus-ic}),
	there exists $T_{\hat A}\in[T_{\hat A_1}, T_{A^+})$ satisfying $f_\ee'(T_{\hat A})\le  f_\ee'(T_{B^+})$.
	Now, since $f_\ee''<0$ on
	$(T_{\hat A_1}, T_{A^+})$, we obtain that
	$f_\ee''(T)<0$ and $f_\ee'(T_{A^+})<f_\ee'(T)\le  f_\ee'(T_{B^+})$
	for all $T\in (T_{\hat A}, T_{A^+})$.
	
	Similarly, we show that there exists $T_{\hat B}\in (T_{B^+}, T_B]$ such that
	$f_\ee''(T)<0$ and  $f_\ee'(T_{A^+})\le f_\ee'(T)<  f_\ee'(T_{B^+})$  for all
	$T\in(T_{B^+}, T_{\hat B})$.
	
	Now (\ref{Properties-Iplus-ib}) is proved with $J^+=(T_{\hat A}, T_{\hat B})$.
\end{proof}

Clearly, the interval, $J^+$, satisfying the properties in
Lemma \ref{Properties-Iplus}(\ref{Properties-Iplus-ib}) is non-unique.
From now on, we choose and fix an interval:
\begin{equation}\label{defAB}
J^+=(T_{\hat A}, T_{\hat B})
\end{equation}
satisfying the properties stated in Lemma \ref{Properties-Iplus}(\ref{Properties-Iplus-ib}).

Now we show the existence of a local minimum point $C\in\overline{I^+}$ along $\Gsh$.
\begin{prop}\label{C-on-convex-Lemma}
	Set
	$$
	w:=\phi_\ee.
	$$
	Then
	\begin{enumerate}[\rm (i)]
		\item \label{C-on-convex-Lemma-i1}
		There exists $T_C\in \overline{I^+}$
		such that
		$$
		w(C)=\min_{[T_{\hat{A}},T_{\hat{B}}]}w(f_\ee(T),T){\rm ;}
		$$
		\item \label{C-on-convex-Lemma-i1-1}
		$C\in\Gsh^0$ with $f''_\ee(T_C)\geq0${\rm ;}
		
		\item \label{C-on-convex-Lemma-i2}
		Furthermore,
		$$w(P)>w(C) \qquad\,\, \mbox{ for all $T_P\in (T_{\hat{A}}, T_{\hat{B}})\setminus[T_{A^+}, T_{B^+}]$}.
		$$
	\end{enumerate}
\end{prop}

\begin{proof}
	Let  $J^+$ be the open interval from \eqref{defAB}, which
	satisfies the properties in Lemma {\rm \ref{Properties-Iplus}(\ref{Properties-Iplus-ib})}.
	Also, recall that $\overline{I^+}=[T_{A+},T_{B^+}]$. Then,  from (\ref{Properties-Iplus-ia}) and
	(\ref{Properties-Iplus-ib}) of Lemma {\rm \ref{Properties-Iplus}}, we obtain that
	$T_{\hat{A}} < T_{A+}<T_{B^+} < T_{\hat{B}}$.
	
	Fix $T_P\in J^+\setminus\overline{I^+}$. Then $f'_{\ee}(T_{A^+})\leq f'_\ee(T_P)\leq f'_\ee(T_{B^+})$
	by Lemma {\rm \ref{Properties-Iplus}(\ref{Properties-Iplus-ib})}.
	Thus, there exists $T_{P_1}\in \overline{I^+}=[T_{A+},\,T_{B^+}]$ such that $f'_\ee(T_{P_1})=f'_\ee(T_P)$.
	In addition,
	since $f_\ee''\ge 0$ in $\overline{I^+}$ by the definition of $I^+$,
	and $f_\ee''< 0$ in
	$J^+\setminus [T_{A+},T_{B^+}]$ by Lemma {\rm \ref{Properties-Iplus}(\ref{Properties-Iplus-ib})},
	then
	\begin{itemize}
		
		\smallskip
		\item If $T_P\in [T_{B^+},T_{\hat{B}}]$,
		$f'_\ee(T)\geq f'_\ee(T_{P_1})$ for all $T\in[T_{P_1},T_{P}]$,
		with strict inequality $f'_\ee(T)> f'_\ee(T_{P_1})$ for $T\in(T_{B^+},T_{P})$,
		
		\smallskip
		\item
		If $T_P\in [T_{\hat{A}},T_{A^+}]$,
		$f'_\ee(T)\leq f_\ee'(T_{P_1})$ for all $T\in[T_P,T_{P_1}]$,
		with strict inequality $f'_\ee(T)< f'_\ee(T_{P_1})$ for $T\in(T_P,T_{A^+})$.
	\end{itemize}
	
	\smallskip
	\noindent
	Thus, defining the function:
	$$
	g(T):=f_\ee(T)-f_\ee(T_{P_1})-f'_\ee(T_{P_1})(T-T_{P_1}),
	$$
	we obtain in the two cases considered above:
	\begin{itemize}
		
		\smallskip
		\item If $T_P\in [T_{B^+},T_{\hat{B}}]$, then
		$$
		g'(T)\begin{cases}
		\ge 0 \quad &\mbox{for $T\in [T_{P_1},T_P]$},\\
		>0 \quad &\mbox{for $T\in (T_{B^+},T_P)$}.
		\end{cases}
		$$
		
		\smallskip
		\item
		If $T_P\in [T_{\hat{A}},T_{A^+}]$, then
		$$
		g'(T)\begin{cases}
		\le 0 \quad &\mbox{for $T\in [T_P,T_{P_1}]$},\\
		<0  \quad &\mbox{for $T\in (T_P, T_{A^+})$}.
		\end{cases}
		$$
	\end{itemize}
	Therefore, in both cases, $g(T_P)>g(T_{P_1})$, which implies
	$$
	f_\ee(T_P)> f_\ee(T_{P_1})+f'_\ee(T_{P_1})(T_P-T_{P_1}).
	$$
	Now, by Lemma \ref{distToSonicCenter-lemma},
	\begin{equation}\label{strictMinOnConvex}
	w(P)>w(P_1).
	\end{equation}
	Thus we have proved that, for any $T_P\in J^+\setminus\overline{I^+}$, there exists $T_{P_1}\in \overline{I^+}$
	such that (\ref{strictMinOnConvex}) holds for $P=(f_\ee(T_{P}), T_{P})$ and
	$P_1=(f_\ee(T_{P_1}), T_{P_1})$.
	This implies that there exists $T_C\in\overline{I^+}$ such that $w(f_\ee(T),T)$ attains its minimum
	over $\overline {J^+}=[T_{\hat{A}},T_{\hat{B}}]$
	at
	$T_C$. This proves assertion (\ref{C-on-convex-Lemma-i1}).
	
	Moreover, we find from $T_C\in \overline{I^+}\subset J^+$ that
	$C\in \Gsh^0$. Also, from (\ref{C-on-convex-Lemma-i1}) and
	$\overline{I^+}\subset J^+= (T_{\hat{A}},T_{\hat{B}})$, $f''_\ee(T_C)\geq0$.
	This proves assertion (\ref{C-on-convex-Lemma-i1-1}).
	
	Assertion (\ref{C-on-convex-Lemma-i2}) for all
    $T_P\in (T_{\hat{A}}, T_{\hat{B}})\setminus[T_{A^+}, T_{B^+}]$ follows from the strict
    inequality in (\ref{strictMinOnConvex}).
\end{proof}

We derive a corollary of Lemma \ref{C-on-convex-Lemma}(\ref{C-on-convex-Lemma-i1-1}).
The property, $C\in\Gsh^0$, guarantees the strict ellipticity of equation \eqref{equ:study} at $C$,
where we have used
assumption (A3) of Theorem \ref{thm:main theorem}.
Combining $f''_\ee(T_C)\geq0$ with Lemma \ref{lem:shock graph}(\ref{lem:shock graph-i3})
implies that $\phi_{\bt\bt}(C)\le 0$.
Thus,
from Lemma \ref{lem:sign of second derivative} and Lemma \ref{lem:shock graph}(\ref{lem:shock graph-i2-00}), we obtain
\begin{cor}\label{cor:CnotLocMin}
	$C$ is not a local minimum point of $\phi_\ee$ with respect to $\overline\Omega$.
\end{cor}
This means that, for any radius $r>0$, there is a point $C_r\in \overline{ B_r(C)\cap\Omega}$ such
that $w(C_r)<w(C)$.

\subsection{\, Step 2: Existence of $T_{C_1}\in (T_A,T_B)\setminus [T_{\hat{A}},T_{\hat{B}}]$
	such that $C_1$ and $C$ are connected by a minimal chain with radius $r_1$, for vector $\ee$ from  condition (A6).}
\label{ProofStep2Sect}

In the argument, we use the minimal and maximal chains
in the sense of  Definition \ref{def:minimal chain}.

Through \S \ref{ProofStep2Sect}--\S \ref{ProofStep3Sect},
we fix $\ee\in Con$ to be the vector from condition (A6) of
Theorem \ref{thm:main theorem}, and use points $\hat{A}, \hat{B}, C\in \Gsh^0$ from Step 1
(which correspond to this vector $\ee$) and constant $r^*$ from Lemma \ref{lem:connectBall}.
In this step, we prove the following proposition:

\begin{prop}\label{Step3-Prop}
	Let $\ee\in Con$ be the vector from condition {\rm (A6)} of Theorem {\rm \ref{thm:main theorem}},
	and let $C$ be the corresponding point obtained in Proposition {\rm \ref{C-on-convex-Lemma}}.
	Then there exists $\hat r_1\in (0, r^*]$ such that, for any $r_1\in (0, \hat r_1)$ and
	any minimal chain $\{C^i\}_{i=0}^{k_1}$ of radius $r_1$
	for $w=\phi_\ee$ starting from point $C$,
	its endpoint $C_1:= C^{k_1}$ is in $\Gsh^0$, i.e.,
	$C_1\in \Gsh^0$.
	Moreover, $C_1$ is a local minimum point of $w$ relative to
	$\overline\Omega$
	such that
	$$
	w(C_1)<w(C).
	$$
\end{prop}

\medskip
In order to prove Proposition \ref{Step3-Prop},
we first notice that, by Corollary \ref{cor:CnotLocMin} and
Lemma \ref{lem:ExistMinMaxChain}, for any $r_1\in(0, r^*)$,
there exists a minimal chain $\{C^i\}_{i=0}^{k_1}$ of radius $r_1$ for $w=\phi_\ee$
in the sense of Definition \ref{def:minimal chain}, starting at $C$, {\it i.e.}, $C^0=C$.
Moreover, $C^{k_1}\in\partial\Omega$ is a local minimum point of $w$ with respect
to $\overline\Omega$, and $w(C^{k_1})<w(C)$.

Now, in order to complete the proof of Proposition \ref{Step3-Prop},
it suffices to prove the following lemma.

\begin{lem}\label{Step3-lemma}
	There exists $\hat r_1>0$ such that,
	if $r_1\in (0, \hat r_1]$, then
	$C^{k_1}\in \Gsh^0$.
\end{lem}

\begin{proof}
	On the contrary,
	if $C^{k_1}\in\overline{\Gamma_1\cup\Gamma_2}$, we derive a
	contradiction for sufficiently small $r_1>0$.
	Now we divide the proof into five steps.
	
	\smallskip
	{\bf 1.}
	We first determine how small $r_1>0$ should be in the minimal chain
	$\{C^i\}$.
	Choose points $A_1, B_1\in\Gsh$ such that
	\begin{equation*}
	\begin{split}
	&T_{A_1}\in[T_A,T_C],\;\quad \;
	\phi_\ee(A_1)=\max_{T\in[T_A,\,T_C]}\phi_\ee(f_\ee(T),T),\\
	&T_{B_1}\in[T_C,T_B],\;\quad \;
	\phi_\ee(B_1)=\max_{T\in[T_C,\,T_B]}\phi_\ee(f_\ee(T),T).
	\end{split}
	\end{equation*}
	Note that the definition of  points
	$A_1$ and $B_1$ is independent of the choice of the minimal chain $\{C^i\}$ and its
	radius. Also, from
	Proposition \ref{C-on-convex-Lemma}(\ref{C-on-convex-Lemma-i2}),
	it follows that $\phi_\ee(A_1)>\phi_\ee(C)$ and  $\phi_\ee(B_1)>\phi_\ee(C)$. Let
	$$
	\delta:=\min\big\{\phi_\ee(A_1)-\phi_\ee(C), \; \phi_\ee(B_1)-\phi_\ee(C)\big\}.
	$$
	Then $\delta>0$. Lemma \ref{lem:chainsDoNotIntersect-MinMax}
	determines $r_1^*(\delta)$, so that $r_1 \in (0, r_1^*(\delta))$ is assumed in the
	minimal chain $\{C^i\}$.

	\smallskip
	{\bf 2.} We start from Case (i) of condition (A6).
	
	\smallskip
	\emph{Claim}: Under the condition of Case (i),
	$A_1$ cannot be a local maximum point of $w=\phi_\ee$ relative to $\overline\Omega$.
	
	In fact, for Case (i),
	if $A_1=A$,
	then $A_1$ cannot be a local maximum point.
	On the other hand, if $A_1\neq A$, and $A_1$ is a local maximum point,
	then
	$$
	f_\ee''(T_{A_1})>0 \qquad \mbox{in the $(S,T)$--coordinates},
	$$
	by Lemmas \ref{lem:sign of second derivative}--\ref{lem:shock graph}.
	Thus, we consider the function:
	$$
	F(T):=f_\ee(T)-f_\ee(T_{A_1})-f_\ee'(T_{A_1})(T-T_{A_1}).
	$$
	Then $F(T_{A_1})=0$, $F'(T_{A_1})=0$, and $F''(T_{A_1})>0$
	so that $F(T)>0$ near $T_{A_1}$.
	Let the maximum of $F(T)$ on $[T_A,T_{A_1}]$ be attained at $T_{A^*}$.
	Then $F(T_{A^*})>0$, which implies that $T_{A^*}\ne T_{A_1}$.
	
	If $T_{A^*}\neq T_{A}$,
	then $F'(T_{A^*})=0$, which implies that $f'_\ee(T_{A^*})-f'_\ee(T_{A_1})=0$.
	If $T_{A^*}=T_A$, then, using $f_\ee'(T_A)\geq f'_\ee(T_{A_1})$,
	condition (A5),
	and $F'(T_A)\leq0$ (since $T_A=T_{A^*}$ is a maximum point of $F(T)$ on $[T_A,T_{A_1}]$),
	we conclude that $f_\ee'(T_{A^*})=f_\ee'(T_A)=f_\ee'(T_{A_1})$.
	Thus, in both cases,
	$$
	f'_\ee(T_{A^*})-f'_\ee(T_{A_1})=0.
	$$
	Also, $F(T_{A^*})>0$ implies
	$$
	f_\ee(T_{A^*})>f_\ee(T_{A_1})+f'_\ee(T_{A_1})(T_{A^*}-T_{A_1}).
	$$
	Then, from Lemma \ref{distToSonicCenter-lemma},
	$\phi_\ee(A^*)>\phi_\ee(A_1)$, which contradicts the definition of $A_1$.
	Now the claim is proved.
	
	\smallskip
	{\bf 3.}
	In this step, for Case (i) of condition (A6),
	we obtain a contradiction to the assumption that
	$C^{k_1}\in\overline{\Gamma_1\cup\Gamma_2}$.
	
	\smallskip
	Since $C^{k_1}\in\overline{\Gamma_1\cup\Gamma_2}$  is a local minimum point
	of $\phi_\ee$, the condition for Case (i)
	implies that $C^{k_1}\in\overline{\Gamma_1}\cup\{B\}$.
	
	\smallskip
	We first consider the case that $C^{k_1}\in\overline{\Gamma_1}\setminus\{A\}$.
	Since $A_1$ is not a local maximum point of $w=\phi_\ee$,
	and $r_1\in (0, r^*]$, then, by
	Lemma \ref{lem:ExistMinMaxChain},
	there exists a maximal chain $\{A^j\}_{j=0}^{k_2}$ of radius $r_1$ for $w$
	in the sense of Definition \ref{def:minimal chain}, starting at $A_1$, {\it i.e.},
	$A^0=A_1$.
	Moreover,
	$A^{k_2}\in\partial\Omega$ is a local maximum point of $w$ with respect to
	$\overline\Omega$,
	and $w(A^{k_2})>w(A_1)$.
	Furthermore, by Lemma \ref{lem:chainsDoNotIntersect-MinMax}
	and the restriction for $r_1$
	described in Step 1,
	it follows that one of the following three cases occurs:
	
	\smallskip
	\begin{enumerate}[\rm (a)]
		\item $A^{k_2}$ lies on
		$\Gamma_1^0$ between $C^{k_1}$ and $A$;
		
		\smallskip
		\item[\rm (b)] $A^{k_2}=A$;
		
		\smallskip
		\item[\rm (c)]
		$A^{k_2}$ lies on $\Gsh^0$ strictly between $A$ and $C$.
	\end{enumerate}
	
	\smallskip
	Since  $A^{k_2}$ is
	a local maximum point of $\phi_\ee$,
	then it cannot lie on $\Gamma_1^0\cup\{A\}$
	by the condition of Case (i).
	Thus, only case (c) can occur,
	{\it i.e.}, $A^{k_2}$ lies on $\Gsh^0$ between $A$ and $C$.
	However, the property that $w(A^{k_2})>w(A_1)$
	contradicts the fact that $T_{A_1}$ is the maximum point
	of $\phi_\ee(f_\ee(T), T)$ on $[T_A,T_C]$.
	Thus, the case that $C^{k_1}\in\overline{\Gamma_1}\setminus\{A\}$
	is not possible.

	Next, consider the case that $C^{k_1}=A$. Then
	$$
	\phi_\ee(C)>\phi_\ee(C^{k_1}) =\phi_\ee(A),
	$$
	so that the definition of $A_1$ implies that $A_1\ne A$.
	Combining with the fact that $A_1\ne C$ proved above, we conclude that
	$A_1$ lies on $\Gsh^0$ strictly between $C$ and $C^{k_1}=A$.
	Then we obtain a contradiction
	by following the same argument as above.
	
	The remaining case, $C^{k_1}=B$, is considered similarly to the case that $C^{k_1}=A$.
	Indeed, in that argument, we have not used the condition that $A$ cannot be a
	local maximum point. Thus, the argument applies to the case that $C^{k_1}=B$, with only notational change:
	points $B$ and $B_1$ are used, instead of  $A$ and $A_1$.
	
	This completes the proof for Case (i) of condition (A6)
	of Theorem \ref{thm:main theorem}.
	
	\smallskip
	{\bf 4.} The proof for Case (ii) of condition (A6) of Theorem \ref{thm:main theorem}
	is similar to Case (i).
	The only difference is to replace both $A$ and $A_1$
	in the argument by $B$ and $B_1$.
	
	\smallskip
	{\bf 5.} Consider Case (iii) of condition (A6) of Theorem \ref{thm:main theorem},
	{\it i.e.}, when $\phi_\ee$ cannot have a local minimum point
	on $\Gamma_1\cup\Gamma_2$.
	For the local minimum point $C^{k_1}\in\overline{\Gamma_1\cup\Gamma_2}$,
	this implies that $C^{k_1}\in\{A, B\}$.
	Then the argument is the same as for the cases: $C^{k_1}=A$ and  $C^{k_1}=B$,
	at the end of Step 3.
	
	Proposition \ref{Step3-Prop} with $C_1=C^{k_1}$ follows directly from Lemma \ref{Step3-lemma}.
\end{proof}

\subsection{\, Step 3: Existence of points $C$ and $C_1$ yields a contradiction}\label{ProofStep3Sect}
In this section, we continue to denote by $\ee\in Con$
the vector from condition (A6) of
Theorem \ref{thm:main theorem}, and use points $\hat{A}, \hat{B}, C\in \Gsh^0$ from Step 1
which correspond to this vector $\ee$.
Then, for each $r_1\in (0, \hat r_1]$,
the corresponding point $C_1$ is defined in Proposition \ref{Step3-Prop}.
In this step, we will arrive at a contradiction to the existence of such $C$ and $C_1$
if $r_1$ is sufficiently small. This implies that \eqref{non-convexity-Assumption} cannot hold
for $\ee$ from condition (A6),
which means that
$f_\ee(\cdot)$ is concave, {\it i.e.},  $\Gsh$ is convex.

For $E_1, E_2\in\Gsh$,  denote by $\Gsh[E_1, E_2]$ the part of $\Gsh$ between points $E_1$ and $E_2$,
including the endpoints.

Fix $r_1\in (0, \hat r_1]$. This choice determines $C_1$.
Let $A_2\in \Gsh[C, C_1]$ be such that
\begin{equation}\label{defA2-eq-0}
\phi_\ee(A_2)=\max_{P\in \Gsh[C, C_1]}\phi_\ee(P).
\end{equation}

\begin{lem}\label{C-A2-separated}
	There exists $\delta>0$
	such that, for any
	$r_1\in (0, \hat r_1]$, the corresponding points $C$, $C_1$, and $A_2$ defined above
	satisfy
	\begin{equation}\label{C-A2-separated-ineq}
	\phi_\ee(A_2)\ge \phi_\ee(C)+\delta>\phi_\ee(C_1)+\delta.
	\end{equation}
\end{lem}

\begin{proof}
	We employ Proposition \ref{C-on-convex-Lemma} for  vector $\ee$ from condition (A6).
	Then, using that $\phi_\ee(C)>\phi_\ee(C_1)$
	by Proposition \ref{Step3-Prop}, it follows from
	Proposition \ref{C-on-convex-Lemma}(\ref{C-on-convex-Lemma-i1})
	that $T_{C_1}\notin [T_{\hat{A}},T_{\hat{B}}]$.
	
	Using this and (\ref{defA2-eq-0}), we conclude that (\ref{C-A2-separated-ineq}) holds
	with
	\begin{equation}\label{defDelta}
	\delta =\min\big\{\max_{P\in\Gsh[\hat{A},C]}\phi_\ee(P),\;
	\max_{P\in\Gsh[\hat{B},C]}\phi_\ee(P)\big\} - \phi_\ee(C),
	\end{equation}
	where
	$\delta>0$ by
	Proposition \ref{C-on-convex-Lemma}(\ref{C-on-convex-Lemma-i2}).
	Notice that the definition of points $\hat{A}$, $\hat{B}$, and  $C$ is independent of
	$r_1$; see \eqref{defAB} and
	Proposition \ref{C-on-convex-Lemma}(\ref{C-on-convex-Lemma-i1}).
	Then the right-hand side of \eqref{defDelta} is independent of $r_1>0$,
	so that $\delta>0$ is independent of $r_1$.
\end{proof}

The rest of the argument in this section involves only part $\Gsh[C, C_1]$ of
the shock curve, independent of the other parts of $\partial\Omega$.
Without loss of generality, we assume that $C_1\in\Gsh[A, C]$ so that
\begin{equation}\label{FixLocation-C-C1}
T_{C_1}\in[T_A,T_C].
\end{equation}
Indeed, if  $C_1\in\Gsh[B, C]$, we re-parameterize the shock curve by
$$
\Gsh=\{(\tilde f_{\ee}(T), T)\;:\;  -T_B\le T\le -T_A\},
$$
where
$\tilde f_{\ee}(T)=f_\ee(-T)$, and $T_A$ and $T_B$ are the
$T$--coordinates of $A$ and $B$ with respect to the original parameterization,
and then switch the notations for points $A$ and $B$.
Thus, (\ref{FixLocation-C-C1}) holds in the new parametrization.

Now (\ref{defA2-eq-0}) has the form:
\begin{equation}\label{defA2-eq}
\phi_\ee(A_2)=\max_{T\in[T_{C_1},T_C]}\phi_\ee(f_\ee(T),\,T).
\end{equation}
In particular, $T_{A_2}\in (T_{C_1}, T_C)$.  See also Fig. 2.

From Lemma \ref{C-A2-separated}
and Proposition \ref{Step3-Prop}, we obtain that, for any $r_1\in(0, \hat r_1]$,
\begin{equation}\label{defA2-ineq}
\phi_\ee(A_2)>\phi_\ee(C)+\delta>\phi_\ee(C_1)+\delta.
\end{equation}

Now we prove
\begin{lem}\label{lem:Case 3 in Step 3}
	If $r_1$ is sufficiently small, then
	\begin{enumerate}[\rm (i)]
		\item\label{lem:Case 3 in Step 3-i1}
		$A_2$ is a local maximum point of $\phi_\ee$ with respect to $\overline\Omega${\rm ;}
		
		\smallskip
		\item\label{lem:Case 3 in Step 3-i2}
		There is no point  $Q\ne A_2$
		between $C$ and $C_1$
		along the shock such that the tangent line
		at $Q$ is parallel to the one at $A_2$.
	\end{enumerate}
\end{lem}

\begin{proof}
	The proof consists of two steps.
	
	\smallskip
	{\bf 1.} In this step, we prove (\ref{lem:Case 3 in Step 3-i1}).
	We first fix $r_1>0$.
	Let $\delta$ be from Lemma \ref{C-A2-separated}, and let
	$r_1^*>0$ be the constant from Lemma \ref{lem:chainsDoNotIntersect-MinMax}
	for this $\delta$.
	We fix $r_1=r_1^*$, and denote $C_1$ and $A_2$ as the corresponding
	points for this choice of $r_1$.
	Suppose that $A_2$ is not a local maximum point of $\phi_\ee$ with respect to $\overline\Omega$.
	Using (\ref{defA2-ineq}) and the existence of a minimal chain of radius $r_1$
	from $C$ to $C_1$,
	we can apply Lemma \ref{lem:chainsDoNotIntersect-MinMax} to obtain
	the existence
	of a maximal chain $\{A^j\}_{j=0}^{k_2}$ of radius $r_1$
	starting from $A_2$ ({\it i.e.},  $A_2=A^0$)
	such that $A^{k_2}$ is on $\Gsh$ between $C$ and $C_1$.
	Since $\phi_\ee(A_2)<\phi_\ee(A^{k_2})$, we obtain a contradiction to
	(\ref{defA2-eq-0}). Thus, $A_2$ is a local maximum point with respect to $\overline\Omega$.

	\begin{figure}[!ht]
		\centering{
		\includegraphics[height=1.35in,width=2.10in]{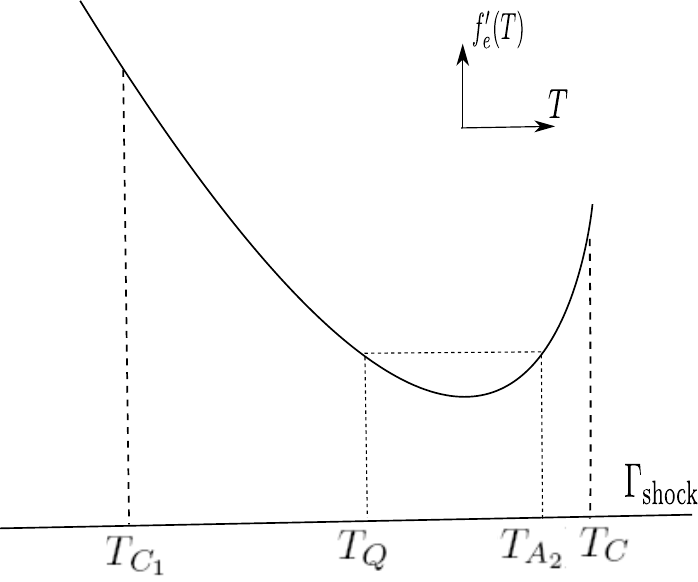}\qquad\quad
		\includegraphics[height=1.35in,width=2.10in]{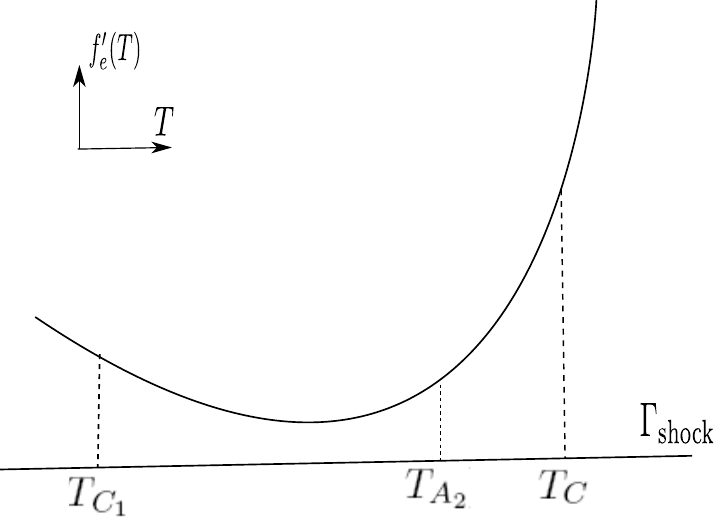}
     \caption{The graphs of function $f_\ee'(T)$}\label{figure:cases in step 3}
     }
	\end{figure}

	{\bf 2.} Now we prove (\ref{lem:Case 3 in Step 3-i2}).
	We use (\ref{FixLocation-C-C1}).
	Assume that there is a point $Q\ne A_2$ between $C$ and $C_1$
	such that the tangent line at $Q$ is parallel to the one at $A_2$.
	Since $A_2$ is a local maximum point of $\phi_\ee$ with respect to $\overline\Omega$
	as shown in Step 1 in this proof,
	we find that $f''_\ee(T_{A_2})>0$,
	by Lemmas \ref{lem:sign of second derivative}--\ref{lem:shock graph}.
	Define
	$$
	F(T):=f_\ee(T)-f_\ee(T_{A_2})-f_\ee'(T_{A_2})(T-T_{A_2}).
	$$
	Then
	\begin{equation}\label{eq:prop-F-E3}
	F(T_{A_2})=F'(T_{A_2})=0,\quad\; F''(T_{A_2})>0,
	\end{equation}
	and there is a point $T_Q\in(T_{C_1}, T_{A_2})\cup (T_{A_2}, T_{C})$
	such that $F'(T_Q)=0$.
	
	If $F(T_Q)>0$, then, by Lemma \ref{distToSonicCenter-lemma},
	we conclude that $\phi_\ee(Q)>\phi_\ee(A_2)$,
	which contradicts (\ref{defA2-eq}).
	
	If $F(T_Q)\le 0$, we first consider the case that $Q\in (T_{C_1}, T_{A_2})$.
	Using ${\displaystyle \max_{T\in [T_Q, T_{A_2}]}F(T)>0}$
	by (\ref{eq:prop-F-E3}) so that this maximum
	is attained at some point $T_{Q_1}\in(T_Q,T_{A_2})$,
	we obtain
	$$
	F(T_{Q_1})>0, \qquad F'(T_{Q_1})=0,
	$$
	so that Lemma \ref{distToSonicCenter-lemma} can be applied to
	obtain that $\phi_\ee(Q_1)>\phi_\ee(A_2)$, which is a contradiction.
	The case that $Q\in (T_{A_2}, T_{C})$ is considered similarly.
	
	Therefore, point $Q$ does not exist.
\end{proof}

\begin{figure}[!ht]
	\centering
	\includegraphics[width=0.5\textwidth]{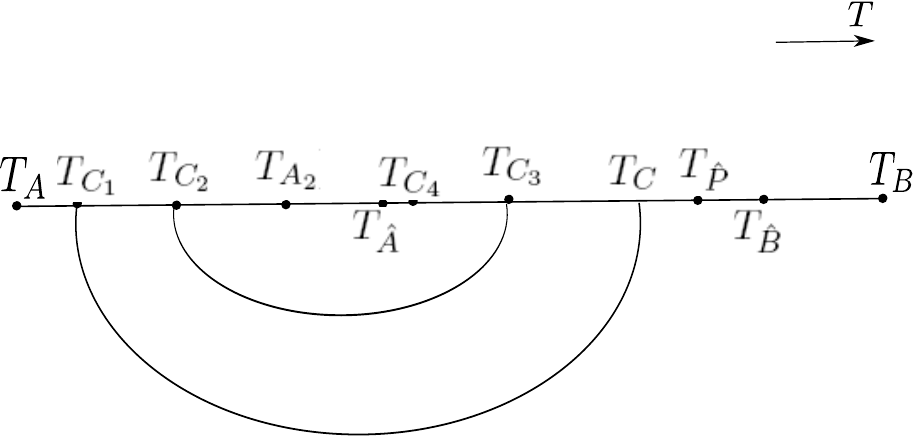}
	\caption{Proof of Step 3 of Theorem \ref{thm:main theorem}}
	\label{figure:proof of main theorem}
\end{figure}

With the facts established in  Lemma \ref{lem:Case 3 in Step 3}, we can conclude the proof of the
main assertion of
Step 3 by a contradiction for sufficiently small $r_1>0$.
The main idea of the remaining
argument  is illustrated in
Fig. \ref{figure:proof of main theorem}.
We first notice the following facts:

\begin{lem}\label{A2-and-tangents}
	$f_\ee(T)$ satisfies the following properties{\rm :}
	\begin{align}\label{2ndDerivFinC1A2}
	&f_\ee''(T_{C_1})<0,\quad\; f_\ee''(T_{A_2})>0,\\
	&f_\ee'(T)\leq f_\ee'(T_{A_2}) \qquad \mbox{for any $T\in [T_{C_1},T_{A_2}]$},
	\label{1stDerivFinC1A2}
	\\
	&f_\ee'(T)\geq f_\ee'(T_{A_2}) \qquad \mbox{for any $T\in [T_{A_2}, T_C]$}.
	\label{1stDerivFinC1A2-pr}
	\end{align}
\end{lem}

\begin{proof}
	Property (\ref{2ndDerivFinC1A2}) follows from
	Lemmas \ref{lem:sign of second derivative}--\ref{lem:shock graph},
	since $A_2$ and $C_1$ are the local maximum and minimum points  of $\phi_\ee$
	with respect to $\overline\Omega$, respectively.
	
	To show (\ref{1stDerivFinC1A2}), we note from $f_\ee''(T_{A_2})>0$ that
	$f_\ee'(T)< f_\ee'(T_{A_2})$ in $(T_{A_2}-\eps,\, T_{A_2})$ for some $\eps>0$.
	Then, if $f_\ee'(T_Q)> f_\ee'(T_{A_2})$ for some $T_Q\in [T_{C_1},T_{A_2})$,
	there exists $T_P\in (T_{Q}, T_{A_2})$ with $f_\ee'(T_P)=f_\ee'(T_{A_2})$,
	which contradicts Lemma \ref{lem:Case 3 in Step 3}(\ref{lem:Case 3 in Step 3-i2}).
	Thus, (\ref{1stDerivFinC1A2}) holds.
	Finally, (\ref{1stDerivFinC1A2-pr}) is proved by similar argument.
\end{proof}

\smallskip
Now we choose $T_{C_2}\in[T_{C_1},T_{A_2}]$ such that
\begin{equation}\label{DefC2-Eq}
\phi_\ee(C_2)=\min_{T\in[T_{C_1},T_{A_2}]}\phi_\ee(f_\ee(T),T).
\end{equation}
We show that
\begin{equation}\label{propertiesOfC2-Eq}
\begin{cases}
\phi_\ee(C_2)<\phi_\ee(C_1),\\
\mbox{$C_2$ is not a local minimum point of $\phi_\ee$ relative
	to domain $\overline\Omega$.}
\end{cases}
\end{equation}
To prove \eqref{propertiesOfC2-Eq}, we first establish
the following more general property of $\Gsh$ (which will also be used in the subsequent development):

\begin{lem}\label{propertiesOfC2}
	Assume that there exist points $E_1$, $E_2$, and $E_3$ on $\Gsh$
	such that
	\begin{enumerate}[\rm (i)]
		\item
		$T_{E_1}<T_{E_2}$ and $T_{E_3}\in[T_{E_1},T_{E_2}]${\rm ,}
		
		\smallskip
		\item\label{propertiesOfC2-i2}
		$f_\ee''(T_{E_1})<0${\rm ,}
		
		\smallskip
		\item\label{propertiesOfC2-i3}
		$f_\ee'(T_{E_1})\le f_\ee'(T_{E_2})${\rm ,}
		
		\smallskip
		\item\label{propertiesOfC2-i4-0}
		$\phi_\ee(E_1)<\phi_\ee(E_2)${\rm ,}
		
		\smallskip
		\item\label{propertiesOfC2-i4}
		$\phi_\ee(E_3)={\displaystyle \min_{T\in[T_{E_1},T_{E_2}]}\phi_\ee(f_\ee(T),T)}$.
	\end{enumerate}
	Then
	$\phi_\ee(E_3)<\phi_\ee(E_1)$,
	and $E_3$ is not a local minimum point of $\phi_\ee$ relative to domain $\overline\Omega$.
	\end {lem}
	
	\begin{proof} We divide the proof into two steps.
		
		\medskip
		{\bf 1}. We first show that $\phi_\ee(E_3)<\phi_\ee(E_1)$. By condition \eqref{propertiesOfC2-i4},
		this is equivalent to the inequality:
		\begin{equation*}
		\phi_\ee(E_1)>\min_{T\in[T_{E_1},T_{E_2}]}\phi_\ee(f_\ee(T),T).
		\end{equation*}
		Thus, it suffices to show that it is impossible that
		\begin{equation}\label{E1-equals-E3}
		\phi_\ee(E_1)=\min_{T\in[T_{E_1},T_{E_2}]}\phi_\ee(f_\ee(T),T).
		\end{equation}
		
		Assume that (\ref{E1-equals-E3}) holds. Consider the function:
		$$
		F(T)=f_\ee(T)-f_\ee(T_{E_1})-f_\ee'(T_{E_1})(T-T_{E_1}).
		$$
		Then $F(T_{E_1})=F'(T_{E_1})=0$, and $F''(T_{E_1})=f_\ee''(T_{E_1})<0$ by condition
		(\ref{propertiesOfC2-i2}).
		This implies that $F(T)<0$ in $(T_{E_1}, T_{E_1}+\delta)$ for some small $\delta>0$.
		Denoting by $T_Q$ a minimum point of $F(T)$ in $[T_{E_1},T_{E_2}]$,
		then $F(T_Q)<0$. This implies that $Q\ne E_1$. Now we consider two cases:
		
		\smallskip
		If $Q\neq E_2$, then $F'(T_Q)=0$, {\it i.e.}, $f_\ee'(T_{Q})=f_\ee'(T_{E_1})$.
		With this, $F(T_Q)<0$ can be rewritten as
		$$
		f_\ee(T_{E_1})> f_\ee(T_Q)+f_\ee'(T_{Q})(T_{E_1}-T_Q).
		$$
		Then, by Lemma \ref{distToSonicCenter-lemma}(\ref{distToSonicCenter-lemma-i2}),
		we obtain that $\phi_\ee(E_1)>\phi_\ee(Q)$, which contradicts
		(\ref{E1-equals-E3}).
		
		If $Q=E_2$, then $F'(T_{E_2})\leq0$.
		Notice that $F'(T_{E_2})=f'_\ee(T_{E_2})-f'_\ee(T_{E_1})\geq0$
		by condition (\ref{propertiesOfC2-i3}).
		Thus, $F'(T_{E_2})=0$, which means that
		$f'_\ee(T_{E_2})=f'_\ee(T_{E_1})$.
		Then, using $F(T_{E_2})=F(T_Q)<0$ and arguing similar to the previous case,
		we employ Lemma \ref{distToSonicCenter-lemma}(\ref{distToSonicCenter-lemma-i2}) to
		obtain that $\phi_\ee(E_1)>\phi_\ee(E_2)$,
		a contradiction to (\ref{E1-equals-E3}).
		
		Therefore, we have proved that (\ref{E1-equals-E3}) is false.
		This implies that $\phi_\ee(E_3)<\phi_\ee(E_1)$, as we have shown above.
		
		\smallskip
		{\bf 2}. We now show that $E_3$ cannot be a local minimum
		point of $\phi_\ee$ relative to domain $\overline\Omega$.
		We have shown in Step 1 that $E_3\ne E_1$. Also, $E_3\ne E_2$
		by conditions (\ref{propertiesOfC2-i4-0})--(\ref{propertiesOfC2-i4}).
		Thus, $T_{E_3}\in(T_{E_1},T_{E_2})$, {\it i.e.}, $E_3\in \Gsh^0$.
		If $E_3$ is a local minimum point of $\phi_\ee$ relative to $\overline\Omega$,
		we obtain by Lemmas \ref{lem:sign of second derivative} and \ref{lem:shock graph}(\ref{lem:shock graph-i3})
		that
		$f''_\ee(T_{E_3})<0$. Let
		$$
		G(T):=f_\ee(T)-f_\ee(T_{E_3})-f'_\ee(T_{E_3})(T-T_{E_3}).
		$$
		Then
		$G(T_{E_3})=G'(T_{E_3})=0$ and $G''(T_{E_3})=f_\ee''(T_{E_3})<0$. This implies that
		$G(T)<0$ in $(T_{E_3}, T_{E_3}+\delta)$ for some $\delta>0$.
		Assume that $T_{Q_1}$ is a minimum point of $G(T)$ in $[T_{E_3},T_{E_2}]$.
		Then, repeating the argument in Step 1 (with $E_3$, $G$, and $T_{Q_1}$ instead of $E_1$, $F$,
		and $T_Q$, respectively),
		we obtain that
		$\phi_\ee(E_3)>\phi_\ee(Q_1)$, which contradicts condition (\ref{propertiesOfC2-i4}).
	\end{proof}
	
	Lemma \ref{propertiesOfC2}
	also holds if $T_{E_1}>T_{E_2}$, with only change in
	the condition that $f_\ee'(T_{E_1})\le f_\ee'(T_{E_2})$ that is now replaced by
    $f_\ee'(T_{E_1})\ge f_\ee'(T_{E_2})$.
	More precisely, we have
	
	\begin{cor}\label{cor:propertiesOfC2}
		Assume that there exist points $E_1$, $E_2$, and $E_3$ on $\Gsh$
		such that
		\begin{enumerate}[\rm (i)]
			\item
			$T_{E_1}>T_{E_2}$ and $T_{E_3}\in[T_{E_2},T_{E_1}]$,
			
			\smallskip
			\item
			$f_\ee''(T_{E_1})<0$,
			
			\smallskip
			\item\label{cor:propertiesOfC2-i3}
			$f_\ee'(T_{E_1})\ge f_\ee'(T_{E_2})$,
			
			\smallskip
			\item\label{cor:propertiesOfC2-i4-0}
			$\phi_\ee(E_1)<\phi_\ee(E_2)$,
			
			\smallskip
			\item\label{cor:propertiesOfC2-i4}
			$\phi_\ee(E_3)=\min_{T\in[T_{E_2},T_{E_1}]}\phi_\ee(f_\ee(T),T)$.
		\end{enumerate}
		Then
		$\phi_\ee(E_3)<\phi_\ee(E_1)$, and $E_3$ is not a local minimum point of $\phi_\ee$
		relative to domain $\overline\Omega$.
	\end{cor}
	
	\begin{proof}
		We prove this by directly repeating the argument in the proof of Lemma \ref{propertiesOfC2}
		with some obvious
		changes. Alternatively, by re-parameterizing the shock curve by
		$$
		\Gsh=\{(\tilde f_{\ee}(T), T)\;:\: -T_B\le T\le -T_A\}
		$$
		so that $\tilde f_{\ee}(T)=f_\ee(-T)$, and $T_A$ and $T_B$ are the
		$T$--coordinates of $A$ and $B$ with respect to the original parameterization,
		then we are under the
		conditions of Lemma \ref{propertiesOfC2} in the new parameterization.
	\end{proof}
	
	\medskip
	\begin{proof}[Proof of \eqref{propertiesOfC2-Eq}]
	Using (\ref{FixLocation-C-C1})--(\ref{defA2-ineq}),
	(\ref{2ndDerivFinC1A2})--(\ref{1stDerivFinC1A2}),  and (\ref{DefC2-Eq}),
	we can
	apply Lemma \ref{propertiesOfC2} with $E_1=C_1$, $E_2=A_2$, and $E_3=C_2$
	to obtain (\ref{propertiesOfC2-Eq}).
	\end{proof}
	
	\medskip
	Let $r_1$ be the constant from Lemma \ref{lem:Case 3 in Step 3},
	and $r_2\in (0, r_1)$.
	Since $C_2$ is not a local minimum point by (\ref{propertiesOfC2-Eq}),
	we use Lemma \ref{lem:ExistMinMaxChain} to obtain the existence
	of  a minimal chain $\{C_2^j\}_{j=0}^{k_2}$ with radius $r_2$;
	see Fig. \ref{figure:proof of main theorem}.
	Next, we restrict $r_2$ to be smaller than $r_2^*$
	from Lemma \ref {lem:chainsDoNotIntersect-MinMin}
	defined by $r_1$ fixed above.
	Then, recalling that there is
	a minimal chain of radius $r_1$ which starts at $C$ and ends at $C_1$,
	and noting that
	$\phi_\ee(C_2)<\phi_\ee(C_1)$
	by (\ref{DefC2-Eq})--(\ref{propertiesOfC2-Eq}),
	we obtain that $C_2^{k_2}$ lies on $\Gsh$
	between $C$ and $C_1$.
	Now, using (\ref{DefC2-Eq}) and
	noting that $\phi_\ee(C_2^{k_2})<\phi_\ee(C_2^0)=\phi_\ee(C_2)$,
	we conclude that $C_2^{k_2}$ lies on the part $[T_{A_2},T_C]$ of $\Gsh$;
	see
	Fig. \ref{figure:proof of main theorem}.
	Denote $C_3:=C_2^{k_2}$ and notice that $C_3$ is a local minimum point
	of $\phi_\ee$ relative to $\overline\Omega$.
	
	From this construction, point $A_2$ (defined by equation (\ref{defA2-eq-0})
	so that (\ref{defA2-eq}) holds)
	satisfies $T_{A_2}\in (T_{C_2}, T_{C_3})\subset(T_{C_1},T_{C})$.
	Then
	$$
	\phi_\ee(A_2)=\max_{T\in[T_{C_1},T_{C}]}\phi_\ee(f_\ee(T),T)
	=\max_{T\in[T_{C_2},T_{C_3}]}\phi_\ee(f_\ee(T),T).
	$$
	Also, from (\ref{defA2-ineq}), (\ref{DefC2-Eq}),
	and the definition of $C_3$ as the endpoint of the minimal chain
	from $C_2$, we have
	$$
	\phi_\ee(A_2)>\phi_\ee(C_2)>\phi_\ee(C_3),\quad f''_\ee(C_3)\le 0,
	$$
	where the last property holds
	by Lemmas \ref{lem:sign of second derivative}--\ref{lem:shock graph},
	since $C_3$ is a local minimum point of $\phi_\ee$
	with respect to $\overline\Omega$.
	Moreover, from (\ref{1stDerivFinC1A2-pr}),
	$$
	f_\ee'(T_{C_3})\geq f_\ee'(T_{A_2}).
	$$
	Choosing $T_{C_4}\in[T_{A_2},T_{C_3}]$ such that
	\begin{equation}\label{DefC4-Eq}
	\phi_\ee(C_4)=\min_{T\in[T_{A_2},T_{C_3}]}\phi_\ee(f_\ee(T),T),
	\end{equation}
	we can apply Corollary \ref{cor:propertiesOfC2}
	with $E_1=C_3$, $E_2=A_2$, and $E_3=C_4$ to
	show that
	$\phi_\ee(C_4)<\phi_\ee(C_3)$
	and $C_4$ cannot
	be a local minimum point.
	
	Then we repeat the same argument as
	those for the minimal chain starting from $C_2$.
	Specifically, for any $r_3\in (0, r_2]$,
	we use Lemma \ref{lem:ExistMinMaxChain} to obtain the existence
	of  a minimal chain $\{C_4^m\}_{m=0}^{k_3}$ with radius $r_3$
	starting from $C_4$, {\it i.e.}, $C_4^0=C_4$;
	see
	Fig. \ref{figure:proof of main theorem}.
	Next, we restrict $r_3$ to be smaller than $r_2^*(r_2)$ from
	Lemma \ref {lem:chainsDoNotIntersect-MinMin}, {\it i.e.}, $r_2$ fixed above
	is used as $r_1$ in Lemma \ref {lem:chainsDoNotIntersect-MinMin}
	to determine $r_2^*(r_2)$.
	Then, recalling that there is
	a minimal chain of radius $r_2$ which starts at  $C_2$ and ends at $C_3$,
	and noting that
	$\phi_\ee(C_4)<\phi_\ee(C_3)$
	as we have shown above,
	we obtain by  Lemma \ref {lem:chainsDoNotIntersect-MinMin} that
	\begin{equation}\label{locPtC4end}
	\mbox{$C_4^{k_3}\,\,$ lies on $\Gsh$ between $C_2$ and $C_3$.}
	\end{equation}
	
	However, combining the properties shown above, we have
	\begin{align*}
	\phi_\ee(C_4)&=\min_{T\in[T_{A_2},T_{C_3}]}\phi_\ee(f_\ee(T),T)<\phi_\ee(C_3)\\
	& <\phi_\ee(C_2)=\min_{T\in[T_{C_1},T_{A_2}]}\phi_\ee(f_\ee(T),T),
	\end{align*}
	so that
	$$
	\phi_\ee(C_4)=\min_{T\in[T_{C_1},T_{C_3}]}\phi_\ee(f_\ee(T),T).
	$$
	Then the property that
	$\phi_\ee(C_4^{k_3})<\phi_\ee(C_4)$ implies that
	$C^{k_3}_4$ cannot lie on $[T_{C_2},T_{C_3}]\subset [T_{C_1},T_{C_3}]$.
	This contradicts (\ref{locPtC4end}).
	
	This contradiction shows that \eqref{non-convexity-Assumption} cannot hold if $\ee$ is the
	vector from condition (A6) of
	Theorem \ref{thm:main theorem}.
	Therefore, in the $(S,T)$--coordinates from Lemma \ref{lem:shock graph} for this vector $\ee$,
    we conclude that
	$$
	f_\ee''(T)\leq0\qquad\mbox{for all $T\in(T_A,T_B)$}.
	$$
	We thus completed the proof of the following fact:
	
	\begin{prop}\label{shockConvex}
		Suppose that conditions {\rm (A1)}--{\rm (A6)} of Theorem {\rm \ref{thm:main theorem}} hold.
		Then the free boundary $\Gsh$ is a convex graph
		as described
		in Theorem {\rm \ref{thm:main theorem}}.
	\end{prop}
	
	\subsection{\, Step 4: Strict convexity of $\Gsh$}
	\label{Step4-sect}
	In this step, we show the strict convexity in the sense that, for  any fixed $\ee\in Con$,
	using the coordinates and function $f_\ee$ from
	Lemma \ref{lem:shock graph}(\ref{lem:shock graph-i1}),
	for every $P\in \Gsh^0$,
	either $f_\ee''(T_P)<0$ or there exists an even integer $k>2$
	such that $f_\ee^{(i)}(T_P)=0$ for all $i=2,\dots, k-1$, and $f_\ee^{(k)}(T_P)<0$.

	Note that $f_\ee''\le 0$ on $(T_A, T_B)$ by Proposition \ref{shockConvex}.

	Let $T_P\in(T_A, T_B)$ be such that $f_\ee''(T_P)=0$.
	By Lemma \ref{noZeroesInfiniteOrderLemma}, there exists an integer $k$ such that
	$$
	f_\ee^{(i)}(T_P)=0 \,\,\,\,\mbox{for $i=2,\,.\,.\,.\,,\,k-1$},
	\qquad\,  f^{(k)}_\ee(T_P)\,\,\,\mbox{is nonzero}.
	$$
	The convexity of the shock in Proposition \ref{shockConvex} implies that $k$ must be even
	and $f^{(k)}_\ee(T_P)<0$. This shows (\ref{strictConvexity-degenerate-graph})
	in the coordinate system with basis $\{\ee, \ee^{\perp}\}$. Moreover, using
	Remark \ref{FB-levelSet-rmk}, we have
	
	\begin{prop}\label{shockConvex-strict}
		Suppose that conditions {\rm (A1)}--{\rm (A6)} of Theorem {\rm \ref{thm:main theorem}} hold.
		Then the free boundary $\Gsh$ is strictly convex in the sense
		that \eqref{strictConvexity-degenerate-graph} holds at every $T\in(T_A, T_B)$ with $f''(T)=0$.
		Moreover,
		\eqref{strictConvexity-degenerate} holds at every point  of $\Gsh^0$,
		at which $\phi_{\bt\bt}=0$.
	\end{prop}
	
	Furthermore, we note the following fact:
	
	\begin{lem}\label{shockConvex-strict-finite-degen}
		Suppose that conditions {\rm (A1)}--{\rm (A6)} of Theorem {\rm \ref{thm:main theorem}} hold.
		Then,  for any  $\eps>0$, there is no more than a finite set of
		points $P=(f(T), T)\in \Gsh$ with $T\in [T_A+\eps, T_B-\eps]$
		such that $f''(T)=0$ {\rm (}or equivalently, $\phi_{\bt\bt}(P)=0${\rm )}.
	\end{lem}
	
	\begin{proof}
		Suppose that $T_i\in [T_A+\eps, T_B-\eps]$ for $i=1,2,\dots$, are such that $f''(T_i)=0$.
		Then a subsequence of $T_i$ converges to $T^*\in [T_A+\eps, T_B-\eps]$, and
		$f^{(n)}(T^*)=0$ for each $n=2,3,\dots$, and $P^*=(f(T^*), T^*)\in\Gsh^0$.
		It follows that $\partial^n_{\bt}\phi (P^*)=0$  for each $n=2,3,\dots$. This
		contradicts (\ref{strictConvexity-degenerate}).
	\end{proof}
	
\smallskip
By Propositions \ref{shockConvex}--\ref{shockConvex-strict} and Lemma \ref{shockConvex-strict-finite-degen},
the proof of  Theorem \ref{thm:main theorem} is completed.
	
	\section{\, Proof of Theorem \ref{thm:main theorem-strictUnif}: Uniform Convexity of Transonic Shocks}
	\label{sec:proof of the main theorem-1}
	In this section, we show the uniform convexity  of $\Gsh^0$ in the sense that $f''(T_P)<0$
	for every $P\in \Gsh^0$ for $f(\cdot)$ in (\ref{shock-graph-inMainThm}), or equivalently,
	$f''_\ee(T)<0$ on $(T_A, T_B)$ for any $\ee\in Con$.

	The outline of the proof is the following:
	By Theorem \ref{thm:main theorem} and Remark \ref{FB-levelSet-rmk},
	$\phi_{\bt\bt}\ge 0$ on $\Gsh^0$.
	Thus, we need to show that $\phi_{\bt\bt}>0$ on $\Gsh^0$.
	Assume that $\phi_{\bt\bt}=0$ at $\Pd\in \Gsh^0$.
	Then we obtain a contradiction by proving
	that
	there exists a unit vector $\ee\in\mathbb{R}^2$ such that $\Pd$ is a
	local minimum point of $\phi_\ee$ along $\Gamma^{0}_{\rm shock}$, but $\Pd$
	is not a local minimum point of $\phi_\ee$ relative to $\overline\Omega$.
	Then we can construct a minimal chain for $\phi_\ee$ connecting $P_{\rm d}$ to $C^{k_1}\in\partial\Omega$.
	We show that
	\begin{itemize}
		\item $C^{k_1}\notin\Gamma_0\cup\Gamma_3$,
		
		\smallskip
		\item $C^{k_1}\notin\Gamma_1\cup\Gamma_2$,
		
		\smallskip
		\item $C^{k_1}\notin\Gsh$.
	\end{itemize}
	This implies that $\phi_{\bt\bt}>0$ on $\Gsh^0$ so that $f''(T)<0$ on $(T_A, T_B)$;
	see Remark \ref{FB-levelSet-rmk}.
	
	Now we follow the procedure outlined above to prove Theorem \ref{thm:main theorem-strictUnif}.
	In the proof, we use the $(S,T)$--coordinates in (\ref{shock-graph-inMainThm}).
	Then we have
	\begin{equation}\label{CoordShock-StrictConv}
	\begin{split}
	&\Gsh=\{S=f(T)\,:\, T_A<T<T_B\}, \quad\, \Omega\subset\{S<f(T) :\,T\in\mr\}, \\
	&\bt=\frac{(f'(T),1)}{\sqrt{(f'(T))^2+1}},\,\,\,
	\bn=\frac{(-1, f'(T))}{\sqrt{(f'(T))^2+1}}, \,\,\,
	f''(T)\le 0
	\qquad \mbox{on $(T_A,T_B)$},
	\end{split}
	\end{equation}
	where we have used the convexity of $\Gsh$ proved in Theorem \ref{thm:main theorem}.
	Note that the orientation of the tangent vector $\bt(P)$ at $P\in\Gsh$ has been chosen
	to be towards endpoint $B$.
	
	\medskip
	First, from the convexity and Lemma \ref{lem:sign of second derivative},
	we have
	
	\begin{lem}\label{lem:cannot be maximum}
		Let $\phi$ be a solution as in Theorem {\rm \ref{thm:main theorem}}.
		For any unit vector $\ee\in\mr^2$, if $\ee\cdot\bn<0$ {\rm (}resp. $\ee\cdot\bn>0${\rm )}
		at $P\in\Gsh^0$,
		then $\phi_\ee$ cannot attain its
		local maximum {\rm (}resp. minimum{\rm )} with respect to $\overline\Omega$ at this point.
	\end{lem}
	
	We now prove the uniform convexity by a contradiction argument.
	From
	Theorem \ref{thm:main theorem}
	and Remark \ref{FB-levelSet-rmk},
	we know that (\ref{strictConvexity-degenerate}) holds so that,
	if $f''(T_{\Pd})=0$
	at some interior point  $\Pd$ of $\Gsh$,  then
	\begin{equation}\label{AtNonstrictConvexPt}
	\begin{split}
    &\phi_{\bt\bt}(\Pd)=0,\\
	&\phi_{\bt\bt}(P)>0\quad\mbox{ for all $P\in\Gsh\cap{\mathcal N}_\eps(\Pd)$ with
		$P\ne \Pd$},
    \end{split}
	\end{equation}
	for some $\eps>0$.
	First we choose a unit vector
	$\ee\in \mathbb{R}^2$ via the following lemma.
	
	\begin{lem}\label{choseVecEE-strictConv}
		There exists a unit vector $\ee\in\mathbb{R}^2$ such that, for any local minimum point $\Pd$
		of $\phi_\ee$ along $\Gamma^{0}_{\rm shock}$, $\ee\cdot\bn(\Pd)<0$.
		In addition,  $\Pd$ is a strict local minimum point  along $\Gamma^{0}_{\rm shock}$ in the following sense{\rm :}
		For the unit tangent vector $\bt=\bt(P)$ to $\Gsh$ at $P$ defined by \eqref{CoordShock-StrictConv},
		$\phi_{\ee\bt}$
		is strictly positive on $\Gsh$ near $\Pd$ in the direction of $\bt$, and
		$\phi_{\ee\bt}$ is
		strictly  negative on $\Gsh$ near $\Pd$ in the direction opposite to $\bt$.
		More precisely, in the coordinates from \eqref{CoordShock-StrictConv}, there exists $\eps>0$ such that
		$T_A<T_{\Pd}-\eps< T_{\Pd}+\eps<T_B$ and
		\begin{equation}\label{choseVecEE-strictConv-Eq}
		\begin{split}
        &\phi_{\ee\bt}(f(T), T)<0\quad\,\,\,\mbox{on $(T_{\Pd}-\varepsilon,\, T_{\Pd})$},\\
		&\phi_{\ee\bt}(f(T), T)>0\quad\,\,\,\mbox{on $(T_{\Pd},\, T_{\Pd}+\varepsilon)$}.
		\end{split}
       \end{equation}
	\end{lem}
	
	\begin{proof}
		Recall that $\phi_{\bt\bt}(\Pd)=0$. Now
		we first use \eqref{DerivRH-nu-tau} at $\Pd$ with $h_{\bn}\ne 0$ by \eqref{equ:27 oblique of h},
		and then use the strictly elliptic equation \eqref{equ:in ST coordinate}
		at $\Pd$ in the $(S,T)$--coordinates with basis $\{\bn(\Pd), \bt(\Pd)\}$ to obtain
		\begin{equation}\label{equ:Sec Derivatives are zero}
		\phi_{\bn\bn}(\Pd)=\phi_{\bn\bt}(\Pd)=\phi_{\bt\bt}(\Pd)=0.
		\end{equation}

		For any unit  vector $\ee\in\mathbb{R}^2$, define a function
		$g(\cdot)\equiv g(\ee)(\cdot)$ on $\Gsh^0$ by
		\begin{equation}\label{functionGTauE}
		g(\ee)(\xxi):=\big(\rho(c^2-\varphi_{\bn}^2)\varphi_{\bn}(\ee\cdot\bt)+
		(\rho\varphi^2_{\bn}+\rho_0c^2)\varphi_{\bt}(\ee\cdot\bn)\big)(\xxi).
		\end{equation}
		Then, at any point of $\Gsh^0$, we see from  \eqref{DerivRH-nu-tau}
		with \eqref{Expression-h} that, for any unit vector $\ee\in\mathbb{R}^2$,
		\begin{equation}\label{eq:phietau}
		\phi_{\ee\bt}=\phi_{\bt\bt}(\ee\cdot\bt)+\phi_{\bt\bn}(\ee\cdot\bn)
		=\frac{\phi_{\bt\bt}g(\ee)}{\rho(c^2-\varphi^2_{\bn})\varphi_{\bn}}.
		\end{equation}
		
		Notice that, from the expression of $g(\ee)(\cdot)$  and assumption (A3) of Theorem \ref{thm:main theorem},
		\begin{equation}\label{sign:gtt}
		g(\bt)>0,\quad  g(-\bt)<0 \,\, \qquad\text{ on $\Gsh^0$}.
		\end{equation}
		Then we can choose a unit vector $\ee$ such that $\ee\cdot\bn<0$ and $g(\ee)=0$ at $\Pd$.
		We fix this vector $\ee$ for the rest of this proof.
		From \eqref{equ:Sec Derivatives are zero}, we have
		\begin{equation}\label{equ:Sec Derivatives are zero-varphi}
		\varphi_{\bt\bt}=\varphi_{\bn\bn}=-1,\quad\varphi_{\bn\bt}=0
		\qquad\,\,\mbox{at $\Pd$}.
		\end{equation}
		Below we use the $(S,T)$--coordinates from \eqref{CoordShock-StrictConv}.
		From (\ref{equ:boundary}) and (\ref{CoordShock-StrictConv}),
		we use  condition (A1) in Theorem \ref{thm:main theorem} to obtain that
		$\phi_S>0$ on $\Gsh$ so that
		$$
		\bt=\frac{(-\phi_{T},\phi_{S})}{|D\phi|},\quad\bn=-\frac{D\phi}{|D\phi|}.
		$$
		Then we can use these expressions to define $\bt$ and $\bn$ in $\Omega$ near $\Gsh$, which
		allows to extend function $g(\ee)(\cdot)$ defined by expression (\ref{functionGTauE}) into this region.
		Since $\phi\in  C^{2}(\Omega\cup\Gsh^0)$,  the extended
		$\bt$, $\bn$, and $g(\ee)(\cdot)$ are $C^1$ up to $\Gsh^0$.
		Then, from \eqref{equ:Sec Derivatives are zero}, $D_{(S,T)}\bt=0$ and $D_{(S,T)}\bn=0$ at
		point $\Pd$. Moreover, differentiating (\ref{1.2}) and
		(\ref{c-through-density-function}), and using \eqref{equ:Sec Derivatives are zero}
		yield that $D_{(S,T)}\rho=0$ and $D_{(S,T)}c^2=0$ at point $\Pd$.
		Therefore, differentiating \eqref{functionGTauE},
		using \eqref{equ:Sec Derivatives are zero-varphi},
		and writing $g(\cdot)$ for $g(\ee)(\cdot)$, we have
		$$
		g_{\bt}(\Pd)
		=-(\ee\cdot\bn)(\rho\varphi_{\bn}^2+\rho_0c^2)\Big|_{\Pd}>0.
		$$
		Then, by \eqref{CoordShock-StrictConv},
		$$
           \frac{\md g(f(T),T)}{\md T}\Big|_{T=T_{\Pd}}=\sqrt{(f'(T_{\Pd}))^2+1}\,g_{\bt}(\Pd)>0.
        $$
		Thus,
        $g(f(T),T)<0$ on $(T_{\Pd}-\varepsilon,\, T_{\Pd})$ and
		$g(f(T),T)>0$ on $(T_{\Pd},\, T_{\Pd}+\varepsilon)$ for some $\varepsilon>0$.
		By (\ref{AtNonstrictConvexPt}) and (\ref{eq:phietau}), the same is true for $\phi_{\ee\bt}$.
		
		Then $\Pd$ is a local minimum point of $\phi_\ee$ along $\Gamma_{\rm shock}$,
		and $\phi_{\ee\bt}$ has the properties asserted.
	\end{proof}
	
	\begin{rem}
		The unit vector $\ee$ is not necessarily in the cone introduced in
		condition {\rm (A5)} of Theorem {\rm \ref{thm:main theorem}}.
	\end{rem}
	
	\begin{lem}\label{PdNotLocMin}
		$\Pd$ is not a local minimum point of $\phi_\ee$ with respect to $\overline\Omega$.
	\end{lem}
	
	\begin{proof}
		If $\Pd$ is a local minimum point, it
		follows from Lemma \ref{lem:sign of second derivative} and
		$\ee\cdot\bn(\Pd)<0$
		that $\phi_{\bt\bt}(\Pd)>0$, which contradicts to the fact that
		$\phi_{\bt\bt}(\Pd)=0$.
	\end{proof}
	
	Now
	we consider a minimal chain
	starting at $\Pd$. In the following argument, we use the $(S,T)$--coordinates
	in (\ref{CoordShock-StrictConv}).
	
	To choose the radius for this chain, we note the following:
	
	\begin{lem}\label{lem:IntAroundPd}
		There exist points $\Pd^\pm\in \Gsh^0$ such that
		\begin{enumerate}[\rm (i)]
			\item\label{lem:IntAroundPd-i1}
			$\Pd$ lies on $\Gsh$ strictly between $\Pd^+$ and $\Pd^-${\rm :}
			$$
			T_A<T_{\Pd^-}<T_{\Pd}<T_{\Pd^+}<T_B;
			$$
			
			\item\label{lem:IntAroundPd-i2}
			Denoting by $\Gsh[P, Q]$ the
			segment of $\Gsh$ with endpoints
			$P$ and $Q$, then
			\begin{equation}\label{minInterv-Pd-strictConv}
			\begin{split}
			&\phi_\ee(\Pd)<\phi_\ee(P)<\phi_\ee(\Pd^-)
			\quad \;\mbox{if $P\in\Gsh[\Pd^-, \Pd]\setminus\{ \Pd^-, \Pd\}$},\\
			&\phi_\ee(\Pd)<\phi_\ee(P)<\phi_\ee(\Pd^+)
			\quad \;\mbox{if $P\in\Gsh[\Pd^+, \Pd]\setminus\{\Pd^+, \Pd\}$};
			\end{split}
			\end{equation}
			
			\item\label{lem:IntAroundPd-i3}
			$\ee\cdot \bn(P) <0$ for all $P\in \Gsh[\Pd^-, \Pd^+]$.
		\end{enumerate}
	\end{lem}
	
	\medskip
	\begin{proof}
		Recall the definition of $\bt$ in
		(\ref{CoordShock-StrictConv}). Then
		we use (\ref{choseVecEE-strictConv-Eq}) in Lemma \ref{choseVecEE-strictConv} to find that,
		for $\varepsilon>0$ defined there,
		\begin{align*}
		\frac{\md\phi_\ee(f(T),T)}{\md T}
        \begin{cases} < 0\quad\;
		\mbox{ if $T\in (T_{\Pd}-\eps, \; T_{\Pd})$},\\[1mm]
        > 0\;
		\quad \mbox{ if $T\in (T_{\Pd}, \; T_{\Pd}+\eps)$}.
        \end{cases}
		\end{align*}
		Thus, for points $\Pd^\pm:=(f(T_{\Pd}\pm \eps), \,T_{\Pd}\pm\eps)$,
		assertions (\ref{lem:IntAroundPd-i1})--(\ref{lem:IntAroundPd-i2}) hold.
		Furthermore, since $\ee\cdot \bn<0$ at $\Pd$,
		then, reducing $\eps$ if necessary, we obtain property (\ref{lem:IntAroundPd-i3}).
	\end{proof}
	
	Denote
	\begin{equation}\label{deltaFor-minInterv-Pd-strictConv}
	\delta =\min\big\{ \max_{P\in\Gsh[\Pd^-,\,\Pd]}\phi_\ee(P),\;
	\max_{P\in\Gsh[\Pd,\,\Pd^+]}\phi_\ee(P)\big\} - \phi_\ee(\Pd).
	\end{equation}
	Note that $\delta>0$ by (\ref{minInterv-Pd-strictConv}).
	Now let $r_1$ be constant $r_1^*$ from Lemma \ref{lem:chainsDoNotIntersect-MinMax}
	determined by $\delta$ from (\ref{deltaFor-minInterv-Pd-strictConv}).
	
	By Lemmas \ref{lem:ExistMinMaxChain} and \ref{PdNotLocMin},
	there exists a minimal chain with radius $r_1$
	which starts at $\Pd$. Denote its endpoint by $C^k$. Then
	\begin{equation}\label{Ck-on-bdry-strictConv}
	C^{k}\in\partial\Omega,
	\end{equation}
	and $C^k$ is a local minimum point of $\phi_\ee$ relative to $\overline\Omega$. Moreover,
	\begin{equation}\label{Ck-Pd-strictConv}
	\phi_\ee(\Pd)>\phi_\ee(C^k).
	\end{equation}
	
	Now we consider case by case all parts of the decomposition:
	$$
    \displaystyle \partial \Omega=\Gsh\bigcup\big(\bigcup_{i=0}^3 \hat{\Gamma}_i \big)
    $$
	defined in Framework (A)(iii) and assumption
	(A7) of Theorem \ref{thm:main theorem-strictUnif},
	and show that $C^k$ cannot lie on the corresponding
	part.
	Eventually, we reach a contradiction by showing
	that $C^k$ cannot lie anywhere on $\partial\Omega$.
	
	In the proof below, we note the following:
	
	\begin{rem}\label{symmetry-H9-rmk}
		We use condition {\rm (A10)}
		of Theorem {\rm \ref{thm:main theorem-strictUnif}} only in the proof of
		Lemma {\rm \ref{NotInGam2-2-cornerPt-Lem}}.
		The other conditions of Theorem {\rm \ref{thm:main theorem-strictUnif}} to be used in the proof
		below include Framework {\rm (A)},
		conditions {\rm (A1)}--{\rm (A6)} of
		Theorem {\rm \ref{thm:main theorem}},
		and {\rm (A7)}--{\rm (A9)} of Theorem {\rm \ref{thm:main theorem-strictUnif}}.
		These conditions are symmetric for
		$\hat \Gamma_0$ and $\hat \Gamma_3$,
		for $\hat\Gamma_1$  and $\hat \Gamma_2$, and for points $A$ and $B$. Also,
		$\delta$ in
		\eqref{deltaFor-minInterv-Pd-strictConv}
		is defined in a symmetric way with respect to the change of direction of $T$
		in \eqref{CoordShock-StrictConv}.
		This allows without loss of generality to make a
		particular choice between points $A$ and $B$, and the corresponding boundary segments
		in order to fix the notations, as detailed in several places below.
	\end{rem}

	\smallskip
	Now we consider all the cases
	for the location of $C^k$ on $\partial\Omega$.
	
	\begin{lem}\label{lem:strict convexity 16}
		$C^k\notin\overline{\hat\Gamma_0}\cup\overline{\hat\Gamma_3}$.
	\end{lem}
	
	\begin{proof}
		On the contrary, if $C^k\in\overline{\hat\Gamma_0}\cup\overline{\hat\Gamma_3}$,
		we now show in the next four steps that it leads to a contradiction.
		
		\smallskip
		{\bf 1}. We first fix the notations.
		In this proof, we do not use condition (A10)
		of Theorem \ref{thm:main theorem-strictUnif}.
		Thus, as discussed
		in Remark \ref{symmetry-H9-rmk},  we can assume
		without loss of generality that $C^k\in\overline{\hat\Gamma_3}$ and
		$B=\Gsh\cap\overline{\hat\Gamma_3}$.
		
		From \eqref{Ck-Pd-strictConv} and condition (A8) of
		Theorem \ref{thm:main theorem-strictUnif},
		\begin{equation}\label{lem:strict convexity 16-1}
		\phi_\ee(\Pd)>\phi_\ee(C^k)=\phi_\ee(B).
		\end{equation}
		
		We now prove
		Lemma \ref{lem:strict convexity 16} by showing the two claims below: Claims 5.7.1--5.7.2.
		
		\medskip
		{\bf 2}. {\bf Claim 5.7.1.} It is impossible that $\ee\cdot \bn(B)\leq0$ at $B$;
		see Fig. \ref{figure:proof of claim-fst}
		for the illustration of the argument below.
		
		\begin{figure}[!ht]
			\centering
			\includegraphics[width=0.5\textwidth]{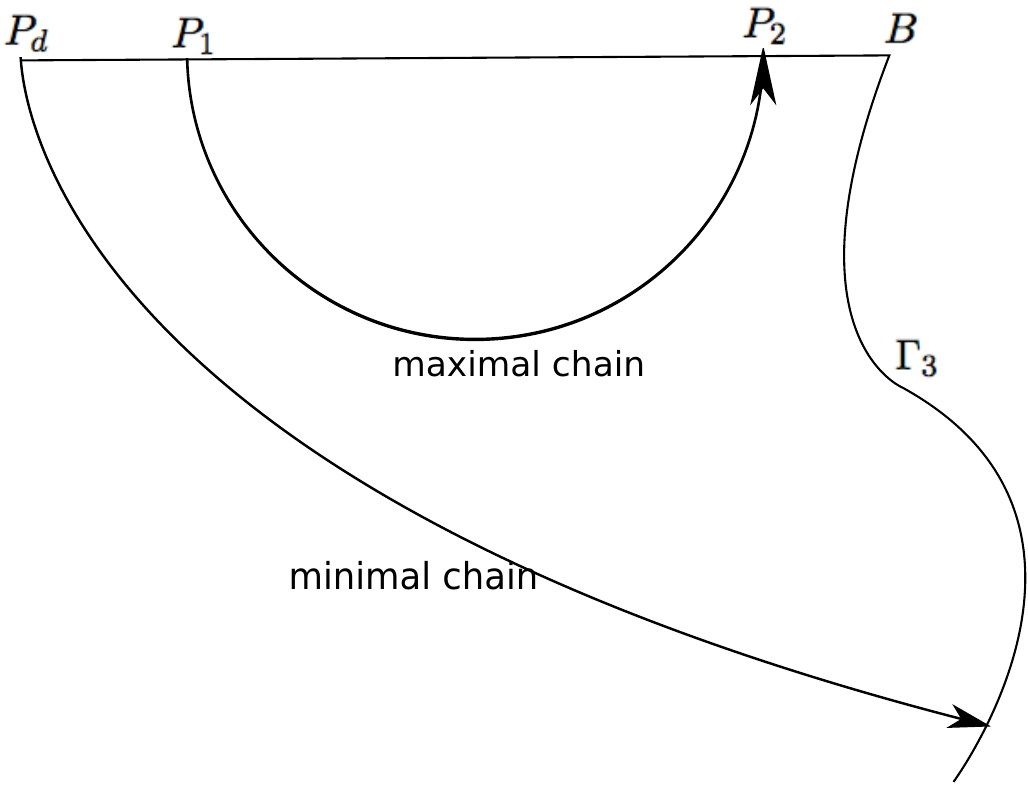}
			\caption{Proof of Claim {\rm 5.7.1}}
			\label{figure:proof of claim-fst}
		\end{figure}
		
		We first show that, if $\ee\cdot \bn(B) \leq0$, then, since  $\ee\cdot \bn(\Pd)<0$, the
		strict convexity of $\Gsh$ (as in Lemma \ref{shockConvex-strict-finite-degen})
		and the graph structure (\ref{CoordShock-StrictConv})
		imply that
		$\bn\cdot \ee<0$ at any point lying strictly between
		$\Pd$ and $B$ along $\Gsh$.
		Indeed, using (\ref{CoordShock-StrictConv})
		and writing $\ee=(e_1, e_2)$ in the $(S,T)$--coordinates, we have
		\begin{equation}\label{NdotE-expression}
		\bn(P)\cdot \ee=\frac{f'(T)e_2-e_1}{\sqrt{(f'(T))^2+1}}\qquad\,\,
		\mbox{ for $P=(f(T), T)$}.
		\end{equation}
		Thus,
		$$
		f'(T_{\Pd})e_2-e_1<0,\qquad f'(T_{B})e_2-e_1\le 0.
		$$
		Using $f''(T)\le 0$ and
		Lemma \ref{shockConvex-strict-finite-degen}, we have
		$$
		f'(T_{\Pd})<f'(T)<f'(T_B) \qquad\,\, \mbox{for all $T\in(T_{\Pd}, T_B)$}.
		$$
		Then
		it follows that
		$$
		f'(T)e_2-e_1<0 \qquad\,\,  \mbox{if $T\in [T_{\Pd}, T_B)$}.
		$$
		Therefore, we have
		\begin{equation}\label{NdotE-neg-inside}
		\bn(f(T), T)\cdot \ee<0 \qquad\; \mbox{ for all $T\in [T_{\Pd}, T_B)$}.
		\end{equation}
		
		\medskip
		Now we show that (\ref{NdotE-neg-inside}) leads to a contradiction.
		Let $P_1\in \Gsh[\Pd,\,B]$ be such that
		\begin{equation}\label{Claim571ineq1}
		\phi_\ee(P_1)=\max_{P\in\Gsh[\Pd,\,B]}\phi_\ee(P).
		\end{equation}
		Since $\Gsh[\Pd,\,\Pd^+]\subset \Gsh[\Pd,\,B]$ by
		Lemma \ref{lem:IntAroundPd}(\ref{lem:IntAroundPd-i1}),
		we obtain from (\ref{deltaFor-minInterv-Pd-strictConv}) that
		\begin{equation}\label{P1-Pd-delta}
		\phi_\ee(P_1)\ge \phi_\ee(\Pd)+\delta,
		\end{equation}
		so that $P_1\ne \Pd$.
		Also, by (\ref{lem:strict convexity 16-1}) and (\ref{P1-Pd-delta}),
		we see that
		$P_1\ne B$. Thus, $\bn(P_1)\cdot\ee<0$ by (\ref{NdotE-neg-inside}).
		Now, by Lemma \ref{lem:cannot be maximum},
		$P_1$ cannot be a local maximum point of $\phi_\ee$ relative to
		$\overline\Omega$.
		Therefore, by Lemma \ref{lem:ExistMinMaxChain}, there
		exists a maximal chain of radius $r_1$, starting from $P_1$ and ending at some point
		$P_2\in\partial\Omega$ which is a local maximum point relative to
		$\overline\Omega$, and $\phi_\ee(P_1)< \phi_\ee(P_2)$.
		
		Next, we show that
		\begin{equation}\label{P2-location-strictConv}
		\mbox{$P_2$ $\,\,$ lies
			on $\Gsh^0$ strictly between $\Pd$ and $B$.}
		\end{equation}
		Indeed, recall that there exists a minimal chain of radius $r_1$
		from $\Pd$ to $C^k\in\hat\Gamma_3$.
		Also, $P_1$ lies on $\Gsh^0$ strictly between $\Pd$ and $B$.
		Then, from (\ref{P1-Pd-delta}) and
		the choice of $r_1$ (see the lines after (\ref{deltaFor-minInterv-Pd-strictConv})),
		we obtain from
		Lemma \ref{lem:chainsDoNotIntersect-MinMax} that
		either (\ref{P2-location-strictConv}) holds or
		$P_2$
		lies on $\overline{\hat\Gamma_3}$ between $B$ and $C^k$ (possibly including $B$).
		However, we use
		condition (A8) of
		Theorem \ref{thm:main theorem-strictUnif},
		(\ref{lem:strict convexity 16-1}), and (\ref{P1-Pd-delta}) to obtain that,
		for any $P\in \overline{\hat\Gamma_3}$,
		$$
		\phi_\ee(P)=\phi_\ee(B)<\phi_\ee(\Pd)<\phi_\ee(P_1)<\phi_\ee(P_2),
		$$
		which implies that $P_2\ne P$.  This proves  (\ref{P2-location-strictConv}).
		
		However, (\ref{P2-location-strictConv}) contradicts
		(\ref{Claim571ineq1}) since $\phi_\ee(P_1)< \phi_\ee(P_2)$. Now Claim 5.7.1
		is proved.
		
		\medskip
		{\bf 3}. {\bf Claim 5.7.2.}
		It is impossible that $\ee\cdot\bn(B)>0$;
		see Figs. \ref{figure:proof of claim-scnd-init}--\ref{figure:proof of claim-scnd-k}
		for the illustration
		of the argument below.
		
		\smallskip
		If  $\ee\cdot\bn(B)>0$, then, using $\ee\cdot\bn(\Pd)<0$, there exists a point
		$P_0\in\Gsh[\Pd, B]$  so that  $\ee\cdot\bn(P_0)=0$.
		
		Then, from (\ref{NdotE-expression}),
		$$
		-e_1+ f'(T)e_2=0 \qquad\,\, \mbox{at $T=T_{P_0}$}.
		$$
		Now, since $f''(T)\le 0$ by the convexity of $\Gsh$,
		we use Lemma \ref{shockConvex-strict-finite-degen} to
		find that the function: $T\to -e_1+ f'(T)e_2$ is strictly monotone on $(T_A, T_B)$, which implies that
		point $P_0$ is unique.
		
		\begin{figure}[!ht]
			\centering
			\includegraphics[width=0.5\textwidth]{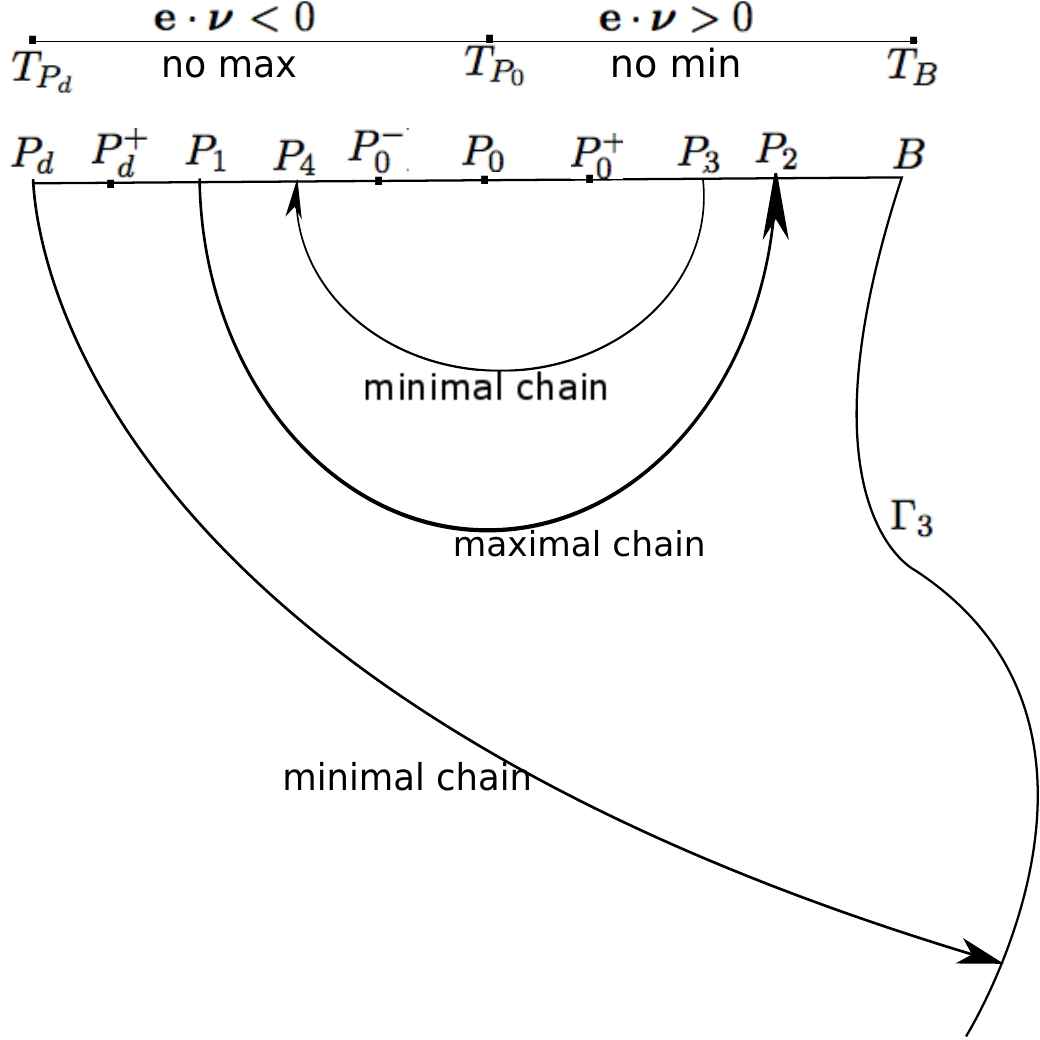}
			\caption{Proof of Claim {\rm 5.7.2}:
				The initial step of the iteration procedure}
			\label{figure:proof of claim-scnd-init}
		\end{figure}
		
		Recall that $\ee\cdot\bn(\Pd)<0$ and $\ee\cdot\bn(P_0)= 0$. Then,
		following
		the proof of (\ref{NdotE-neg-inside}), we have
		\begin{equation}\label{NdotE-neg-inside-P0}
		\bn(f(T), T)\cdot \ee<0 \qquad\; \mbox{ for all $T\in [T_{\Pd}, T_{P_0})$}.
		\end{equation}
		Similarly, using $\ee\cdot\bn(P_0)= 0$ and $\ee\cdot\bn(B) > 0$,
		and
		arguing similar to
		the proof of (\ref{NdotE-neg-inside}), we have
		\begin{equation}\label{NdotE-pos-P0B}
		\bn(f(T), T)\cdot \ee>0 \qquad\; \mbox{ for all $T\in (T_{P_0}, T_B]$}.
		\end{equation}
		From (\ref{NdotE-neg-inside-P0})--(\ref{NdotE-pos-P0B})
		and Lemma \ref{lem:cannot be maximum}, we conclude that
		\begin{equation}\label{NdotE-pos-neg-MaxMin}
		\begin{split}
		&\mbox{If $P\in\partial\Omega$ is a
			local maximum (resp. minimum) point of $\phi_\ee$ relative}\\
		&\mbox{to $\overline\Omega$, then $P\notin(\Gsh[\Pd, P_0])^0$
          (resp. $P\notin (\Gsh[P_0, B])^0$)}.
		\end{split}
		\end{equation}
		
		\smallskip
		Next, since $\ee\cdot\bn(P_0)=0$, then $\ee=\pm \bt(P_0)$.
		Moreover, by (\ref{CoordShock-StrictConv}),
		we have
		$$
		\bn(\Pd)\cdot \bt(P_0)=
		\frac{f'(T_{\Pd})-f'(T_{P_0})}
		{\sqrt{\big((f'(T_{\Pd}))^2+1\big)\big((f'(T_{P_0}))^2+1\big)}}\ge 0,
		$$
		because $f''(T)\le 0$ and $T_{\Pd}\le T_{P_0} \le T_B$.
		Then, since $\bn(\Pd)\cdot \ee<0$, we conclude
		\begin{equation}\label{EequalsMinTau}
		\ee=-\bt(P_0).
		\end{equation}
		With this, recalling that $\phi_{\bt\bt}\ge 0$
		on $\Gsh$, we use  \eqref{eq:phietau}--\eqref{sign:gtt}
		and Lemma
		\ref{shockConvex-strict-finite-degen} to obtain the existence of two points $P_0^-$ and $P_0^+$ such that
		$P_0^\pm=(f(T_{P_0^\pm}), \,T_{P_0^\pm})\in\Gsh([\Pd, B])^0$ and
		\begin{align}
		\label{ine:P0-coord}
		&T_{\Pd}<T_{P_0^-}<T_{P_0}<T_{P_0^+}<T_B, \\
		\label{ine:P0}
		&\phi_{\ee\bt}(P)<0 \quad \;\;\mbox{ for all $P\in\Gsh[P_0^-, P_0^+]$ and
			$P\ne P_0$}.
		\end{align}
		Then there exists $\hat\delta>0$ such that
		\begin{equation}\label{ine:P0-pr}
		\phi_\ee(P_0^-)-\hat\delta\ge
		\phi_\ee(P_0)\ge \phi_\ee(P_0^+)+\hat\delta.
		\end{equation}
		Moreover,  combining (\ref{NdotE-pos-neg-MaxMin}) with (\ref{ine:P0}), we conclude
		\begin{equation}\label{NdotE-pos-neg-MaxMin-Up}
		\begin{split}
		&\mbox{If $P\in\partial\Omega$ is a
			local maximum (resp. minimum) point of $\phi_\ee$ relative}\\
		&\mbox{to $\overline\Omega$, then $P\notin\Gsh[\Pd, P_0^+]\setminus\{\Pd\}$
          (resp. $P\notin \Gsh[P_0^-, B]\setminus\{B\}$)}.
		\end{split}
		\end{equation}
		Note that (\ref{NdotE-pos-neg-MaxMin-Up}) improves
		(\ref{NdotE-pos-neg-MaxMin}), which follows from (\ref{ine:P0-coord}).
		
		\smallskip
		Let $P_1\in\Gsh[\Pd,\, P_0]$ such that
		$$
		\phi_\ee(P_1)=\max_{P\in\Gsh[\Pd,\, P_0]}\phi_\ee(P).
		$$
		By Lemma \ref{lem:IntAroundPd}(\ref{lem:IntAroundPd-i1})--(\ref{lem:IntAroundPd-i2})
		and (\ref{deltaFor-minInterv-Pd-strictConv}),
		$$
		T_{\Pd}<T_{\Pd^+}\le T_{P_1}, \qquad \phi_\ee(P_1)\ge \phi_\ee(\Pd)+\delta.
		$$
		Moreover, from \eqref{ine:P0-coord} and \eqref{ine:P0-pr}, we obtain
		$$
		\phi_\ee(P_1)\ge\phi_\ee(P_0^-)\ge \phi_\ee(P_0)+\hat\delta.
		$$
		Also, by (\ref{ine:P0}), $T_{P_1}\le T_{P_0^-}$.
		Combining all these facts, we have
		\begin{align}\label{restrictP1}
		&T_{\Pd}<T_{\Pd^+}\le T_{P_1}\le T_{P_0^-},\\
		\label{ineqP1}
		&\phi_\ee(P_1)\ge \phi_\ee(\Pd)+\delta,\quad
		\phi_\ee(P_1)\ge \phi_\ee(P_0)+\hat\delta.
		\end{align}
		From (\ref{NdotE-pos-neg-MaxMin-Up}) with (\ref{ine:P0-coord}) and (\ref{restrictP1}),
		$P_1$ cannot be a local maximum point of $\phi_\ee$ with respect to $\overline\Omega$.
		
		\smallskip
		Therefore,  by Lemma \ref{lem:ExistMinMaxChain}, we can
		construct a maximal chain   of any radius $r_2\in(0, r_1]$ starting from $P_1$.
		We choose $r_2$ so that it works in the argument below. For this,
		we use constant $\hat\delta$ from \eqref{ine:P0-pr},
		choose $\tilde r_2$ the smaller constant $r_1^*$ from
		Lemmas \ref{lem:chainsDoNotIntersect-MinMax}--\ref{lem:chainsDoNotIntersect-MaxMin}
		determined by $\hat\delta$, and then define
		$$
		r_2:=\min\{r_1, \tilde r_2\}.
		$$
		
		Fix a maximal chain of radius $r_2$ starting from $P_1$.
		It ends at some point
		$P_2\in\partial\Omega$ that is a local maximum point of
		$\phi_\ee$ relative to $\overline\Omega$.
		Moreover, by \eqref{ineqP1},
		$\phi_\ee(P_1)\ge \phi_\ee(\Pd)+\delta$; that is,  (\ref{P1-Pd-delta}) holds in the
		present case.
		Since $r_2\le r_1$, then the proof of (\ref{P2-location-strictConv})
		works in the present case so that $P_2$
		lies on $\Gsh^0$ strictly between $\Pd$ and $B$.
		Since $P_2$ is a local maximum point of $\phi_\ee$ relative to $\overline\Omega$,
		we obtain from (\ref{NdotE-pos-neg-MaxMin-Up})  with (\ref{ine:P0-coord})
		that
		$P_2$ lies strictly between $P_0^+$ and $B$  on $\Gsh$.
		Combining with (\ref{ineqP1}),
		we have
		\begin{equation}\label{strictConv-P0-P1-P2}
		T_{P_2}\in(T_{P_0^+},\,T_B)\subset (T_{P_0},\,T_B),
		\qquad \phi_\ee(P_2)>\phi_\ee(P_1)>\phi_\ee(P_0).
		\end{equation}
		
		Let $P_3$ be such that
		$$
		T_{P_3}\in [T_{P_0},T_{P_2}], \qquad \phi_\ee(P_3)=\min_{T_P\in[T_{P_0},T_{P_2}]}\phi_\ee(P).
		$$
		By (\ref{ine:P0})--(\ref{ine:P0-pr})
		and (\ref{ineqP1})--(\ref{strictConv-P0-P1-P2}),
		\begin{align}\label{restricP3}
		&T_{P_3}\in (T_{P_0^+},T_{P_2}],\\
		\label{restricP3-pr}
		&
		\phi_\ee(P_3)<\phi_\ee(P_0^+)\le\phi_\ee(P_0)-\hat\delta<\phi_\ee(P_1)-\hat\delta
		<\phi_\ee(P_2)-\hat\delta.
		\end{align}
		Then, from (\ref{NdotE-pos-neg-MaxMin-Up}) combined with (\ref{ine:P0-coord})
		and (\ref{strictConv-P0-P1-P2}),
		$P_3$ cannot be a local minimum point of $\phi_\ee$
		relative to $\overline\Omega$.
		
		Therefore,  there exists a minimal chain of radius $r_2$ starting from $P_3$ and ending at
		$P_4\in \partial\Omega$. Recall that there exists a maximal chain
		of radius $r_2$ from $P_1$ to $P_2$. Also, it follows from
		(\ref{restricP3-pr})
		that $P_3\ne P_2$ so that $P_3$ lies in $(\Gsh[P_1, P_2])^0$.
		Moreover,
		$\phi_\ee(P_3)\le\phi_\ee(P_1)-\hat\delta$ by (\ref{restricP3-pr}).
		Using the choice of $r_2$ and
		Lemma \ref{lem:chainsDoNotIntersect-MaxMin}, we conclude that
		$P_4\in (\Gsh[P_1, P_2])^0$ and is a local minimum point
		of $\phi_\ee$ relative to $\overline\Omega$. Then,
		from (\ref{NdotE-pos-neg-MaxMin-Up}) combined with (\ref{ine:P0-coord}),
		(\ref{restrictP1}),
		and (\ref{strictConv-P0-P1-P2}), we obtain
		\begin{equation}\label{locationOfP4}
		P_4\in (\Gsh[P_1, P_0^-])^0.
		\end{equation}
		Moreover, combining the facts about the locations of points discussed above together, we have
		\begin{equation}\label{locations-TP1-TP4}
		T_{\Pd}<T_{\Pd^+}\le T_{P_1}<T_{P_4}<T_{P_0^-}<T_{P_0^+}<T_{P_3}<T_{P_2}<T_{B}.
		\end{equation}
		
		\medskip
		Now we follow the previous argument for defining points
		$P_1$, \dots, $P_4$ inductively to
		construct points $P_{4k+1}$, \dots, $P_{4k+4}$ for $k=1, 2, \dots$, as follows
		({\it cf}. Fig. \ref{figure:proof of claim-scnd-k}):
		
		\begin{figure}[!ht]
			\centering
			\includegraphics[width=0.5\textwidth]{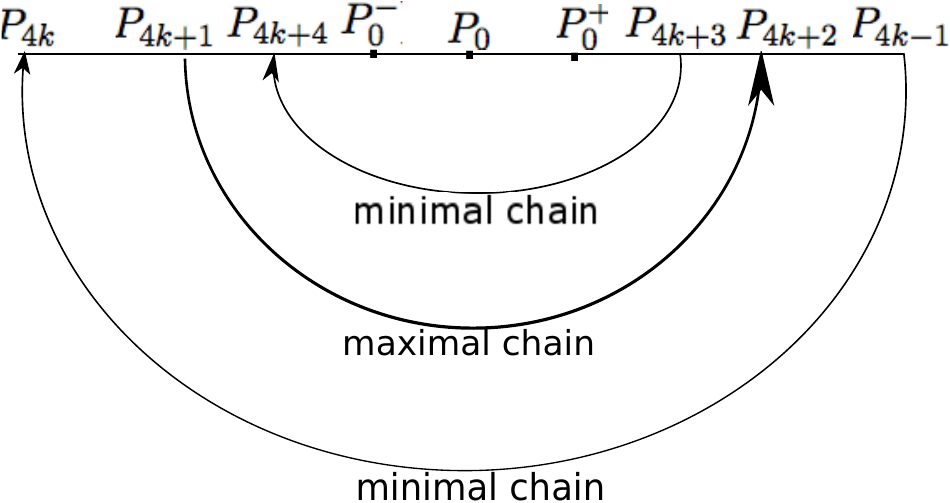}
			\caption{Proof of Claim {\rm 5.7.2}:
				The k-th step of the iteration procedure}
			\label{figure:proof of claim-scnd-k}
		\end{figure}
		
		Fix integer $k\ge 1$ and assume that points
		$P_{4k-1}$ and $P_{4k}$ have been constructed
		with the following properties:
		\begin{align}
		\label{inductAssumptK-1}
		&P_{4k-1}\in(\Gsh[P_0^+, B])^0, \;\;
		P_{4k}\in(\Gsh[\Pd, P_0^-])^0, \\[1mm]
		\label{inductAssumptK-3}
		&\phi_\ee(P_{4k-1})\le \phi_\ee(P_0)-\hat\delta,\\[1mm]
		\label{inductAssumptK-4}
		&\mbox{There exists a minimal chain of radius $r_2$ from $P_{4k-1}$ to $P_{4k}$}.
		\end{align}
		From (\ref{ine:P0-coord}), it follows
		that (\ref{inductAssumptK-1}) can be written as
		\begin{equation}\label{inductAssumptK-1-pr}
		T_{\Pd}< T_{P_{4k}}<T_{P_0^-}<T_{P_0}<T_{P_0^+}<T_{P_{4k-1}}<T_B.
		\end{equation}
		
		We first notice that, for $k=1$,  points $P_3=P_{4k-1}$ and $P_4=P_{4k}$
		satisfy conditions
		(\ref{inductAssumptK-1})--(\ref{inductAssumptK-4}).
		Indeed, for (\ref{inductAssumptK-1}), the first inclusion follows from (\ref{restricP3}) combined with
		(\ref{strictConv-P0-P1-P2}), while the second inclusion follows from
		(\ref{locationOfP4}) combined with (\ref{restrictP1}).
		Property
		(\ref{inductAssumptK-4}) for $P_3$ and $P_4$ follows
		directly from the definition of these points above,
		and (\ref{inductAssumptK-3}) for $P_3$ follows from
		(\ref{restricP3-pr}). Thus, we have the starting point for the induction.
		
		Now, for $k=1,2,\dots$, given $P_{4k-1}$ and $P_{4k}$, we construct
		$P_{4k+1}$, \dots, $P_{4k+4}$. Choose
		$$
		P_{4k+1}\in \Gsh[P_{4k}, P_0]\;
         \mbox{ so that $\phi_\ee(P_{4k+1})=\max_{P\in \Gsh[P_{4k}, P_0]} \phi_\ee(P)$}.
		$$
		Combining (\ref{ine:P0-pr}) with
		(\ref{inductAssumptK-3})--(\ref{inductAssumptK-1-pr}),
		we obtain
		\begin{equation}\label{P-4kp1-estimates}
		\phi_\ee(P_{4k+1})\ge\phi_\ee(P_0^-)\ge\phi_\ee(P_0)+\hat\delta\ge
		\phi_\ee(P_{4k-1})+2\hat\delta>\phi_\ee(P_{4k})+2\hat\delta.
		\end{equation}
		In particular, $P_{4k+1}\ne P_{4k}$.
		Then, from (\ref{ine:P0}) and (\ref{inductAssumptK-1-pr}),
		\begin{align}
		\label{P-4kp1-locat}
		&T_{P_{4k+1}}\in (T_{P_{4k}}, T_{P_0^-}).
		\end{align}
		From
		(\ref{NdotE-pos-neg-MaxMin-Up}),
		\begin{align}
		\label{P-4kp1-NlocMax}
		&\mbox{$P_{4k+1}$ is not a local maximum point of $\phi_\ee$
			relative to $\overline\Omega$.}
		\end{align}
		Thus, there exists a maximal chain of radius $r_2$ starting at $P_{4k+1}$
		and ending at some point
		$P_{4k+2}\in\partial\Omega$, which is a local maximum point of
		$\phi_\ee$ relative to $\overline\Omega$. Moreover,
		\begin{equation}\label{endOfChainIneq-4kPl2}
		\phi_\ee(P_{4k+2})>\phi_\ee(P_{4k+1}).
		\end{equation}
		
		By (\ref{P-4kp1-estimates}), $\phi_\ee(P_{4k+1})\ge\phi_\ee(P_{4k-1})+2\hat\delta$.
		With this,
		using (\ref{inductAssumptK-4})--(\ref{inductAssumptK-1-pr}), (\ref{P-4kp1-locat}),
		the choice of $r_2$,
		and Lemma \ref{lem:chainsDoNotIntersect-MinMax}, we obtain
		$$
		P_{4k+2}\in (\Gsh[P_{4k}, P_{4k-1}])^0.
		$$
		Since $P_{4k+2}$ is a local maximum point of
		$\phi_\ee$ relative to $\overline\Omega$, we use
		(\ref{NdotE-pos-neg-MaxMin-Up}) and (\ref{inductAssumptK-1-pr})
		to obtain
		\begin{equation}\label{P-4kp2-locat}
		T_{P_{4k+2}}\in (T_{P_0^+}, T_{P_{4k-1}}).
		\end{equation}
		Now choose
		$$
		P_{4k+3}\in \Gsh[ P_0, P_{4k+2}]\quad
        \mbox{ so that $\phi_\ee(P_{4k+3})=\min_{P\in \Gsh[P_0, P_{4k+2}]} \phi_\ee(P)$}.
		$$
		Note that $T_{P_0^+}\in (T_{P_0}, T_{P_{4k+2}})$ by (\ref{inductAssumptK-1-pr})
		and (\ref{P-4kp2-locat}).
		Then, from  the definition of $P_{4k+3}$, (\ref{ine:P0-pr}), and (\ref{P-4kp1-estimates}),
		\begin{equation}\label{P-4kp3-estimates}
		\phi_\ee(P_{4k+3})\le\phi_\ee(P_0^+)\le\phi_\ee(P_0)-\hat\delta\le
		\phi_\ee(P_{4k+1})-2\hat\delta.
		\end{equation}
		By (\ref{endOfChainIneq-4kPl2}) and (\ref{P-4kp3-estimates}),
		$\phi_\ee(P_{4k+3})<\phi_\ee(P_{4k+2})$ so that $P_{4k+3}\ne P_{4k+2}$.
		Also, by (\ref{ine:P0-coord})--(\ref{ine:P0}),
		$P_{4k+3}\notin \Gsh[P_0, P_0^+]$.
		Then, using (\ref{P-4kp1-locat}), we have
		\begin{equation}\label{P-4kp3-locat}
		T_{P_{4k+3}}\in ( T_{P_0^+}, T_{P_{4k+2}})\subset (T_{P_{4k+1}}, T_{P_{4k+2}}).
		\end{equation}
		
		In particular, $T_{P_{4k+3}}\in ( T_{P_0}, T_B)$.
		Thus, by (\ref{NdotE-pos-neg-MaxMin-Up}) and (\ref{inductAssumptK-1-pr}),
		$T_{P_{4k+3}}$ is not  a local minimum point of
		$\phi_\ee$ relative to $\overline\Omega$.
		Then there exists a minimal chain of radius $r_2$ starting at $P_{4k+3}$
		and ending at some point
		$P_{4k+4}\in\partial\Omega$ that is a local minimum point of
		$\phi_\ee$ relative to $\overline\Omega$.
		Since there exists a maximal chain of radius $r_2$ from $P_{4k+1}$
		to $P_{4k+2}$, we use (\ref{P-4kp3-estimates})--(\ref{P-4kp3-locat}) and
		Lemma \ref{lem:chainsDoNotIntersect-MaxMin} to conclude that
		$P_{4k+4}\in (\Gsh[P_{4k+1},P_{4k+2}])^0$.
		Since $P_{4k+4}$ is a local minimum point of
		$\phi_\ee$ relative to $\overline\Omega$, we use
		(\ref{NdotE-pos-neg-MaxMin-Up}), (\ref{inductAssumptK-1-pr}), and (\ref{P-4kp1-locat}) to obtain
		\begin{equation}\label{P-4kp4-locat}
		T_{P_{4k+4}}\in (T_{P_{4k+1}}, T_{P_0^-})\subset (T_{\Pd}, T_{P_0^-}).
		\end{equation}
		From (\ref{P-4kp3-locat}) combined with \eqref{inductAssumptK-1-pr} and \eqref{P-4kp2-locat},
		$T_{P_{4k+3}}\in ( T_{P_0^+}, T_B)$. From this
		and
		(\ref{P-4kp4-locat}), we see that
		points $P_{4k+3}$ and $P_{4k+4}$ satisfy
		(\ref{inductAssumptK-1})
		with
		$k+1$ instead of $k$.
		Also, from (\ref{inductAssumptK-1-pr}), (\ref{P-4kp1-locat}), and
		(\ref{P-4kp4-locat}),
		\begin{equation}\label{P-4kp4-locat-next}
		T_{\Pd}<T_{4k}<T_{P_{4k+4}} < T_{P_0^-}.
		\end{equation}
		
		Therefore, we obtain local minimum points  $P_{4k}\in\Gsh^0$,  $k=1,2,\dots$,
		of $\phi_\ee$ which satisfy (\ref{P-4kp4-locat-next}) for each $k$.
		Then there exists a limit $P^*=\lim_{k\to\infty}P_{4k}$ with
		$T_{P^*}\in [T_{\Pd}, T_{P_0^-}]$,
		which implies
		$$
		P^*\in \Gsh^0.
		$$
		Since $P_{4k}\in \Gsh$ is a local minimum point of $\phi_\ee$,
		$\partial_{\bt} \phi_\ee(P_{4k})=0$, so that
		\begin{equation*}
		\frac{\md\phi_\ee(f(T),T)}{\md T}\Big|_{T=T_{P_{4k}}}=0\qquad\,\,
		\mbox{ for $k=1,2,\dots$}.
		\end{equation*}
		From this, since $\{T_{P_{4k}}\}$ is a strictly increasing sequence by (\ref{P-4kp4-locat-next}), we obtain
		\begin{equation}\label{equ:df-param}
		\frac{\md^{n} \phi_\ee(f(T),T)}{\md T^n}\Big|_{T=T_{P^*}}=0
		\qquad\,\, \mbox{ for $n=1,2,\dots$}.
		\end{equation}
		The analyticity of functions $\phi_\ee$ and $f(T)$, shown in
		Lemma \ref{lem:analyticity}, implies that the function:
		$T\mapsto \phi_\ee(f(T),T)$ is real analytic on $(T_A,T_B)$. Then
		we conclude from (\ref{equ:df-param}) that
		$\phi_\ee(f(T),T)\equiv const.$ on $(T_A,T_B)$.
		By (\ref{EequalsMinTau}), we see that
		$\ee=-\bt(P_0)$, so that $\phi_\ee(P_0)=\phi_{\bt}(P_0)=0$,
		where the last equality holds by the first condition
		in \eqref{equ:boundary}.
		That is,
		$$
		\phi_\ee\equiv 0\qquad\, \mbox{on $\Gsh$}.
		$$
		Then, using that $\phi_{\bt}\equiv 0$
		along $\Gsh$
		by the first condition in \eqref{equ:boundary} and that
		$\ee\cdot\bn<0$ at $\Pd$ by Lemma \ref{choseVecEE-strictConv},
		we obtain that
		$D\phi=0$ at $\Pd$.
		This, combined with (\ref{state0-zero}) and the first condition in \eqref{equ:boundary},
		implies that $\rho=\rho_0$ at $\Pd$, which
		contradicts condition (A1)
		of Theorem \ref{thm:main theorem}.
		Therefore, Claim 5.7.2
		is proved.
		
		\medskip
		{\bf 4.}
		Combining Claim 5.7.1 with Claim 5.7.2, we finally conclude Lemma \ref{lem:strict convexity 16}.
	\end{proof}
	
	\begin{lem}\label{NotInGam2-2-inter-Lem}
		$C^k\notin\hat\Gamma_i^0$ for $i=1,2$.
	\end{lem}
	
	\begin{proof}
		Since $C^k$ is a local minimum point of $\phi_\ee$, then condition (A9)
		of Theorem \ref{thm:main theorem-strictUnif}
		and the regularity property $\phi\in C^{1,\alpha}(\overline\Omega)$ imply that
		$\phi_\ee=const.$ on $\overline{\hat\Gamma_i}$.
		Combining this with  (A7)--(A8),
		we obtain that $\phi_\ee=const.$ on $\overline{\hat\Gamma_0\cup\hat\Gamma_1}$
		(resp. on  $\overline{\hat\Gamma_2\cup\hat\Gamma_3}$) if $i=1$ (resp. $i=2$),
		where one or both of $\hat\Gamma_0$ and $\hat\Gamma_3$ may be empty.
		Then, following Remark \ref{symmetry-H9-rmk},  we can assume
		without loss of generality that $C^k\in\overline{\hat\Gamma_2}$ ({\it i.e.}, $i=2$).
		In this case, $B\in \overline{\hat\Gamma_2\cup\hat\Gamma_3}$ so that
		$\phi_\ee(P)=\phi_\ee(B)$ for any $P\in \overline{\hat\Gamma_2\cup\hat\Gamma_3}$.
		From this and
		\eqref{Ck-Pd-strictConv}, we obtain that
		\eqref{lem:strict convexity 16-1} holds in the present case.
		
		Then we are in the same situation as in Lemma \ref{lem:strict convexity 16}.
		Therefore, the proof of  Lemma \ref{lem:strict convexity 16} applies, which yields a
		contradiction.
	\end{proof}
	
	\begin{rem}\label{strictConvRmk1}
		Combining Lemmas {\rm \ref{lem:strict convexity 16}}--{\rm \ref{NotInGam2-2-inter-Lem}},
		we obtain that, if condition {\rm (i)} of assumption {\rm (A10)} holds,
		the only remaining possible location of $C^k$ is on $\Gsh$.
		On the other hand,
		if condition {\rm (ii)} of assumption {\rm (A10)} holds, then
		the remaining possible locations of $C^k$ are either on $\Gsh$ or at the common endpoint $Q^*$
		of $\hat\Gamma_1$ and $\hat\Gamma_2$.
	\end{rem}
	
	\begin{lem}\label{NotInGam2-2-cornerPt-Lem}
		Assume that condition {\rm (ii)}
		of assumption {\rm (A10)} holds,
		and let $Q^*$ be the point defined there. Then $C^k\ne Q^*$.
	\end{lem}
	
	\begin{proof}
		Assume $C^k=Q^*$.
		If  $\phi_\ee$ attains  a local minimum or maximum  relative to $\overline\Omega$
		on $\hat\Gamma_2^0$,
		then condition (A9)
		of Theorem \ref{thm:main theorem-strictUnif}
		and the regularity property $\phi\in C^{1,\alpha}(\overline\Omega)$ imply that
		$\phi_\ee=const.$ on $\overline{\hat\Gamma_2}$. Since $B\in\overline{\hat\Gamma_2}$
		by condition (ii) of assumption (A10), we obtain that $\phi_\ee(P)=\phi_\ee(B)$ for all
		$P\in\overline{\hat\Gamma_2}$.
		Because of $C^k=Q^*\in \overline{\hat\Gamma_2}$, we can complete the proof as in
		Lemma \ref{NotInGam2-2-inter-Lem} above.
		
		Thus, we can assume that
		\begin{equation}\label{phi-minmax-onGamma2hat-strict}
		\mbox{$\phi_\ee$ does not attain its local minimum or maximum
           relative to $\overline\Omega$ on $\hat\Gamma_2^0$}.
		\end{equation}
		Then we consider three cases, depending on whether $\ee\cdot\bn_{\rm sh}(B)$
		is positive, negative, or zero.
		In the argument, we take into account that
		$\hat\Gamma_3=\emptyset$  by condition (ii) of (A10)
		so that $\hat\Gamma_2$ has endpoints
		$Q^*$ and $B$.
		
		If $\ee\cdot\bn_{\rm sh}(B)<0$, then we argue similar to the proof of Claim {5.7.1},
		replacing $\hat\Gamma_3$ by $\hat\Gamma_2$,
		with the differences described below. First, we show (\ref{NdotE-neg-inside})
		without changes in the argument. Next, we choose $P_1\in \Gsh[\Pd,\,B]$ satisfying
		(\ref{Claim571ineq1}) so that the proof of (\ref{P1-Pd-delta}) holds without changes in the
		present case, which implies that $P_1\ne \Pd$.
		However, since (\ref{lem:strict convexity 16-1}) is not available
		in the present case,
		we cannot conclude that $P_1\ne B$. That is, we now obtain that  $P_1\in \Gsh[\Pd, B]\setminus\{\Pd\}$.
		If  $P_1\in (\Gsh[\Pd, B])^0$, then,
		by  (\ref{NdotE-neg-inside}) and Lemma \ref{lem:cannot be maximum},
		$P_1$ cannot be a local maximum point of $\phi_\ee$ relative to
		$\overline\Omega$. If $P_1=B$, then the same conclusion follows from  condition (ii) of (A10)
		since  $\ee\cdot\bn_{\rm sh}(B)<0$.
		Thus, there exists a maximal chain of
		radius $r_2$, starting from $P_1$ and ending at some point
		$P_2\in\partial\Omega$ which is a local maximum point relative to
		$\overline\Omega$, and $\phi_\ee(P_1)< \phi_\ee(P_2)$.
		Now, instead of (\ref{P2-location-strictConv}),
		we show a weaker statement,
		\begin{equation}\label{P2-location-strictConv-pr}
		P_2\in\Gsh[\Pd,B].
		\end{equation}
		To prove (\ref{P2-location-strictConv-pr}), recall that there exists a minimal chain of radius $r_1$
		from $\Pd$ to $C^k=Q^*\in\overline{\hat\Gamma_2}$. Also, $P_1\in \Gsh[\Pd,B]\setminus\{\Pd\}$.
		Then, from (\ref{P1-Pd-delta}) and
		the choice of $r_2$,
		we obtain from
		Lemma \ref{lem:chainsDoNotIntersect-MinMax} that
		either (\ref{P2-location-strictConv-pr}) holds or
		$P_2$
		lies on $\Gamma_2^0$ between $B$ and $C^k$.
		On the other hand, the last case is ruled out by (\ref{phi-minmax-onGamma2hat-strict})
		since  $P_2$ is a local maximum point of $\phi_\ee$ relative to
		$\overline\Omega$.
		Thus, (\ref{P2-location-strictConv-pr}) holds. However, (\ref{P2-location-strictConv-pr}) contradicts
		(\ref{Claim571ineq1}) since $\phi_\ee(P_1)< \phi_\ee(P_2)$.
		Therefore, we reach a contradiction
		in the case that
		$\ee\cdot\bn(B)<0$.
		
		\smallskip
		If $\ee\cdot\bn(B)=0$, we use condition (ii) of (A10)
		and the fact that $C^k=Q^*$ to conclude
		$$
		\phi_\ee(\Pd)>\phi_\ee(C^k)=\phi_\ee(Q^*)=\phi_\ee(B),
		$$
		which implies (\ref{lem:strict convexity 16-1}).
		Now we follow the argument of the proof of Claim 5.7.1
		via replacing $\hat\Gamma_3$ by $\hat\Gamma_2$, up to (\ref{P2-location-strictConv}).
		Instead of (\ref{P2-location-strictConv}),
		we can show (\ref{P2-location-strictConv-pr}) whose
		proof, given above, still works in the present case without changes. Then, as shown above, (\ref{P2-location-strictConv-pr}) contradicts
		(\ref{Claim571ineq1}).  Therefore, we reach a contradiction
		in the case that
		$\ee\cdot\bn(B)=0$.
		
		\smallskip
		If $\ee\cdot\bn(B)>0$, then we argue as in Claim 5.7.2, via replacing $\hat\Gamma_3$
		by $\hat\Gamma_2$, and
		with modifications similar to the ones described above. Specifically,
		(\ref{phi-minmax-onGamma2hat-strict}) is used to conclude that
		$P_2\notin\hat\Gamma_2^0$.
		From this, we conclude that  $P_2$ lies on
		$\Gsh$ between $P_0^+$ and $B$, possibly including $B$.
		However,
		we now cannot rule out the possibility that
		$P_2= B$ as in the proof of  Claim 5.7.2
		(again,
		since  (\ref{lem:strict convexity 16-1}) is not available).
		Thus, instead of
		(\ref{strictConv-P0-P1-P2}), we have
		\begin{equation}\label{strictConv-P0-P1-P2-pr}
		T_{P_2}\in(T_{P_0^+},\,T_B]\subset (T_{P_0},\,T_B],
		\qquad \phi_\ee(P_2)>\phi_\ee(P_1)>\phi_\ee(P_0).
		\end{equation}
		From this, using (\ref{ine:P0})--(\ref{ine:P0-pr}) and (\ref{ineqP1}),
		it follows that
		(\ref{restricP3})--(\ref{restricP3-pr}) hold.
		From
		(\ref{restricP3-pr}), $P_3\ne P_2$, and then
		(\ref{restricP3}) implies
		$$
		T_{P_3}\in (T_{P_0^+},T_{P_2})\subset (T_{P_0},\,T_B).
		$$
		Then, from (\ref{NdotE-pos-neg-MaxMin-Up}) combined with (\ref{ine:P0-coord}),
		$P_3$ cannot be a local minimum point of $\phi_\ee$
		relative to $\overline\Omega$.
		Thus, there exists a minimal chain
		of radius $r_2$  starting from $P_3$. The rest of the proof of
		Claim 5.7.2
		applies without changes.
		Therefore, we obtain a contradiction in the case that
		$\ee\cdot\bn(B)>0$. This completes the proof.
	\end{proof}
	
	\begin{rem}\label{strictConvRmk2}
		Combining Lemmas {\rm \ref{NotInGam2-2-inter-Lem}} and {\rm \ref{NotInGam2-2-cornerPt-Lem}},
		we obtain that, if condition {\rm (ii)} of assumption {\rm (A10)} holds,
		then $C^k$ cannot lie within $\hat\Gamma_1^0\cup\{Q^*\}\cup \hat\Gamma_2^0$.
		Combining this with  Lemma {\rm \ref{lem:strict convexity 16}}, we see that,
		if condition {\rm (ii)} of assumption {\rm (A10)} holds,
		the only remaining possible location of $C^k$ is at $\Gsh$.
	\end{rem}
	
	From Remarks \ref{strictConvRmk1} and \ref{strictConvRmk2},
	in order to complete the proof
	of Theorem \ref{thm:main theorem-strictUnif}, it remains to show
	\begin{lem}\label{NotInShock-Lem}
		$C^k\notin\Gsh$.
	\end{lem}
	
	\begin{proof}
		The proof consists of two steps.
		
		\smallskip
		{\bf 1.} Recall that $\Gsh$ includes its endpoints $A$ and $B$.
		Thus, we first consider the case that $C^k$ is either $A$ or $B$.
		Note that Lemma \ref{lem:strict convexity 16} does not cover this case
		if either $\hat \Gamma_0$ or $\hat \Gamma_3$, or both, are empty.
		
		The argument below does not use condition (A10)
		of Theorem \ref{thm:main theorem-strictUnif}.
		Thus, as discussed
		in Remark \ref{symmetry-H9-rmk}, we can assume
		without loss of generality that  $C^k=B$.
		Then, since there is a minimal chain from
		$\Pd$ to $C^k=B$, we conclude that
		(\ref{lem:strict convexity 16-1}) holds.
		Now the proofs of Claims 5.7.1--5.7.2
		apply, with the following simplification:
		From Lemma \ref{lem:chainsDoNotIntersect-MinMax}
		and the definition of point $P_2$ in each of these claims,
		we conclude that  (\ref{P2-location-strictConv}) holds.
		The rest of the proofs of Claims 5.7.1--5.7.2
		work without changes. Therefore, we reach a contradiction, which
		shows  that $C^k$ is neither $A$ nor $B$.

		\smallskip
		{\bf 2.} It remains to consider the case that  $C^k\in \Gsh^0$.
		Notice that $C^k$ is a local minimum point of $\phi_\ee$.
		Then, from Lemma \ref{lem:cannot be maximum}, we see that $\ee\cdot\bn\leq 0$ at $C_1$.
		Now the argument as in Claim 5.7.1,
		with point $B$ replaced by point $C^k$, works without change.
		This yields a contradiction. Therefore, $C^k\notin \Gsh^0$.
	\end{proof}
	
	\smallskip
	\begin{proof}[Proof of Theorem {\rm \ref{thm:main theorem-strictUnif}}]
		Combining Lemmas \ref{lem:strict convexity 16}--\ref{NotInGam2-2-inter-Lem}
		with Lemma \ref{NotInShock-Lem}, we obtain that
		$C^k$ cannot lie on the set:
		$$
		G:=\overline{\hat\Gamma_0}\cup\Gamma_1^0\cup\Gamma_2^0\cup\overline{\hat\Gamma_3}
		\cup\Gsh.
		$$
		Since $\Gsh$ includes its endpoints,  $G$ covers all
		$\partial\Omega$ except
		point $Q^*$ defined in Case (ii) of (A10), if $Q^*$ exists.
		In Case (i)  of (A10), point $Q^*$ does not exist,
		so that $G=\partial\Omega$, which implies that $C^k\notin\partial\Omega$.
		In Case (ii) of (A10), point  $Q^*$ exists,
		and Lemma \ref{NotInGam2-2-cornerPt-Lem} implies that $C^k\ne Q^*$,
		so that $C^k\notin\partial\Omega$ in this case as well.
		However, the fact that $C^k\notin\partial\Omega$ contradicts
		(\ref{Ck-on-bdry-strictConv}).
		This completes the proof of Theorem \ref{thm:main theorem-strictUnif}.
	\end{proof}
	
	\section{\, Proof of Theorem \ref{theorem:1x}:
		Equivalence between the Strict Convexity and the Monotonicity}
	\label{EquivalenceSect}
	
\begin{proof}[Proof of Theorem {\rm \ref{theorem:1x}}]
	By the boundary condition \eqref{equ:boundary},  $\phi_{\bt}=0$ on $\Gsh$. Also, by assumption (A1),
	$\phi_{\bn}<0$ on $\Gsh$ for the interior normal vector $\bn$.  Then the monotonicity
	property $\phi_{\mathbf{e}}>0$ in $\Gsh^{0}$ for any unit vector
	$\mathbf{e}\in \overline{Con}$ implies that assumption $(A5)$ in
	Theorem \ref{thm:main theorem} holds.
	Now it follows from Theorem \ref{thm:main theorem} that, under the assumptions of
	Theorem \ref{theorem:1x}, the monotonicity
	property is the sufficient condition for the strict convexity of the free boundary
	$\Gsh$ in the sense of
	\eqref{shock-graph-inMainThm}--\eqref{strictConvexity-degenerate-graph}.
	
	On the other hand, if the shock graph is strictly convex in the sense of
	Theorem \ref{thm:main theorem}, then, at any point on $\Gsh^0$, the tangent vector
$\bt$ is not in $\overline{Con}$,
	where we have used the strict convexity in the sense of \eqref{strictConvexity-degenerate-graph}
	to have this property for the boundary directions of the cone.
	Then, using again that  $\phi_{\bt}=0$ and $\phi_{\bn}<0$ on $\Gsh$ in
	\eqref{equ:boundary} and condition $(A1)$ in Theorem \ref{thm:main theorem},
	it follows that
	$\phi_\be> 0$ on $\Gsh^0$
	for any unit vector $\mathbf{e}\in \overline{Con}$; that is, the monotonicity property holds.
	This completes the proof of Theorem \ref{theorem:1x}.
	\end{proof}
	
	\medskip
	\begin{proof}[Proof of the Assertion in Remark {\rm \ref{theorem:1xRem}}]
	By equation \eqref{equ:phi e}
	and condition $(A3)$ in
	Theorem \ref{thm:main theorem}, $\phi_\be$ satisfies the strong minimum principle
	in $\Omega$. This implies
	$$
	\phi_\be>\min\{\min_{\Gsh}\phi_\be,\;\min_{\overline{\Gamma_1\cup\Gamma_2}}\phi_\be\}
	\qquad \mbox{in $\Omega$},
	$$
	where we have used the assumption in Theorem 2.1 that $\phi$ is not a constant state.
	Note that, by the assumption of Theorem \ref{theorem:1x}, $\phi_\ee>0$ on $\Gsh$,
	and $\phi_\be$  on
	$\Gamma_1\cup\Gamma_2$  satisfies that either $\phi_\be\geq0$ or
	$\phi_\be$ cannot attain its local minimum with respect to $\Omega$. Thus, $\phi_\be>0$ in
	$\Omega\cup\Gsh^0$.
	\end{proof}

	\section{\, Applications to Multidimensional Transonic Shock Problems}
	\label{sec:application}
	
	In this section, we apply Theorem \ref{thm:main theorem} to prove the convexity
	of multidimensional transonic shocks for two longstanding shock problems.
	
	\subsection{\, Shock reflection-diffraction problem}
	
	When a plane incident shock hits a two-dimensional wedge, shock reflection-diffraction configurations
    take shape; also see  Chen-Feldman \cite{cf-book2014shockreflection}.

    The wedge is of the shape: $\{|x_2| < x_1\tan \theta_{\rm w}\}$ with $\theta_{\rm w} \in (0, \frac{\pi}{2})$.
    Then the positive $x_1$--axis is the symmetry axis of the wedge,
    the wedge vertex is at the origin, and $\theta_{\rm w}$ is the (half) angle of the wedge.
    The incident shock $S_0$ separates two constant states: state $(0)$
	with velocity $\vv_0=(0,0)$ and density $\rho_0$ ahead of $S_0$, and
	state (1) with velocity $\vv_1=(u_1,0)$ and density $\rho_1$
	behind $S_0$, where $\rho_1>\rho_0$, and $u_1>0$ is
	determined by $(\rho_0, \rho_1, \gamma)$ through the Rankine-Hugoniot conditions on $S_0$.
	The shock, $S_0$, moves in the direction of the $x_1$--axis and hits the wedge vertex at the initial time.
	Also, the slip boundary
	condition: ${\bf v}\cdot\bn=0$ is prescribed on the solid wedge boundary,
	where ${\bf v}$ is the velocity of gas.
	Since state (1) does not satisfy the slip boundary condition,
	the shock reflection-diffraction configurations
	form at later time, which are self-similar so that the problem can be reformulated
	in the self-similar coordinates $\xxi=(\xi_1,\xi_2)=(\frac{x_1}{t}, \frac{x_2}{t})$.
	Depending on the
	flow parameters and wedge angle, there may be various patterns of shock reflection-diffraction
	configurations,
	including
	Regular Reflection and Mach Reflection.
	Because of the symmetry of the problem with respect to the $\xi_1$--axis,
	it suffices to consider the problem only on the upper half-plane $\{\xi_2>0\}$.
	
	The regular reflection configuration is characterized by the fact that
	the reflection occurs at the intersection point $P_0$
	of the incident shock with the wedge boundary.
	Figs.  \ref{figure:shock refelction}--\ref{figure:shock refelction-subs}
	show the structure of regular reflection configurations in self-similar coordinates.
	The regular reflection solutions are piecewise-smooth; that is,
	they are smooth away from the incident and reflected-diffracted shocks,
	as well as the sonic circle $P_1P_4$ for the supersonic regular reflection case across which
	${\bf v}$ is only Lipschitz.
	
	From the description of state (1) above, its pseudo-potential is
	$$
	\varphi_1(\xxi)=-\frac{|\xxi|^2}2+ u_1\xi_1 +C_1.
	$$
	A necessary condition for the existence of piecewise-smooth regular reflection
	configurations is the existence of
	the constant state $(2)$  with pseudo-potential $\varphi_2$ that satisfies both
	the slip boundary condition on the wedge boundary and the Rankine-Hugoniot conditions
	with state $(1)$ across the reflected shock
	$S_1:=\{\varphi_1=\varphi_2\}$.
	Owing to the constant state
	structure (\ref{constantStatesForm}),
	it suffices to require these conditions at $P_0$.
	Thus, the  conditions at $P_0$ are
	\begin{equation}\label{condState2}
	\begin{split}
	&D\varphi_2\cdot\bn_{\rm w}=0,\\
	&\varphi_2=\varphi_1,\\
	&\r(|D\varphi_2|^2,\varphi_2)D\varphi_2\cdot\bn_{S_1}=\rho_1D\varphi_1\cdot\bn_{S_1},\;\;
	\end{split}
	\end{equation}
	where  $\bn_{\rm w}$ is the outward
	(with respect to the wedge) normal vector to the wedge boundary,
	$\theta_{\rm w}$ is the wedge angle in the upper half-plane,
	and $\bn_{S_1}=\frac{D(\varphi_1-\varphi_2)}{|D(\varphi_1-\varphi_2)|}$.
	Therefore, we have three algebraic equations for parameters $(u_2, v_2, C_2)$
	in  expression  (\ref{constantStatesForm}) for $\varphi_2$.
	Since the piecewise-smooth regular reflection solution
	must satisfy (\ref{condState2}) at $P_0$ with $\varphi$ replaced by $\varphi_2$, then
	$(\varphi, D\varphi)=(\varphi_2, D\varphi_2)$ at $P_0$, if $\varphi_2$ exists.
	
	It is well-known (see {\it e.g.} \cite[Chapter 7]{cf-book2014shockreflection}) that,
	given the parameters of states $(0)$ and $(1)$, there
	exists a detachment angle $\theta_{\rm w}^{\rm d}\in (0, \frac{\pi}{2})$ such
	that equations (\ref{condState2}) have two solutions for each wedge angle
	$\theta_{\rm w}\in (\theta_{\rm w}^{\rm d},\frac{\pi}{2})$, which become
	equal when $\theta_{\rm w}=\theta_{\rm w}^{\rm d}$.
	Thus, two types of two-shock configurations occur at $P_0$ in the wedge interval
	$\theta_{\rm w}\in (\theta_{\rm w}^{\rm d}, \frac\pi 2)$.
	For each such $\theta_{\rm w}$,
	state $(2)$ with the smaller density is called a weak state $(2)$.
	The global existence
	of regular reflection solutions for
	all $\theta_{\rm w}\in (\theta_{\rm w}^{\rm d}, \frac{\pi}{2})$
	with $(\varphi, D\varphi)$ at $P_0$ determined by the weak states $(2)$
	has been established in
	\cite{cf-annofmath20101067regularreflection,cf-book2014shockreflection}.
	Below, state $(2)$ always refers to the weak state $(2)$.
	
	If state (2) exists, its pseudo-potential is
	$$
	\varphi_2(\xxi)=-\frac{|\xxi|^2}2+u_2\xi_1+v_2\xi_2+C_2,
	$$
	where $v_2=u_2\tan\theta_{\rm w}$.
	In particular, state (2) satisfies the first condition in \eqref{condState2}
	on the whole wedge boundary (in the upper half-plane $\{\xi_2>0\}$):
	\begin{equation}\label{bcWedgeState2}
	D\varphi_2\cdot\bn_{\rm w}=0 \qquad\;\mbox{ on $\{\xi_2=\xi_1\tan \theta_{\rm w},\; \xi_1>0\}$}.
	\end{equation}
	Depending on the wedge angle,  state $(2)$ can be either supersonic or
	subsonic at $P_0$. Moreover, for $\theta_{\rm w}$ near $\frac\pi 2$
	(resp. for $\theta_{\rm w}$ near $\theta_{\rm w}^{\rm d}$),
	state $(2)$ is supersonic
	(resp. subsonic) at $P_0$; see \cite[Chapter 7]{cf-book2014shockreflection}.
	The type of state $(2)$ at $P_0$ for a given wedge
	angle
	$\theta_{\rm w}$  determines the type of reflection, supersonic or subsonic,
	as shown  in Fig. \ref{figure:shock refelction} or Fig. \ref{figure:shock refelction-subs}
   respectively,
	when $u_1<c_1$.

	\begin{figure}[!ht]
		\centering
		\begin{minipage}{0.45\textwidth}
           \centering
			\includegraphics[width=0.51\textwidth]{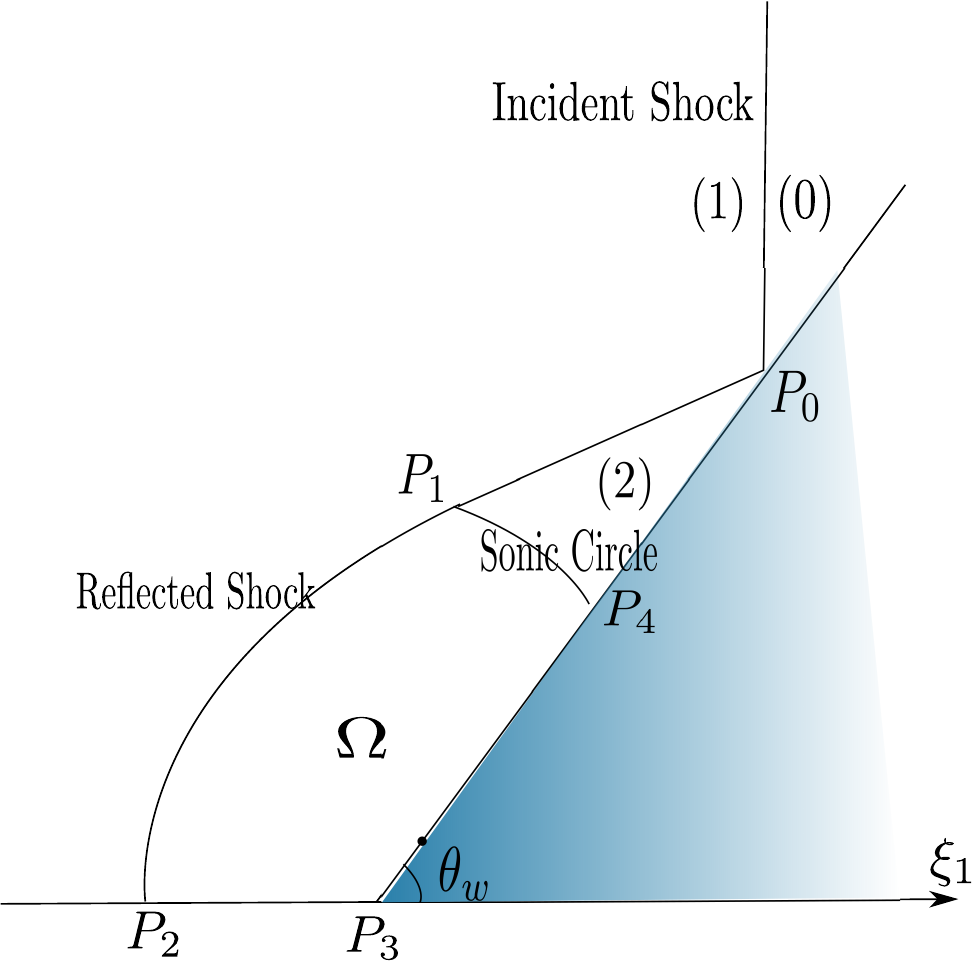}
			\caption{Supersonic regular reflection}
			\label{figure:shock refelction}
		\end{minipage}
		\hspace{0.1in}
		\begin{minipage}{0.43\textwidth}
			\centering
			\includegraphics[width=0.50\textwidth]{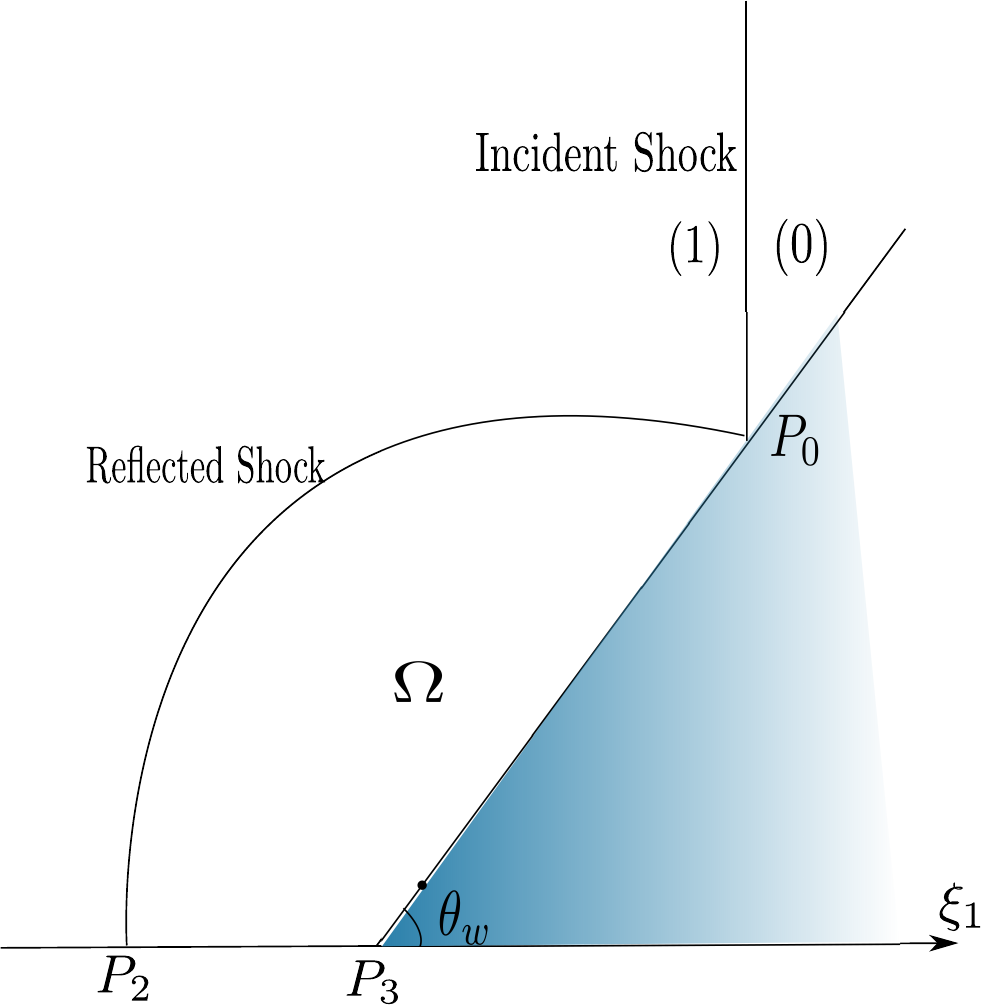}
			\vspace{0.025in}
			\caption{Subsonic regular reflection}
			\label{figure:shock refelction-subs}
		\end{minipage}
	\end{figure}
	
	When $u_1>c_1$, besides the
	configurations shown in  Figs. \ref{figure:shock refelction}--\ref{figure:shock refelction-subs},
	there is an additional possibility that the reflected-diffracted shock is attached to the wedge vertex $P_3$, {\it i.e.}, $P_2=P_3$; see
	Figs. \ref{figure:shock refelction-b}--\ref{figure:shock refelction-subs-b}.

	\begin{figure}[!ht]
		\centering
		\begin{minipage}{0.44\textwidth}
			\centering
			\includegraphics[width=0.51\textwidth]{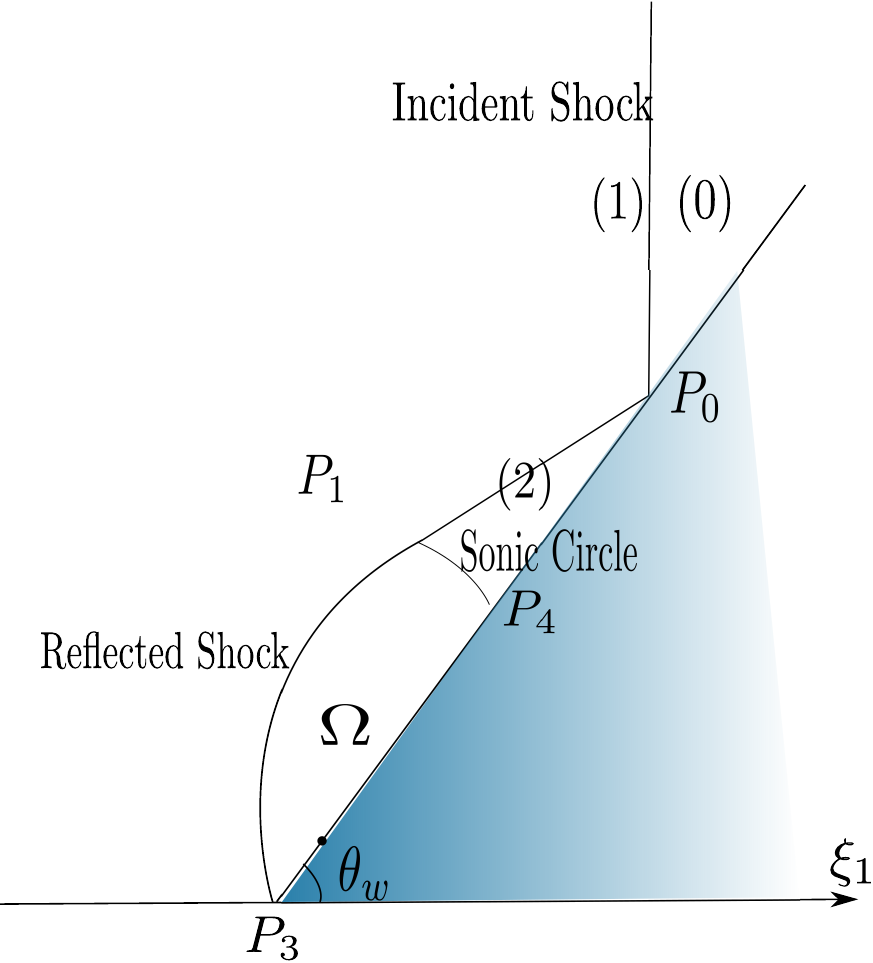}
            \caption{Attached supersonic regular reflection}
			\label{figure:shock refelction-b}
		\end{minipage}
		\hspace{0.1in}
		\begin{minipage}{0.42\textwidth}
			\centering
			\includegraphics[width=0.51\textwidth]{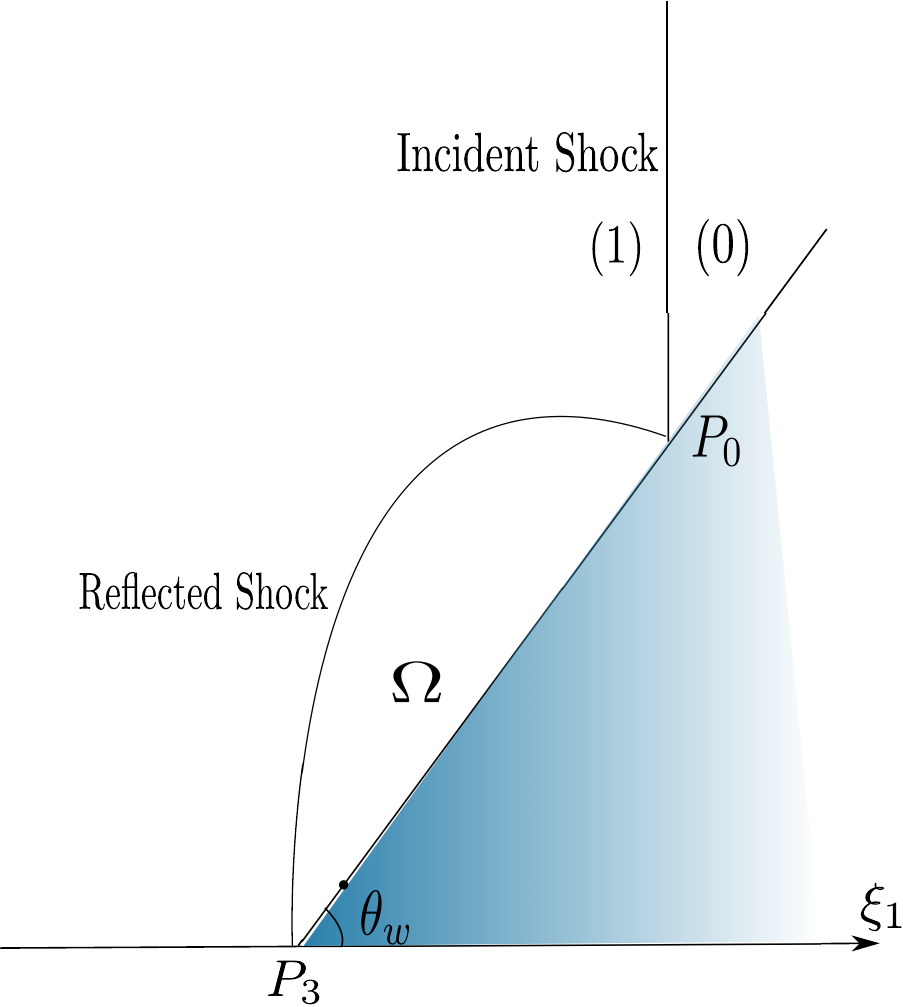}
			\vspace{0.05in}
			\caption{Attached subsonic regular reflection}
			\label{figure:shock refelction-subs-b}
		\end{minipage}
	\end{figure}

	The regular reflection problem is posed in the region:
	$$
	\Lambda=\mR^2_+\setminus\{\xxi\; : \; \xi_1>0, 0<\xi_2<\xi_1\tan\theta_{\rm w}\},
	$$
	where $\mR^2_+:=\mR^2\cap\{\xi_1>0\}$.
	
	\begin{defi}\label{def:weak solutionRegRefl}
		$\varphi\in C^{0,1}(\overline\Lambda)$ is a weak solution of the
		shock reflection-diffraction problem if $\varphi$ satisfies
		equation \eqref{equ:potential flow equation} in $\Lambda$, the boundary conditions{\rm :}
		\begin{equation}\label{BCfor RegRefl}
		\partial_{\bn} \varphi=0\qquad \mbox{on $\partial\Lambda$}
		\end{equation}
		in the weak sense {\rm (}defined below{\rm )},
		and the asymptotic conditions{\rm :}
		\begin{equation}\label{2.13a}
		\lim_{R\to\infty}\|\varphi-\overline{\varphi}\|_{0, \Lambda\setminus
			B_R(\mathbf{0})}=0,
		\end{equation}
		where
		\begin{equation*}
		\bar{\varphi}=
		\begin{cases} \varphi_0 \qquad\mbox{for $\xi_1>\xi_1^0,\,\,\, \xi_2>\xi_1 \tan\theta_{\rm w}$},\\[1mm]
		\varphi_1 \qquad \mbox{for $\xi_1<\xi_1^0, \,\,\, \xi_2>0$},
		\end{cases}
		\end{equation*}
		and $\xi_1^0>0$ is the location of the incident shock $S_0$.
	\end{defi}
	
	In Definition \ref{def:weak solutionRegRefl}, the solution
	is understood in the following weak sense:
	We consider solutions
	with a positive lower bound for the density, so that  (\ref{BCfor RegRefl})
	is equivalent to the conormal condition{\rm :}
	$$
	\r(|D\varphi|^2,\varphi)\partial_{\bn} \varphi=0.
	$$
	Thus, a weak solution
	of problem (\ref{equ:potential flow equation})
	and (\ref{BCfor RegRefl}) is given by
	Definition \ref{def:weak solution} in region $\Lambda$,
	with the following change:
	(\ref{def:weak solution-i3-Eqn}) is satisfied for any
	$\z\in C^{\infty}_{c}(\mR^2)$ (whose support does not have to
	be in $\Lambda$).
	
	Next, we define the points and lines on
	Figs. \ref{figure:shock refelction}--\ref{figure:shock refelction-subs}.
	The incident shock is line $S_0:=\{\xi_1=\xi_1^0\}$ with
	$\xi_1^0=\frac{\rho_1u_1}{\rho_1-\rho_0}>0$.
	The center,  $O_2=(u_2, v_2)$, of the sonic circle
	$B_{c_2}(O_2)$ of state $(2)$
	lies on the wedge boundary between the reflection point $P_0$
	and the wedge vertex $P_3$
	for both the supersonic and subsonic cases.
	
	Then,  for the supersonic case, {\it i.e.},
	when $|D\varphi_2(P_0)|=|P_0 O_2|>c_2$ with
	$P_0\notin \overline{B_{c_2}(O_2)}$,
	we denote by
	$P_4$ the {\it upper} point of intersection of
	$\partial B_{c_2}(O_2)$ with the wedge boundary
	so that  $O_2\in P_3P_4$.
	Also, $\partial B_{c_2}(O_2)$ of state $(2)$ intersects
	line $S_1$,
	and one of the points of intersection, $P_1\in\Lambda$, is such that
	segment $P_0P_1$ is outside $B_{c_2}(O_2)$.
	We denote  the arc of $\partial B_{c_2}(O_2)$ by $\Gso=P_1P_4$.
	The curved part of the reflected-diffracted shock is $\Gsh=P_1P_2$,
	where $P_2\in\{\xi_2=0\}$.
	Denote the line segments $\Gamma_{\rm sym}:=P_2P_3$ and $\Gw:=P_3P_4$.
	The lines and curves $\Gsh$, $\Gso$, $\Gamma_{\rm sym}$, and $\Gw$ do not have
	common points, except their endpoints $P_1$, \dots, $P_4$.
	Thus, $\Gsh\cup\Gso\cup\Gamma_{\rm sym}\cup\Gw$ is a closed curve without
	self-intersection.
	Denote by $\Omega$ the domain bounded by this curve.

	For the subsonic/sonic case, {\it i.e.},
	when $|D\varphi_2(P_0)|=|P_0 O_2|\le c_2$ so that $P_0\in \overline{B_{c_2}(O_2)}$,
	the curved reflected-diffracted shock
	is $\Gsh=P_0P_2$ that does not have common interior points with
	the line segments $\Gamma_{\rm sym}=P_2P_3$ and $\Gw=P_0P_3$.
	Then $\Gsh\cup\Gamma_{\rm sym}\cup\Gw$ is a closed curve without
	self-intersection, and $\Omega$ is the domain bounded
	by this curve.
	
	Furthermore, in some parts of the argument below, it is convenient to
	extend problem (\ref{equ:potential flow equation}) and (\ref{BCfor RegRefl}),
	given in $\Lambda$ by even reflection about the $\xi_1$--axis, {\it i.e.},
	defining $\varphi^{\rm ext}(-\xi_1, \xi_2):=\varphi(\xi_1, \xi_2)$ for any
	$\xxi=(\xi_1,\xi_2)\in\Lambda$.
	Then $\varphi^{\rm ext}$ is defined in region
	$\Lambda^{\rm ext}$ obtained from $\Lambda$ by adding the reflected region $\Lambda^-$,
	{\it i.e.}, $\Lambda^{\rm ext}=\Lambda\cup\{(\xi_1, 0)\; :\;\xi_1<0\}\cup\Lambda^-$.
	In a similar way, region $\Omega$ and curve
	$\Gsh\subset\partial\Omega$ can be extended
	into the corresponding
	region $\Omega^{\rm ext}$ and curve
	$\Gsh^{\rm ext}\subset\partial\Omega^{\rm ext}$.
	
	Now we define a class of solutions, with structure as shown
	on Figs. \ref{figure:shock refelction}--\ref{figure:shock refelction-subs}.
	\begin{defi}\label{admisSolnDef}
		Let $\theta_{\rm w}\in (\theta_{\rm w}^{\rm d},\frac{\pi}2)$.
		A function $\varphi\in C^{0,1}(\overline\Lambda)$ is an admissible solution of the regular reflection
		problem \eqref{equ:potential flow equation} and \eqref{BCfor RegRefl}--\eqref{2.13a}
		if $\varphi$ is a solution in the sense of Definition {\rm \ref{def:weak solutionRegRefl}}
		and satisfies the following properties{\rm :}
		\begin{enumerate}[\rm (i)]
			\item\label{RegReflSol-Prop0}
			The structure of solutions is as follows{\rm :}
			
			\smallskip
			\begin{itemize}
				\item
				If $|D\varphi_2(P_0)|>c_2$,	then $\varphi$ is of the {\em supersonic} regular shock
				reflection-diffraction configuration shown on
				Fig. {\rm \ref{figure:shock refelction}}			
				and satisfies{\rm :}
				
				The reflected-diffracted shock $\Gsh$ is $C^{2}$ in its relative interior.
				Curves $\Gsh$, $\Gso$, $\Gw$, and $\Gamma_{\rm sym}$ do not
				have common points except their endpoints.

				$\varphi$ satisfies the following properties{\rm :}
				\begin{align}
				& \varphi\in C^{0,1}(\Lambda)\cap C^1(\Lambda\setminus (S_0\cup \overline{P_0P_1P_2})),\notag\\[2mm]
				&\varphi\in C^{1}(\overline{\Omega})\cap C^{3}(\overline\Omega\setminus(\overline\Gso\cup\{P_2, \,P_3\})), \notag\\[2mm]
				&\varphi=\left\{\begin{array}{ll}
				\varphi_0 \qquad\mbox{for $\xi_1>\xi_1^0$ and $\xi_2>\xi_1\tan\theta_{\rm w}$},\\[2mm]
				\varphi_1 \qquad\mbox{for $\xi_1<\xi_1^0$ and above curve $P_0P_1P_2$},\\[2mm]
				\varphi_2 \qquad \mbox{in region $P_0P_1P_4$}.
				\end{array}\right. \label{phi-states-0-1-2-MainThm}
				\end{align}

				\smallskip
				\item If $|D\varphi_2(P_0)|\le c_2$,	then $\varphi$ is of the {\em subsonic} regular shock
				reflection-diffraction configuration
				shown on Fig. {\rm \ref{figure:shock refelction-subs}}			
				and satisfies{\rm :}
				
				The reflected-diffracted shock $\Gsh$ is $C^{2}$ in its relative interior.
				Curves $\Gsh$,  $\Gw$, and $\Gamma_{\rm sym}$ do not
				have common points except their endpoints.
				
				$\varphi$ satisfies the following properties{\rm :}
				\begin{align}
				& \varphi\in C^{0,1}(\Lambda)\cap C^1(\Lambda\setminus \overline\Gsh),\notag\\[2mm]
				&\varphi\in  C^{1}(\overline{\Omega})\cap
				C^3(\overline\Omega\setminus\{P_0, P_2, P_3\}), \notag\\[2mm]
				&\varphi=\left\{\begin{array}{ll}
				\varphi_0 \quad&\mbox{for $\xi_1>\xi_1^0$ and $\xi_2>\xi_1\tan\theta_{\rm w}$},\\[2mm]
				\varphi_1 \quad&\mbox{for $\xi_1<\xi_1^0$ and above curve $P_0P_2$},\\[2mm]
				\varphi_2(P_0) \quad &\mbox{at $P_0$},
				\end{array}\right. \label{phi-states-0-1-2-MainThm-Subs}\\[2mm]
				&D\varphi(P_0)=D\varphi_2(P_0). \notag
				\end{align}
			\end{itemize}
			Furthermore, in both supersonic and subsonic cases,
			\begin{equation}\label{regularityOfRegReflSupers}
			\Gsh^{\rm ext}\,\,\,\,\mbox{is $C^1$ in its relative interior}.
			\end{equation}
			
			\smallskip		
			\item\label{RegReflSol-Prop1}
			Equation \eqref{equ:potential flow equation} is strictly elliptic in
			$\overline\Omega\setminus\,\overline{\Gso}${\rm :}
			$$
			|D\varphi|<c(|D\varphi|^2, \varphi)
			\qquad\mbox{ in $\overline\Omega\setminus\,\overline{\Gso}$},
			$$
			where, for the subsonic and sonic cases,  we have used notation $\overline{\Gso}=\{P_0\}$.
			\smallskip
			\item\label{RegReflSol-Prop1-1}
			$\partial_{\bn}\varphi_1>\partial_{\bn}\varphi>0$ on $\Gsh$, where $\bn$ is the normal vector
			to $\Gsh$ pointing into $\Omega$.
			
			\smallskip
			\item \label{RegReflSol-Prop1-1-1}
			$\varphi_2\le\varphi\le\varphi_1$ in $\Omega$.
			
			\smallskip
			\item\label{RegReflSol-Prop2}
			Let $\mathbf{e}_{S_1}$ be the unit vector parallel to $S_1:=\{\varphi_1=\varphi_2\}$, oriented so that
			$\mathbf{e}_{S_1}\cdot D\varphi_2(P_0)>0${\rm :}
			\begin{align}\label{defEs1a}
			&\mathbf{e}_{S_1}=
			-\frac {(v_2,\, u_1-u_2)}{\sqrt{(u_1-u_2)^2+v_2^2}}.
			\end{align}
			Let $\mathbf{e}_{\xi_2}=(0,1)$. Then
			\begin{equation}\label{nonstritConeMonot}
			\partial_{\mathbf{e}_{S_1}}(\varphi_1-\varphi)\leq0, \quad\,\,
			\partial_{\xi_2}(\varphi_1-\varphi)\leq0 \qquad\;\;\mbox{ on $\Gsh$}.
			\end{equation}	
		\end{enumerate}
	\end{defi}
	
	Below we continue to use the notational convention:
	\begin{equation}\label{subsonicConvenct}
	\mbox{$\overline{\Gso}:=\{P_0\}$, $\;P_1:=P_0$, $P_4:=P_0$ $\quad$ for the subsonic and sonic cases}.
	\end{equation}
	
	\begin{rem}\label{eqnInOmega-rmk}
		Since the admissible solution $\varphi$ in Definition {\rm \ref{admisSolnDef}} is a weak solution in the
		sense of
		Definition {\rm \ref{def:weak solutionRegRefl}} and has the regularity as in
		Definition {\rm \ref{admisSolnDef}(\ref{RegReflSol-Prop0})}, it satisfies
		\eqref{equ:study} classically in $\Omega$ with $\phi=\varphi-\varphi_1$,
		the Rankine-Hugoniot conditions{\rm :}
		\begin{equation}\label{equ:boundary-RH-RegRefl}
		\varphi=\varphi_1,\quad \r(|D\varphi|^2,\varphi)D\varphi\cdot\bn=\rho_1D\varphi_1\cdot\bn
		\qquad
		\mbox{ on $\Gsh$},
		\end{equation}
		and the boundary conditions{\rm :}
		\begin{equation}\label{equ:boundary-ofOmega-RegRefl}
		\partial_{\bn}\varphi=0\qquad\;\mbox{ on $\Gw\cup\Gamma_{\rm sym}$}.
		\end{equation}
		Note also that, rewriting \eqref{equ:boundary-ofOmega-RegRefl} in terms of $\phi=\varphi-\varphi_1$, we have
		\begin{equation}\label{equ:boundary-ofOmega-RegRefl-phi}
		\begin{split}
		&\partial_{\bn}\phi=-u_1\sin\theta_{\rm w}\qquad\; \mbox{ on $\Gw$}, \\
		&\partial_{\bn}\phi=0\qquad\qquad\qquad\, \mbox{ on $\Gamma_{\rm sym}$}.
		\end{split}
		\end{equation}
	\end{rem}
	
	\begin{rem}\label{nonconsta-rmk}
		An admissible solution $\varphi$ is not a constant state in $\Omega$ $($recall that $\theta_{\rm w}<\frac\pi 2${\rm )}.
		Indeed, if
		$\varphi$ is a constant state in $\Omega$, then $\varphi=\varphi_2$ in $\Omega${\rm :}
		This follows from \eqref{phi-states-0-1-2-MainThm} for the supersonic case since
		$\varphi$ is $C^{1}$ across $\Gso$, and  from
		the property that
		$(\varphi, D\varphi)=(\varphi_2, D\varphi_2)$ at $P_0$ for the subsonic case.
		However, $\varphi_2$
		does not satisfy \eqref{equ:boundary-ofOmega-RegRefl} on $\Gamma_{\rm sym}$
		since $\vv_2=(u_2, v_2)=(u_2, u_2\tan\theta_{\rm w})$ with $u_2>0$ and $\theta_{\rm w}\in(0, \frac\pi 2)$.
	\end{rem}
	
	\begin{rem}\label{BConStraightBdrySegm}
		Let $\varphi$ be an admissible solution and $\phi:=\varphi-\varphi_1$.
		For a unit vector $\ee\in\mr^2$,
		denote
		$$
        w=\phi_\ee.
        $$
		Then, from the regularity in Definition {\rm \ref{admisSolnDef}(\ref{RegReflSol-Prop0})},
		$$
         w\in C(\overline{\Omega})\cap C^{2}(\overline\Omega\setminus(\overline\Gso\cup\{ P_3\})),
        $$
		where we have used \eqref{subsonicConvenct} for the subsonic and sonic cases.
		
		We first notice that $w$
		satisfies equation \eqref{equ:phi e} in the $(S,T)$--coordinates with basis $\{\ee, \ee^\perp\}$.
		Equation \eqref{equ:phi e} has the same coefficients of the second-order terms as equation
		\eqref{nondivMainEq}, so that \eqref{equ:phi e} is strictly elliptic
		in $\overline\Omega\setminus\,\overline{\Gso}$ by
		Definition {\rm \ref{admisSolnDef}}\eqref{RegReflSol-Prop1}.
		
		Furthermore, by {\rm \cite[{\it Lemma} 5.1.3]{cf-book2014shockreflection}}, $w$ satisfies the
		following boundary conditions on the straight segments $\Gw$ and
		$\Gamma_{\rm sym}${\rm :} If $\ee\cdot\bt\ne 0$ for a unit tangent vector $\bt$ on
		$\Gw$ {\rm (}resp. $\Gamma_{\rm sym}${\rm )}, then
		\begin{equation}\label{equ:boundary-ofOmega-RegRefl-w}
		 w_{\bn}+\frac{(\ee\cdot\bn)(c^2-\varphi_{\bt}^2)}{(\ee\cdot\bt)(c^2-\varphi_{\bn}^2)}
		w_{\bt}=0 \qquad \mbox{on $\Gw^0$ {\rm (}resp. $\Gamma_{\rm sym}^0${\rm )}}.
		\end{equation}
		The coefficients are continuous and hence locally bounded,
		which implies that these boundary conditions are
		oblique on $\Gw^0$ {\rm (}resp. $\Gamma_{\rm sym}^0${\rm )}.
	\end{rem}
	
	\begin{lem}\label{admisEquiv1-2}
		Definition {\rm \ref{admisSolnDef}} is equivalent to the definition of admissible
		solutions in {\rm \cite{cf-book2014shockreflection}}{\rm ;} see Definitions {\rm 15.1.1}--{\rm 15.1.2} there.
	\end{lem}
	\begin{proof}
		In order to show that the solutions in Definition {\rm \ref{admisSolnDef}} satisfy all the properties in
		Definitions {\rm 15.1.1}--{\rm 15.1.2} of Chen-Feldman {\rm \cite{cf-book2014shockreflection}}, it requires to show
		that they satisfy:
		\begin{align}
		\label{eqDefAdm-1}
		&\xi_{1 P_2}\le \xi_{1 P_1}, \quad\,\,
		\overline{\Gsh}\subset(\Lambda\setminus \overline{B_{c_1}(O_1)})\cap
		\{\xi_{1 P_2}\le\xi_1\le \xi_{1 P_1}\}, \\
		\label{eqDefAdm-2}
		&\partial_{\mathbf{e}_{S_1}}(\varphi_1-\varphi)\leq0, \quad
		\partial_{\xi_2}(\varphi_1-\varphi)\leq0 \qquad\;\; \mbox{in $\overline\Omega$},
		\end{align}
		where $O_1=(u_1, 0)$ is the center of sonic circle of state (1) and, in the subsonic
		reflection case (see Fig. {\rm \ref{figure:shock refelction-subs}}), we have used the notational
		convention (\ref{subsonicConvenct}).
		Moreover, note that the inequalities in (\ref{nonstritConeMonot}) hold
		on $\Gsh$, while these inequalities in \eqref{eqDefAdm-2} hold in the larger domain $\overline\Omega$.
		
		We first show both (\ref{eqDefAdm-2}) and
		the stronger property:
		\begin{equation}\label{strictConeMonot-OmegaLem}
		\partial_{\mathbf{e}_{S_1}}(\varphi_1-\varphi)<0, \quad\,\,
		\partial_{\xi_2}(\varphi_1-\varphi)<0 \qquad\;\;\mbox{ in $\Omega$}.
		\end{equation}
		The argument is the same as the one in the proof of Remark \ref{theorem:1xRem} (see \S \ref{EquivalenceSect})
		for $\phi=\varphi-\varphi_1$ in the present case.
		We only need to check
		for $\mathbf{e}=\mathbf{e}_{S_1}$ and $\mathbf{e}=\mathbf{e}_{\xi_2}$
		that, for any point $\xxi \in \partial\Omega\setminus\Gsh^0$, $\phi_{\mathbf{e}}$ satisfies that
		\begin{equation}\label{fixedBdEquivCond}
		\mbox{ either $\phi_{\mathbf{e}}(\xxi)\geq0$ or
			$\phi_{\mathbf{e}}$ cannot attain its local minimum at $\xxi$.}
		\end{equation}
		Note that $\partial\Omega\setminus\Gsh^0=\overline{\Gso}\cup\Gw\cup\Gamma_{\rm symm}\cup\{P_3\}$.

		Consider first $\ee=\ee_{S_1}$.  Since $D\varphi(P_3)=(0,0)$ by (\ref{equ:boundary-ofOmega-RegRefl}) and
		$\varphi\in C^1(\overline\Omega)$, we conclude that $w(P_3)=0$.
		Next,  $\ee_{S_1}\cdot\bt\ne 0$ on $\Gw\cup \Gamma_{\rm sym}$ by \cite[Lemma 7.5.12]{cf-book2014shockreflection}.
		Then, by Remark \ref{BConStraightBdrySegm}, $\phi_{\mathbf{e}}$ satisfies a homogeneous elliptic equation
		in $\Omega$ and
		the oblique boundary conditions (\ref{equ:boundary-ofOmega-RegRefl-w}) on $\Gw^0\cup\Gamma_{\rm sym}^0$, so that $w$
		cannot attain its local minimum on $\Gw^0\cup \Gamma_{\rm sym}^0$, unless $w$ is constant in $\Omega$ in which case
		$w\equiv w(P_3)=0$ in $\overline\Omega$.
		On $\overline\Gso$,
		$(\varphi, D\varphi)=(\varphi_2, D\varphi_2)$
		as shown in Remark \ref{nonconsta-rmk},
		where we have used notation (\ref{subsonicConvenct}).
		Also, $\mathbf{e}_{S_1}\cdot D(\varphi_2-\varphi_1)=0$ by \eqref{defEs1a}.
		Thus, $\phi_{\mathbf{e}_{S_1}}=\mathbf{e}_{S_1}\cdot D(\varphi_2-\varphi_1)=0$ on $\overline\Gso$,
		which implies (\ref{fixedBdEquivCond}) for $\ee=\mathbf{e}_{S_1}$.
		
		Now we show (\ref{fixedBdEquivCond}) for $\ee=\mathbf{e}_{\xi_2}$, {\it i.e.}, $w=\phi_{\xi_2}$.
		The argument is similar to the previous case, with the following
		differences:
		First, $\mathbf{e}_{\xi_2}\cdot\bt =0$ on $\Gamma_{\rm sym}$ so that, instead of
		(\ref{equ:boundary-ofOmega-RegRefl-w}), we obtain that $w=0$ on $\Gamma_{\rm sym}$ by
		\eqref{equ:boundary-ofOmega-RegRefl-phi}.
		Also, on $\Gso$, we use again that $D\varphi=D\varphi_2$
		to obtain that $w=\phi_{\xi_2}=(\varphi_2-\varphi_1)_{\xi_2}=v_2\ge 0$.
		The rest of the argument is the same as above, which leads to
		(\ref{fixedBdEquivCond}) for $\ee=\mathbf{e}_{\xi_2}$.
		
		Repeating the proof of Remark \ref{theorem:1xRem} (see \S 6), in which
		$\phi$ is not a constant state by Remark \ref{nonconsta-rmk},
		we obtain (\ref{strictConeMonot-OmegaLem}). With this, (\ref{eqDefAdm-2}) is proved.
		
		Next we show \eqref{eqDefAdm-1}. Since $\overline{\Gsh}\subset \Lambda\setminus\overline{B_{c_1}(O_1)}$,
		then $\varphi_1$ is supersonic on $\Gsh$. This is a standard consequence
		of the Rankine-Hugoniot conditions \eqref{equ:boundary-RH-RegRefl}, combined with the entropy
		condition of Definition \ref{admisSolnDef}(\ref{RegReflSol-Prop1-1}).
		
		It remains to show that
		$\xi_{1 P_2}\le \xi_{1 P_1}$ and
		$\overline{\Gsh}\subset\{\xi_{1 P_2}\le\xi_1\le \xi_{1 P_1}\}$.
		From (\ref{strictConeMonot-OmegaLem}), $\phi_{\xi_2}>0$ in $\Omega$.
		Also, $\phi=0$ on $\Gsh$ and $\phi\le 0$ in $\Omega$ by
		(\ref{RegReflSol-Prop1-1-1}).  From these properties and
		the regularity of curve $\Gsh$,
		it follows that any vertical line that has a non-empty intersection with
		$\Gsh$ intersects $\Gsh$ either at one point or on a closed interval.
		Moreover,
		\begin{equation}\label{locationOfShockRegRefl}
		\mbox{
			If $(\xi_1^*, \xi_2^*)\in \Gsh$, then
			$\Omega\cap\{(\xi_1, \xi_2)\;:\;\xi_1=\xi_1^*\}\subset \{(\xi_1, \xi_2)\;:\;\xi_2<\xi_2^*\}$.
		}
		\end{equation}
		From
		these properties, we conclude that
		$$
         \overline\Gsh\subset \{\min(\xi_{1 P_1},\,\xi_{1 P_2})\le\xi_1\le \max(\xi_{1 P_1},\,\xi_{1 P_2})\}.
        $$
		It remains to show that $\xi_{1 P_2}\le \xi_{1 P_1}$.
		Assume that $\xi_{1 P_2}> \xi_{1 P_1}$.
		Then, from (\ref{locationOfShockRegRefl}) and
		the structure of $\Omega$ described in Definition \ref{admisSolnDef}(\ref{RegReflSol-Prop0}),
		we conclude that $\overline\Gsh$ is contained within the following subregion of
		$\{\xi_{1 P_1}\le \xi_1 \le \xi_{1 P_2}\}$:
		Above $\overline\Gso$ on $\{\xi_{1 P_1}\le \xi_1 \le \min(\xi_{1 P_2}, \xi_{1 P_4})\}$,
		and above  $\Gw$ on $\{\xi_{1 P_4}\le \xi_1 \le \xi_{1 P_2}\}$ if $\xi_{1 P_2}\ge \xi_{1 P_4}$.
		This implies that $\overline\Gsh\subset \{\xi_2>0\}$. This contradicts the fact that endpoint $P_2$
		of $\Gsh$ lies on $\{\xi_2=0\}$.
		Now \eqref{eqDefAdm-1} is proved.
		
		Therefore, we have shown that the solutions in Definition {\rm \ref{admisSolnDef}}
		satisfy all the properties in
		Definitions {\rm 15.1.1}--{\rm 15.1.2} of {\rm \cite{cf-book2014shockreflection}}.
		
		Now we show that the admissible solutions defined in Definitions {\rm 15.1.1}--{\rm 15.1.2} of
		{\rm \cite{cf-book2014shockreflection}} satisfy all the properties of Definition {\rm \ref{admisSolnDef}}.
		For that, we need to show that the admissible solutions  in Definitions {\rm 15.1.1}--{\rm 15.1.2}
		of {\rm \cite{cf-book2014shockreflection}} satisfy property
		(\ref{RegReflSol-Prop1-1}) of Definition \ref{admisSolnDef}.
		This is proved in  \cite[Lemma 8.1.7 and Proposition 15.2.1]{cf-book2014shockreflection}.
	\end{proof}
	
	From Lemma \ref{admisEquiv1-2}, all the estimates and properties of admissible solutions shown in {\rm \cite{cf-book2014shockreflection}}
	hold for the admissible solutions defined above.
	We list some of these properties in the following theorem.
	
	Below we use the notation that, for two unit vectors $\ee, {\mathbf f}\in\mR^2$ with
	$\ee\ne \pm {\mathbf f}$,
	\begin{equation}\label{3.5a}
	Con(\ee, \mathbf{f}):=\{a\mathbf{e}+b\mathbf{f}\; : \;a,b>0\}.
	\end{equation}
	
	\begin{thm}[Properties of admissible solutions]\label{Regularity-RegRefl-Th1}
		$\,\,$	There exits a constant $\alpha=\alpha(\rho_0,\rho_1,\gamma)\in (0, \frac{1}{2})$
		such that any admissible solution in the sense of Definition {\rm \ref{admisSolnDef}}
		with wedge angle $\theta_{\rm w}\in (\theta_{\rm w}^{\rm d},\frac{\pi}2)$ has the following properties{\rm :}
		\begin{enumerate}[\rm (i)]
			\item\label{RegReflSol-Prop0-th}
			Additional regularity{\rm :}
			\smallskip
			\begin{itemize}
				\item
				If $|D\varphi_2(P_0)|>c_2$,	i.e., when $\varphi$ is of the {\em supersonic} regular shock
				reflection-diffraction configuration
				as in Fig. {\rm \ref{figure:shock refelction}}, 		
				it satisfies
				$$
				\varphi\in C^\infty(\overline{\Omega}\backslash(\overline{\Gso}\cup\{P_3\}))
				\cap C^{1,1}(\overline{\Omega}\backslash\{P_3\})\cap C^{1,\alpha}(\overline{\Omega}).
				$$
				The reflected-diffracted shock $P_0P_1P_2$ $($where $P_0P_1$ is the straight segment and $P_1P_2=\Gsh${\rm )} is $C^{2,\beta}$
				up to its endpoints for any $\beta\in [0, \frac{1}{2})$
				and	$C^\infty$ except $P_1$.
				
				\smallskip
				\item If $|D\varphi_2(P_0)|\le c_2$,
				i.e., when $\varphi$ is of the {\em subsonic} regular shock reflection-diffraction configuration
				as in Fig. {\rm \ref{figure:shock refelction-subs}}, it satisfies
				$$
				\varphi\in  C^{1,\beta}(\overline{\Omega})\cap
				C^{1,\alpha}(\overline\Omega\setminus \{P_0\})
				\cap
				C^{\infty}(\overline\Omega\setminus \{P_0,  P_3\})
				$$
				for some $\beta=\beta(\rho_0,\rho_1,\gamma, \theta_{\rm w})\in (0, \alpha]$ which
				is non-decreasing with respect to $\theta_{\rm w}$,
				and the reflected-diffracted shock $\Gsh$ is $C^{1,\beta}$
				up to its endpoints
				and $C^\infty$ except $P_0$.			
			\end{itemize}
			
			Furthermore, in both supersonic and subsonic cases,
			$$
			\varphi^{\rm ext}\in C^\infty(\Omega^{\rm ext}\cup(\Gsh^{\rm ext})^0).
			$$
			
			\smallskip
			\item\label{RegReflSol-Prop2-th}
			For each $\mathbf{e}\in Con(\mathbf{e}_{S_1}, \mathbf{e}_{\xi_2})$,
			\begin{equation}\label{strictCone-regrefl}
			\partial_{\mathbf{e}}(\varphi_1-\varphi)< 0 \qquad\mbox{in $\overline\Omega$},
			\end{equation}
			where vectors  $\mathbf{e}_{S_1}$ and $\mathbf{e}_{\xi_2}$ are
			introduced in  Definition {\rm \ref{admisSolnDef}}\eqref{RegReflSol-Prop2}.
			
			\smallskip
			\item \label{RegReflSol-Prop3-th}
			Denote by $\bn_{\rm w}$ the
			unit interior normal vector to $\Gw$ {\rm (}pointing into $\Omega${\rm )}, {\it i.e.},
			$\bn_{\rm w}=(-\sin\theta_{\rm w}, \cos\theta_{\rm w})$.
			Then
			$\partial_{\bn_{\rm w}}(\varphi-\varphi_2)\leq0$ in $\overline\Omega$.
		\end{enumerate}
	\end{thm}
	\begin{proof}
		Below we use the equivalence shown in Lemma \ref{admisEquiv1-2}.
		
		Assertion (\ref{RegReflSol-Prop0-th})
		follows from Definition \ref{admisSolnDef}(\ref{RegReflSol-Prop0})
		and
		\cite[Corollary 11.4.7, Proposition 11.5.1, Corollary 16.6.12]{cf-book2014shockreflection}.
		Assertion (\ref{RegReflSol-Prop2-th}) is obtained in
		\cite[Corollary 8.2.10, Proposition 15.2.1]{cf-book2014shockreflection}.
		Assertion (\ref{RegReflSol-Prop3-th}) follows from
		\cite[Lemma 8.2.19, Proposition 15.2.1]{cf-book2014shockreflection},
		where ${\mathbf n}_{\rm w}=-\bn_{\rm w}$.
	\end{proof}
	
	\begin{rem}\label{nuWinCone}
		We note that $\bn_{\rm w}\in Con(\mathbf{e}_{S_1}, \mathbf{e}_{\xi_2})$ for
		any wedge angle $\theta_{\rm w}\in (\theta_{\rm w}^{\rm d},\frac{\pi}2)$,
		which is proved in {\rm \cite[{\it Lemma} 8.2.11]{cf-book2014shockreflection}}.
	\end{rem}
	
	\smallskip
	Now we state the results on the existence of admissible solutions.

	\begin{thm}[Global solutions up to the detachment angle for the case: $u_1\le c_1$]\label{ExistRegRefl-Th1}
		Let the initial data $(\rho_0, \rho_1, \gamma)$ satisfy that $u_1\le c_1$.
		Then, for each $\theta_{\rm w}\in (\theta_{\rm w}^{\rm d},\frac{\pi}2)$, there
		exists an admissible solution of the regular reflection problem
		in the sense of Definition {\rm \ref{admisSolnDef}}, which
		satisfies the properties stated in Theorem {\rm \ref{Regularity-RegRefl-Th1}}.
	\end{thm}
	\begin{proof}
		The existence of admissible solutions directly follows from Lemma \ref{admisEquiv1-2}
		and \cite[Theorem 2.6.7 and Remark 2.6.8]{cf-book2014shockreflection}.
	\end{proof}
	
	When $u_1> c_1$, the results of Theorem \ref{ExistRegRefl-Th1} hold
	for any wedge angle $\theta_{\rm w}$ from $\frac{\pi}{2}$ until either
	$\theta_{\rm w}^{\rm d}$ or
	$\theta_{\rm w}^{\rm c}\in (\theta_{\rm w}^{\rm d}, \frac{\pi}2)$
	when the shock hits the wedge vertex $\PtLwR$.
	
	\smallskip
	\begin{thm}[Global solutions up to the detachment angle for the case: $u_1> c_1$]\label{ExistRegRefl-Th2}
		Let the initial data  $(\rho_0, \rho_1, \gamma)$ satisfy that $u_1>c_1$.	
		Then there is
		$\theta_{\rm w}^{\rm c}\in[\theta^{\rm d}_{\rm w},\frac{\pi}{2})$ such
		that, for each $\theta_{\rm w}\in (\theta_{\rm w}^{\rm c},\frac{\pi}2)$,
		there
		exists an admissible solution of the regular reflection problem
		in the sense of Definition {\rm \ref{admisSolnDef}}, which satisfies
		the properties stated in Theorem {\rm \ref{Regularity-RegRefl-Th1}}.
		
		If $\theta_{\rm w}^{\rm c}>\theta^{\rm d}_{\rm w}$, then, for the wedge angle $\theta_{\rm w}=\theta_{\rm w}^{\rm c}$, there
		exists an attached shock solution $\varphi$ with all the properties listed in
		Definition {\rm \ref{admisSolnDef}}
		and Theorem {\rm \ref{Regularity-RegRefl-Th1}(\ref{RegReflSol-Prop2-th})--(\ref{RegReflSol-Prop3-th})}
		except that $P_3=P_2$ $($we denote $P_3$ for that point below$)$.
		In addition, for the regularity of solution $\varphi$, we have
		\begin{itemize}
			\item For the supersonic case with $\theta_{\rm w}=\theta^{\rm c}_{\rm w}$,
			$$
			\varphi\in C^\infty(\overline{\Omega}\backslash(\overline{\Gso}\cup\{P_3\}))\cap C^{1,1}(\overline{\Omega}\backslash\{P_3\})\cap C^{0,1}(\overline{\Omega}),
			$$
			and the reflected shock $P_0P_1P_3$ is Lipschitz up to the endpoints,
			$C^{2,\beta}$ for any $\beta\in[0,\frac{1}{2})$ except point $P_3$,
			and $C^{\infty}$ except points $P_1$ and $P_3$.
			
			\smallskip
			\item For the subsonic case with $\theta_{\rm w}=\theta^{\rm c}_{\rm w}$,
			$$
			\varphi\in C^\infty(\overline{\Omega}\backslash\{P_0,\,P_3\})
			\cap C^{1,\beta}(\overline{\Omega}\backslash\{P_3\})\cap C^{0,1}(\overline{\Omega})
			$$
			for $\beta$ as in Theorem  {\rm \ref{Regularity-RegRefl-Th1}},
			and the reflected shock $P_0P_3$ is Lipschitz up to the endpoints, $C^{1,\beta}$ except
			point $P_3$, and $C^{\infty}$ except points $P_0$ and $P_3$.
		\end{itemize}
	\end{thm}
	\begin{proof}
		The existence of admissible solutions directly follows from Lemma \ref{admisEquiv1-2}
		and \cite[Theorem 2.6.9 and Remark 2.6.8]{cf-book2014shockreflection}, where we note
		that \cite[Remark 2.6.8]{cf-book2014shockreflection} applies to the case: $u_1> c_1$ as well,
		although this is not stated explicitly.
	\end{proof}

	Now we show that the admissible solutions
	satisfy the conditions of
	Theorems \ref{thm:main theorem}--\ref{thm:main theorem-strictUnif}.
	
	\begin{lem}\label{RegReflSatisfisCond-Lemma}
     The following statements hold{\rm :}
		\begin{enumerate}[\rm (i)]
			\item\label{RegReflSatisfisCond-Lemma-i1}
			Any admissible solution in the sense of Definition {\rm \ref{admisSolnDef}} satisfies the conditions
			of
			Theorems {\rm \ref{thm:main theorem}} and
			{\rm \ref{thm:main theorem-strictUnif}}.
			\item\label{RegReflSatisfisCond-Lemma-i2}
			Any regular reflection-diffraction solution in the sense of
			Definition {\rm \ref{def:weak solutionRegRefl}} with properties
			{\rm (\ref{RegReflSol-Prop0})}--\eqref{RegReflSol-Prop1-1-1} of
			Definition {\rm \ref{admisSolnDef}} and with shock $\Gsh$ being
			a strictly convex graph
			in the sense of \eqref{shock-graph-inMainThm}--\eqref{strictConvexity-degenerate-graph}
			satisfies
			property
			{\rm (\ref{RegReflSol-Prop2})} of
			Definition {\rm \ref{admisSolnDef}}.
		\end{enumerate}
	\end{lem}
	
	\begin{proof}
		We divide the proof into seven steps: Assertion (\ref{RegReflSatisfisCond-Lemma-i1}) is proved in Steps 1--6,
		while assertion \eqref{RegReflSatisfisCond-Lemma-i2} is proved in Step 7.
		
		\smallskip
		{\bf 1.}
		We use $\Lambda^{\rm ext}$, $\Gsh^{\rm ext}$, and $\varphi^{\rm ext}$
		defined before Definition \ref{admisSolnDef}.
		Combining the structure of equation (\ref{equ:potential flow equation})
		with the boundary conditions (\ref{BCfor RegRefl})
		on the negative $\xi_1$--axis yields that the reflected/extended function
		$\varphi^{\rm ext}$ is a weak solution of
		equation (\ref{equ:potential flow equation}) in $\Lambda^{\rm ext}$.
		By the boundary conditions (\ref{BCfor RegRefl}),
		state $(1)$ satisfies $\partial_{\bn}\varphi_1=0$ on the $\xi_1$--axis.
		Then the structure
		of the constant state (see \S \ref{subsec:potential flow rh condition})
		implies that $\varphi_1(-\xi_1, \xi_2)=\varphi_1(\xi_1, \xi_2)$
		in $\mR^2$ so that $\varphi_1^{\rm ext}=\varphi_1$.
		We also note the regularity of $\varphi^{\rm ext}$
		in Theorem \ref{Regularity-RegRefl-Th1}\eqref{RegReflSol-Prop0-th}.
		Thus, the extended shock $\Gsh^{\rm ext}$ separates the constant state
		$\varphi_1$ from the smooth solution $\varphi^{\rm ext}$ of
		equation (\ref{equ:potential flow equation})
		in $\Omega^{\rm ext}$, and the Rankine-Hugoniot conditions  (\ref{equ:boundary-RH-RegRefl})
		are satisfied for $\varphi^{\rm ext}$ and $\varphi_1$ on $\Gsh^{\rm ext}$.

		\smallskip
		{\bf 2.} Region $\Omega$ satisfies the conditions in Framework (A).
		Indeed,
		for the supersonic reflection case (see Fig. \ref{figure:shock refelction}),
		the required piecewise-regularity holds, since  $\Gw$ and $\Gamma_{\rm sym}$
		are straight segments, $\Gso$ is an arc of circle,
		and $\Gsh$ has the
		regularity stated in Theorem \ref{Regularity-RegRefl-Th1}(\ref{RegReflSol-Prop0-th}).
		The fact that all the angles of the corners of $\Omega$
		are less than $\pi$ is verified as follows:
		
		Consider first the supersonic case. Since curve $P_0P_1P_2$
		is $C^2$ at $P_1$, and $P_0P_1$ is a straight segment,
		we use that the center of sonic circle of state $(2)$ is on $\Gamma^0_{\rm wedge}$ and $P_0$ is outside that circle
		to conclude that the angle at $P_1$ is between $(\frac{\pi}{2}, \pi)$, and the angle
		at $P_4$ is $\frac{\pi}{2}$.
		Also, since \eqref{regularityOfRegReflSupers} shows that $\Gsh^{\rm ext}$ is smooth near $P_2$,
		it follows that the interior angle to $\Omega$ at $P_2$ is $\frac{\pi}{2}$.
		Finally, the angle at $P_3$ is $\pi-\theta\in (\frac{\pi}{2}, \pi)$.
		
		For the subsonic reflection case, the angles at $P_2$ and $P_3$ are handled similarly.
		The angle at $P_0$ is in $(0, \frac{\pi}{2})$ for the following reason:
		By \cite[Lemma 8.2.11, Proposition 15.2.1]{cf-book2014shockreflection},
		for any $\theta_{\rm w}\in(\theta_{\rm w}^{\rm d}, \frac\pi 2)$,
		$\bn_{\rm w}\in  Con(\ee_{S_1}, \ee_{\xi_2})$ so that,
		using the regularity of $\Gsh$ in Theorem \ref{ExistRegRefl-Th1}(\ref{RegReflSol-Prop0}),
		property (\ref{RegReflSol-Prop1-1}) in Theorem \ref{ExistRegRefl-Th1},
		and $\varphi=\varphi_1$ on $\Gsh$, we conclude that
		$\Gsh$ is a graph:
		$$
		\Gsh=\{(f(T), T)\; : \; T_{P_2}\le T\le T_{P_0}\}
		$$
		of a function $f(T)\in C^2([T_{P_2}, T_{P_0}))\cap C^{1,\beta}([T_{P_2}, T_{P_0}])$,
		where the $(S,T)$--coordinates are along
		the normal and tangent directions to $\Gw$.
		
		\smallskip
		{\bf 3.}
		The entropy condition (A1) of Theorem \ref{thm:main theorem}
		follows directly from
		property (\ref{RegReflSol-Prop1-1}) of Definition \ref{admisSolnDef}, where
		state $(0)$ in Theorem \ref{thm:main theorem} is state $(1)$ in the
		regular shock reflection problem.
		
		From the regularity of $\varphi$ and $\Gsh$ in
		Theorem \ref{Regularity-RegRefl-Th1}(\ref{RegReflSol-Prop0}),
		we see that conditions (A2) and (A4) of Theorem \ref{thm:main theorem} hold.
		
		Property (\ref{RegReflSol-Prop1}) of Definition \ref{admisSolnDef}
		implies that condition (A3) of Theorem \ref{thm:main theorem} holds.
		
		\smallskip
		{\bf 4.}
		Using the notations of the endpoints
		of $\Gsh$ as in Framework (A)
		by $A:={P_1}$ and $B:=P_2$,
		we see from the properties of Definition \ref{admisSolnDef}\eqref{RegReflSol-Prop0} that
		$$
		\bt_A=\ee_{S_1}, \qquad \bt_{B}=\ee_{\xi_2}.
		$$
		As we discussed in Step 2,  $\Gsh$ is orthogonal to the $\xi_1$--axis at $P_2$. From this
		and \cite[Lemma 7.5.12]{cf-book2014shockreflection}, $\ee_{\xi_1}\ne \pm \ee_{S_1}$.
		Also, combining
		property (\ref{RegReflSol-Prop2-th}) of Theorem \ref{Regularity-RegRefl-Th1} with the fact
		that
		$\Gsh$ is the level
		set $\varphi-\varphi_1=0$,
		we obtain that $\{P+Con\}\cap\Omega=\emptyset$ for all $P\in\Gsh$.
		Thus, condition (A5) of Theorem \ref{thm:main theorem} is satisfied.
		
		\smallskip
		{\bf 5.}
		Next, we discuss condition (A6) of Theorem \ref{thm:main theorem}.
		We
		recall that $\phi:=\varphi-\varphi_1$.
		All the local minima and maxima discussed below are
		relative to $\overline\Omega$.
		Also, we discuss the supersonic and subsonic/sonic cases
		together below, and use notations \eqref{subsonicConvenct} for the subsonic/sonic case.
		Furthermore, since conditions (A1)--(A5) have been verified, we can use Lemma \ref{lem:shock graph} in the argument below.
		
		Fix  $\ee=\bn_{\rm w}$, where $\bn_{\rm w}$ is
		defined in   Theorem \ref{Regularity-RegRefl-Th1}\eqref{RegReflSol-Prop3-th}.
		By Remark \ref{nuWinCone},
		$\ee\in Con$.  We first notice that,
		by Remark \ref{BConStraightBdrySegm},
		$w=\phi_\ee$
		satisfies equation (\ref{equ:phi e}), which is strictly elliptic in
		$\Omega\cup\Gsh^0\cup\overline{\Gamma_{\rm sym}}\cup\Gw^0$. Furthermore,
		since $\bt=\ee_{\xi_1}$ on $\Gamma_{\rm sym}$ so that
		$\ee\cdot\bt=-\sin\theta_{\rm w}\ne 0$  on $\Gamma_{\rm sym}$,
		then
		$w$
		satisfies \eqref{equ:boundary-ofOmega-RegRefl-w}
		on $\Gamma_{\rm sym}^0$, and this
		boundary condition is
		oblique.
		Thus, by Hopf's lemma,
		the local maximum and minimum of $\phi_\ee$
		relative to $\overline\Omega$ cannot be attained
		on $\Gamma_{\rm sym}^0$, unless $\phi_\ee$ is constant.
		
		We now show the similar property
		on $\overline\Gw\cup\overline{\Gso}$.
		From \eqref{bcWedgeState2} and (\ref{equ:boundary-ofOmega-RegRefl}),
		$\partial_{\ee}(\varphi-\varphi_2)=\partial_{\bn}(\varphi-\varphi_2)=0$ on $\overline\Gw$.
		Also,
		$D\varphi=D\varphi_2$ on
		$\overline{\Gso}$ by
		Definition \ref{admisSolnDef}(\ref{RegReflSol-Prop0}).
		Thus, $\partial_\ee(\varphi-\varphi_2)=0$ on
		$\overline\Gw\cup\overline{\Gso}$,
		which is the global maximum over
		$\overline\Omega$ by Theorem \ref{Regularity-RegRefl-Th1}(\ref{RegReflSol-Prop0-th}).
		Then
		$\partial_\ee(\varphi-\varphi_2)$ cannot attain its local minimum at some
		$P\in\overline\Gw\cup\overline{\Gso}$ unless $\partial_\ee(\varphi-\varphi_2)\equiv 0$ in $\Omega$.
		Indeed, if $P\in \overline\Gw\cup\overline{\Gso}$ is a point of local minimum of
		$\partial_\ee(\varphi-\varphi_2)$, then, since $P$ is
		also a point of
		global maximum and $\partial_\ee(\varphi-\varphi_2)(P)= 0$ as shown above,
		we obtain that
		$\partial_\ee(\varphi-\varphi_2)\equiv 0$ in $B_r(P)\cap\Omega$
		for some  $r>0$.
		Since $\partial_\ee(\varphi-\varphi_2)
		=\partial_\ee(\varphi-\varphi_1)+\partial_\ee(\varphi_1-\varphi_2)=\partial_\ee\phi+u_1\sin\theta_{\rm w}$
		(where we have used that $D\varphi_2\cdot\bn_{\rm w}=0$)
		so that $\partial_\ee(\varphi-\varphi_2)$ satisfies the strictly elliptic equation  (\ref{equ:phi e})
		in $\Omega$, the strong maximum principle implies that $\partial_\ee(\varphi-\varphi_2)\equiv 0$
		in
		$\Omega$. Recalling that $\partial_\ee\phi=\partial_\ee(\varphi-\varphi_2)-u_1\sin\theta_{\rm w}$,
		we conclude that $\phi_\ee\equiv -u_1\sin\theta_{\rm w}$ in $\overline\Omega$ if $\phi_\ee$
		attains its local minimum at some
		$P\in\overline\Gw\cup\overline{\Gso}$.
		
		Combining the two cases discussed above, we conclude that, if
		$\phi_\ee$ attains its local minimum at some point
		$P\in\overline\Gw\cup\overline{\Gso}\cup \Gamma_{\rm sym}^0$,
		then $\phi_\ee$  is constant in $\Omega$,
		specifically $\phi_\ee\equiv -u_1\sin\theta_{\rm w}=:a$.
		
		Now we show that, if $\phi_\ee\equiv a$ in $\Omega$ for an admissible solution $\varphi$, then $\varphi$
		is a uniform state in $\Omega$.
		To fix notations, we consider first the
		supersonic case. We use the $(S,T)$--coordinates with basis $\{\ee, \ee^\perp\}$ and the origin at $P_3$
		for $\ee^\perp$ determined as in Lemma \ref{lem:shock graph} for $\ee=\bn_{\rm w}$, {\it i.e.},
		$\ee^\perp=-(\cos\theta_{\rm w}, \sin\theta_{\rm w})$.
		We recall that $A=P_1$ and $B=P_2$; see Step 4.
		Then
		$T_B=T_{P_2}>T_{P_3}=0>T_{P_1}=T_A>T_{P_4}$. Also,
		$$
		\Gso=\{S=f_{\rm so}(T), \;T\in (T_{P_4},\,T_{P_1})\},\,
		\Gamma_{\rm sym}^0=\{S=T\tan\theta_{\rm w}, \;T\in (T_{P_3},\,T_{P_2})\},
		$$
		where $f_{\rm so}\in C^\infty((T_{P_4},\,T_{P_1}))$ and $f_{\rm so}>0$ on $(T_{P_4},\,T_{P_1})$.
		The function, $f_\ee$, from Lemma \ref{lem:shock graph}\eqref{lem:shock graph-i1}
		for $\ee=\bn_{\rm w}$ satisfies that $f_\ee(T)>\max(T\tan\theta_{\rm w}, 0)$ on $(T_{P_1},\,T_{P_2})$.  Also,
		\begin{align*}
		\Omega=\{(S, T): T\in (T_{P_4},\,T_{P_2}), \;\;
		\max(T\tan\theta_{\rm w}, \; 0)< S< \hat f(T)  \},
		\end{align*}
		where $\hat f\in C(T_{P_4},\,T_{P_2})$ satisfies
		$$
		\hat f=f_{\rm so}\;\;\mbox{on $(T_{P_4},\,T_{P_1})$}, \qquad
		\hat f=f_{\ee}\;\mbox{on $(T_{P_1},\,T_{P_2})$}.
		$$
		Let $\phi_\ee\equiv a$ in $\Omega$. Then, from the structure of $\Omega$ described above,
		$\phi(S,T)=aS+g(T)$ in $\Omega$ for some $g\in C^1(\mR)$.
		Since $\phi_{\xi_2}=0$ on
		$\Gamma_{\rm sym}$ by \eqref{equ:boundary-ofOmega-RegRefl-phi}, we see that
		$a\,\ee\cdot\ee_{\xi_2}+g'(T)\,\ee^\perp\cdot\ee_{\xi_2}=0$ for all
		$T\in (T_{P_3}, \;T_{P_2})$,
		where we have used the expression of $\Gamma_{\rm sym}^0$ in the $(S,T)$--coordinates
		given above.
		Note that
		$\ee^\perp\cdot\ee_{\xi_2}=-\sin\theta_{\rm w}\ne 0$.
		Thus, $g'(T)$ is constant
		on $(T_{P_3}, \;T_{P_2})$, which implies that
		$\phi(S,T)=aS+bT+c$ in $\hat\Omega$ for some $b, c\in\mR$, where
		$$
		\hat\Omega:=\{(S, T): T\in (T_{P_3},\,T_{P_2}), \;\;
		\max(T\tan\theta_{\rm w}, \; 0)< S< \hat f(T) \} \subset \Omega.
		$$
		Since $\phi$ is real analytic in $\Omega$ by Lemma \ref{lem:analyticity},
		it follows that
		$\phi(S,T)=aS+bT+c$ in $\Omega$. That is,
		$\varphi=\varphi_2+\phi$ is a constant state in $\Omega$,
		which contradicts Remark \ref{nonconsta-rmk}.
		
		For the subsonic/sonic case, the argument is the same,
		except that the structure of $\Omega$ now becomes
		\begin{align*}
		\Omega=\{(S, T): T\in (T_{P_0},\,T_{P_2}), \;\;
		\max(T\tan\theta_{\rm w}, \; 0)< S< f_\ee(T)  \}.
		\end{align*}
		Therefore, we have shown that $\phi$
		cannot attain its local minimum
		on $\overline\Gw\cup\overline{\Gso}\cup\Gamma_{\rm sym}^0$.
		
		Then we define $\Gamma_1:=(\overline \Gw\cup\overline\Gso\cup\Gamma_{\rm sym}^0)\setminus\{P_1\}=
		\partial\Omega\setminus\overline\Gsh$,
		and $\Gamma_2:=\emptyset$ in both the supersonic and subsonic/sonic cases.
		Clearly, $\Gamma_1$ is connected.
		Now Case (iii) of condition (A6) of Theorem \ref{thm:main theorem}
		holds in both the supersonic and
		subsonic/sonic cases.
		
		\smallskip
		{\bf 6.}
		We now check the conditions of Theorem  \ref{thm:main theorem-strictUnif}.
		Since the conditions of  Theorem \ref{thm:main theorem}
		have been checked, the conclusions of that theorem hold; in particular,
		$\varphi_{\bt\bt}\ge 0$ on $\Gsh$.
		
		Let $\hat\Gamma_0:=\Gso^0\cup\{P_4\}$ in the supersonic case, and
		$\hat\Gamma_0:=\emptyset$
		in the subsonic case.
		Let $\hat{\Gamma}_1:=\Gw^0\cup\{P_3\}$,
		$\hat{\Gamma}_2:=\Gamma_{\rm sym}^0$, and $\hat\Gamma_3:=\emptyset$.
		In the supersonic case, for any nonzero $\ee\in \mR^2$, $\phi_\ee =D(\varphi_2-\varphi_1)\cdot\ee$ on $\Gso$,
		{\it i.e.}, $\phi_\ee$ is constant on $\hat\Gamma_0$.
		Then (A7)--(A8) hold.
		
		Let $\ee\in \mR^2$ be a unit vector. We have shown in Step 5 that $\phi_\ee$ is not a constant in $\Omega$.
		Then,
		by (\ref{equ:boundary-ofOmega-RegRefl-w}),
		$\phi_\ee$ can attain a local minimum or maximum on $\Gw^0$ only if
		$\ee\cdot\bt_{\rm w}=0$, {\it i.e.}, $\ee=\pm\bn_{\rm w}$.
		In that case, by (\ref{equ:boundary-ofOmega-RegRefl-phi}),
		$\phi_\ee$ is constant on $\overline{\Gw}$.
		This verifies (A9) on $\hat{\Gamma}_1:=\Gw^0\cup\{P_3\}$.
		On $\hat{\Gamma}_2:=\Gamma_{\rm sym}^0$, (A9) is checked similarly.
		On $\hat\Gamma_0:=\Gso^0\cup\{P_4\}$ in the supersonic case, $D\varphi=D\varphi_2$ so that
		$\phi_\ee=(D\varphi_2-D\varphi_1)\cdot \ee =const$.
        Now (A9) is proved.
		
		To check condition (ii) of (A10) at point $B=P_2$,
		we note that, by Step 1, $\phi^{\rm ext}:=\varphi^{\rm ext}-\varphi_1$
		satisfies equation \eqref{equ:study} in $\Omega^{\rm ext}$ and
		conditions \eqref{equ:boundary} on $\Gsh^{\rm ext}$.
		Also, we have shown above that the original problem in $\Omega$ satisfies hypotheses
		(A1)--(A3) of Theorem \ref{thm:main theorem}. It follows that
		the problem for $\phi^{\rm ext}$ in $\Omega^{\rm ext}$ satisfies (A1)--(A3) of
		Theorem \ref{thm:main theorem}.
		
		Now it follows that the extended
		problem in $\Omega^{\rm ext}$ satisfies the
		conditions of
		Lemma \ref{lem:sign of second derivative}.
		Also, $P_2$ is an interior point of the extended
		shock $\Gsh^{\rm ext}$.
		Furthermore, using (\ref{regularityOfRegReflSupers}),
		we have
		$$
		\bn^{\rm ext}(P_2)=\bn_{\rm sh}(P_2),
		$$
		where $\bn_{\rm sh}(P_2)$ is defined in (A10).
		Since $\phi_{\bt\bt}\ge 0$ on $\Gsh$ as noted above, which implies that $\phi_{\bt\bt}(P_2)\ge 0$
		from the regularity of $\varphi$ in Theorem \ref{ExistRegRefl-Th1}, we apply
		Lemma \ref{lem:sign of second derivative} for the extended problem to conclude that,
		if $\bn_{\rm sh}(P_2)\cdot \ee<0$,
		then $\phi_\ee$ cannot attain its local maximum at $P_2$.
		If $\bn_{\rm sh}(P_2)\cdot \ee=0$, we use
		that $\bn_{\rm sh}(P_2)=\ee_{\xi_1}$
		by (\ref{regularityOfRegReflSupers}) to conclude
		that $\ee=\pm\ee_{\xi_2}$ in that case. Then we use the $C^1(\overline\Omega)$--regularity
		of $\phi$ to conclude that $\phi_\ee=0$ on $\overline\Gamma_{\rm sym}$
		by (\ref{equ:boundary-ofOmega-RegRefl-phi}).
		Thus, $\phi_\ee(P_2)=\phi_\ee(P_3)$ if $\bn(P_2)\cdot \ee=0$, so that condition (ii) of (A10) holds.
		
		\smallskip
		{\bf 7.}
		Now we show assertion \eqref{RegReflSatisfisCond-Lemma-i2}.
		Any admissible solution has a strictly convex shock by Theorems {\rm \ref{thm:main theorem}} and
		{\rm \ref{thm:main theorem-strictUnif}}, since we have verified the conditions of these theorems in Steps 1--6 of this proof.
		Then it remains to show
		that any regular reflection-diffraction solution in the sense of
		Definition \ref{def:weak solutionRegRefl}, with properties
		{\rm (\ref{RegReflSol-Prop0})}--\eqref{RegReflSol-Prop1-1-1} of
		Definition \ref{admisSolnDef} and with shock $\Gsh$ being
		a strictly convex graph
		in the sense of \eqref{shock-graph-inMainThm}--\eqref{strictConvexity-degenerate-graph},
		satisfies property
		{\rm (\ref{RegReflSol-Prop2})} of
		Definition \ref{admisSolnDef}.
		
		Recall that \eqref{shock-graph-inMainThm} holds in the present case with $A=P_1$ and $B=P_2$ as discussed
		in Steps 2 and 4. Then, using the properties of Definition \ref{admisSolnDef}\eqref{RegReflSol-Prop0},
		we find that, in the coordinates of \eqref{shock-graph-inMainThm},
		$$
		\ee_{\rm S_1}=\frac{(1, f'(T_A))}{|(1, f'(T_A))|}, \qquad
		\ee_{\xi_2}=-\frac{(1, f'(T_B))}{|(1, f'(T_B))|}.
		$$
		Also, from the strict concavity of $f$ in the sense of
		\eqref{strictConvexity-degenerate-graph}, we obtain that $f'(T_A)>f'(T)>f'(T_B)$
		and $f(T)<f(T_1)+f'(T_1)(T-T_1)$ for all
		$T, T_1\in (T_A, T_B)$. From this, we see that $\{P+Con\}\cap\Omega=\emptyset$
		for any $P\in\Gsh$. Then, since $\varphi\le\varphi_1$ in $\Omega$ from
		Definition \ref{admisSolnDef}\eqref{RegReflSol-Prop1-1-1}
		and $\varphi=\varphi_1$ on $\Gsh$ by \eqref{equ:boundary-RH-RegRefl}, we obtain that
		$\partial_\ee\varphi\ge\partial_\ee\varphi_1$ for any $\ee\in Con$,  which implies
		\eqref{nonstritConeMonot}.
	\end{proof}
	
	From Lemma \ref{RegReflSatisfisCond-Lemma}
	and Theorems \ref{thm:main theorem}--\ref{thm:main theorem-strictUnif},
	we have
	
	\begin{thm}
		If $\varphi$ is an admissible solution of the shock reflection-diffraction problem,
		then its shock curve $\Gsh$ is uniformly convex
		in the sense described in
		Theorem {\rm \ref{thm:main theorem-strictUnif}}.
		
		Furthermore, if a weak solution in the sense of Definition {\rm \ref{def:weak solutionRegRefl}} satisfies
		properties {\rm (\ref{RegReflSol-Prop0})}--\eqref{RegReflSol-Prop1-1-1} of
		Definition {\rm \ref{admisSolnDef}},
		then the transonic shock $\Gsh$ is
		a strictly convex graph
		in the sense of \eqref{shock-graph-inMainThm}--\eqref{strictConvexity-degenerate-graph} if and only if property \eqref{RegReflSol-Prop2} of
		Definition {\rm \ref{admisSolnDef}} holds.
	\end{thm}
	
	\begin{proof}
		The uniform convexity of $\Gsh$ for admissible solutions follows from
		Lemma \ref{RegReflSatisfisCond-Lemma} and Theorems {\rm \ref{thm:main theorem}} and
		{\rm \ref{thm:main theorem-strictUnif}}.
		
		Moreover, if a shock solution in the sense of Definition {\rm \ref{def:weak solutionRegRefl}}
		satisfies
		properties {\rm (\ref{RegReflSol-Prop0})}--\eqref{RegReflSol-Prop1-1-1} of
		Definition {\rm \ref{admisSolnDef}}, and its
		shock is a strictly convex graph, then, by Lemma \ref{RegReflSatisfisCond-Lemma}\eqref{RegReflSatisfisCond-Lemma-i2},
		the solution satisfies
		property
		{\rm (\ref{RegReflSol-Prop2})} of
		Definition \ref{admisSolnDef}.
	\end{proof}

	\subsection{\, Reflection problem for supersonic flows past a solid ramp}
	
	The second example is the Prandtl-Meyer reflection problem.
	This is a self-similar reflection that
	occurs when a two-dimensional supersonic flow with velocity $\vv_\infty=(u_\infty, 0)$, $u_\infty>0$,
	in the direction along the wedge axis hits the wedge at $t=0$.
	The slip boundary condition on the wedge boundary
	yields a self-similar reflection pattern;
	see Figs.
	\ref{figure:Prandtl Meyer Reflection}--\ref{figure:Prandtl Meyer Reflection-subs}.
	Also see Bae-Chen-Feldman \cite{baechenfeldman-prandtlmayerReflection,Bae-Chen-Feldman-2}.
	
	\begin{figure}[!ht]
		\centering
		\begin{minipage}{0.43\textwidth}
			\centering
			\includegraphics[width=0.61\textwidth]{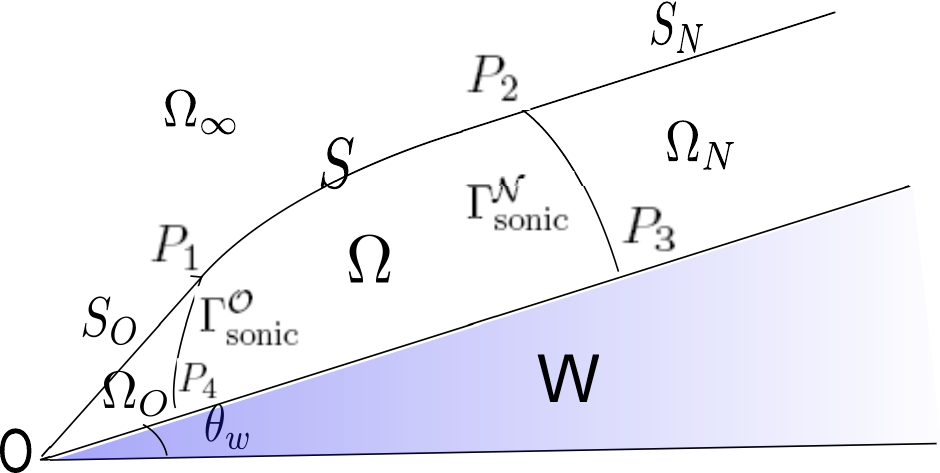}
			\caption{Supersonic Prandtl-Meyer $\quad$ reflection}
			\label{figure:Prandtl Meyer Reflection}
		\end{minipage}
		\hspace{0.1in}
		\begin{minipage}{0.43\textwidth}
			\vspace{-0.06in}
			\centering
			\includegraphics[width=0.61\textwidth]{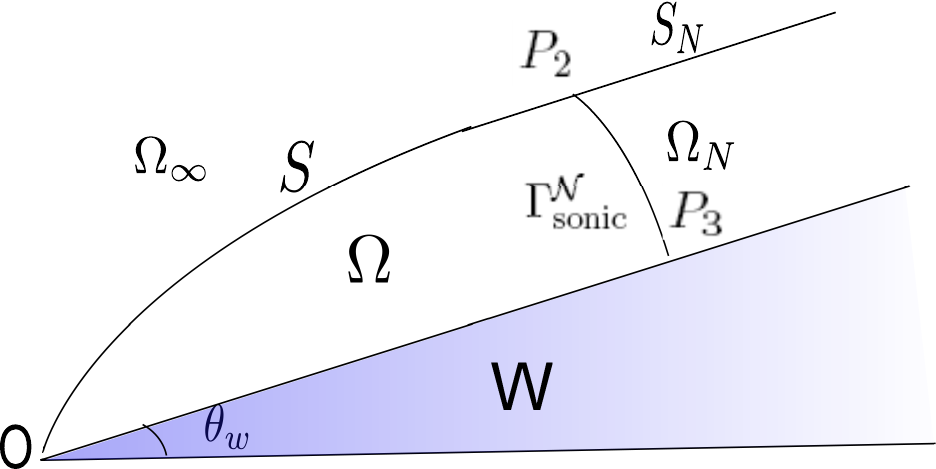}
			\vspace{0.04in}
			\caption{Subsonic Prandtl-Meyer $\quad$ reflection}
			\label{figure:Prandtl Meyer Reflection-subs}
		\end{minipage}
	\end{figure}
	
	We consider this problem in the self-similar coordinates.
	Using the symmetry with respect to the $\xi_1$--axis,
	the problem can be posed in the region:
	$$
	\Lambda=\mR^2_+\setminus\{\xxi \; :\;  \xi_2>\max(0, \xi_1\tan\theta_{\rm w})\}.
	$$
	Denote by $\varphi_\infty$ the pseudo-potential of the incoming state.
	
	\begin{defi}\label{def:weak solutionPrandtlRefl}
		$\varphi\in C^{0,1}(\overline\Lambda)$ is a weak solution
		of the Prandtl-Meyer reflection problem if $\varphi$ satisfies
		equation \eqref{equ:potential flow equation} in $\Lambda$,
		the boundary conditions \eqref{BCfor RegRefl} on $\partial\Lambda$,
		and the asymptotic conditions{\rm :}
		$$
		\displaystyle
		\lim_{R\to\infty}\|\varphi-\varphi_\infty\|_{C(L_\theta\setminus B_R(\mathbf{0}))}=0
		$$
		along ray $L_\theta:=\{\xi_1=\xi_2\cot\theta, \xi_2>0\}$ for each
		$\theta\in (\theta_{\rm w}, \pi)$ in the weak sense
		{\rm (}as in Definition {\rm \ref{def:weak solutionRegRefl}}{\rm )}.
	\end{defi}
	
	We consider the solutions with the structure shown in Figs.
	\ref{figure:Prandtl Meyer Reflection}--\ref{figure:Prandtl Meyer Reflection-subs}.
	These solutions are piecewise-smooth and
	equal to the constant states outside region $\Omega$ described below.
	
	The constant states are defined as follows
	(see \cite[\S 2]{Bae-Chen-Feldman-2} for the details and proofs): Given the constant self-similar state
	with velocity $\vv_\infty=(u_\infty, 0)$ and density $\rho_\infty$ which is supersonic
	at the origin (the wedge vertex), there exist the
	detachment wedge angle $\theta_{\rm w}^{\rm d}\in (0, \frac\pi 2)$ and
	the sonic wedge angle
	$\theta_{\rm w}^{\rm s}\in (0, \theta_{\rm w}^{\rm d})$,
	which depend only on $(\rho_\infty, u_\infty)$,
	such that, for any wedge angle $\theta_{\rm w}\in (0,\theta_{\rm w}^{\rm d})$,
	
	\smallskip
	\begin{enumerate}[\rm (i)]
		\item There exists a unique constant state $\varphi_\mN$, which determines
		the normal reflection state of $\varphi_\infty$
		from the wedge boundary $\partial W:=\{\xi_1>0, \xi_2=\xi_1\tan\theta_{\rm w}\}$; that is,
		$\varphi_\mN$ satisfies that $\partial_{\bn} \varphi_\mN=0$
		on $\partial W$, half-line
		$S_\mN:=\{\xxi\; :\;\varphi_\mN=\varphi_\infty\}\cap\{\xi_2>0\}$
		lies in $\Lambda$ and is parallel to $\partial W$,
		and the Rankine-Hugoniot condition holds on $S_\mN$:
		$$
		\rho_\infty\partial_{\bn}\varphi_\infty=\rho_\mN\partial_{\bn}\varphi_\mN
		\qquad\,\, \mbox{on $S_\mN$}.
		$$
		
		\item There exists a constant state $\varphi_\mO$ such that
		$\partial_{\bn} \varphi_\mO=0$ on $\partial W$,
		half-line $S_\mO:=\{\xxi\; :\;\varphi_\mO=\varphi_\infty\}\cap\{\xi_2>0\}$
		lies in $\Lambda$, the wedge vertex is on $\overline{S_\mO}$ ({\it i.e.},
		$\mathbf{0}\in \overline{S_\mO}$),
		and the Rankine-Hugoniot condition holds on $S_\mO$:
		$$
		\rho_\infty\partial_{\bn}\varphi_\infty=\rho_\mO\partial_{\bn}\varphi_\mO\qquad\,\,
		\mbox{on $S_\mO$}.
		$$
		In fact, there exist two states for $\varphi_\mO$, weak and strong,
		and we always choose the weak one with the smaller density
		{\rm (}so that the unique state $\varphi_\mO$ is often referred{\rm )}.
		
		\smallskip
		\item
		$\varphi_\mO$ is supersonic (resp. subsonic) at the origin for all
		$\theta_{\rm w}\in (0, \theta_{\rm w}^{\rm s})$
		(resp. $\theta_{\rm w}\in (\theta_{\rm w}^{\rm s}, \theta_{\rm w}^{\rm d}))$.
		This determines the supersonic and subsonic Prandtl-Meyer reflection configurations below.
	\end{enumerate}

	Next, we define the points, lines, and regions in Figs.
	\ref{figure:Prandtl Meyer Reflection}--\ref{figure:Prandtl Meyer Reflection-subs}
	for a given wedge angle $\theta_{\rm w}\in (0,\theta_{\rm w}^{\rm d})$
	as follows:
	\begin{enumerate}[\rm (a)]
		\item The sonic arcs $\Gso^\mN$ and $\Gso^\mO$ are the arcs (defined below)
		of the sonic circles of the constant states
		$\varphi_\mN$ and $\varphi_\mO$, respectively, with the centers on $\partial W$,
		since these states satisfy the slip boundary condition on $\partial W$:
		\begin{itemize}
			
			\smallskip
			\item
			$\Gso^\mN$ is the {\it upper} arc of
			$\partial B_{c_\mN}(O_\mN)$ between lines $\partial W$ and $S_\mN$. It follows
			that $\Gso^\mN\subset\Lambda$, since
			$\partial B_{c_\mN}(O_\mN)$
			intersects the full line $S_\mN$ at two points.
			Denote the endpoints of
			$\Gso^\mN$ by $P_2$ and $P_3$,
			which lie on $S_\mN$ and $\partial W$, respectively.
			
			\smallskip
			\item Arc $\Gso^\mO$ is defined only for the supersonic reflection configurations,
			{\it i.e.}, for $\theta_{\rm w}\in (0,\theta_{\rm w}^{\rm s})$.
			In this case,
			$\partial B_{c_\mO}(O_\mO)$
			intersects half-line $S_\mO$ at two points within $\Lambda$,
			and $\Gso^\mO$ is the {\it lower} arc of
			$\partial B_{c_\mO}(O_\mO)$ between lines $\partial W$ and $S_\mO$.
			Then $\Gso^\mO\subset\Lambda$.
			Denote  the endpoints of
			$\Gso^\mO$ by $P_1$ and $P_4$,
			which lie on $S_\mO$ and $\partial W$, respectively.
			\item For the supersonic configurations, $S_{\mO, {\rm seg}}$ is segment $OP_1$.
			Note that
			$S_{\mO, {\rm seg}} \subset S_\mO$.
			\item $S_{\mN, {\rm seg}}$ is the portion of $S_\mN$ with the left endpoint $P_2$,
			{\it i.e.},
			$S_{\mN, {\rm seg}} = S_\mN\cap\{\xi_1>\xi_{1, P_2}\}$.
		\end{itemize}
		
		\smallskip
		\item $\Gw$ is the segment of $\partial W$ between points $P_3$ and $P_4$
		for the supersonic case
		(resp. between $\boldsymbol{0}$ and $P_3$ for the subsonic case).
		
		\smallskip
		\item
		There exists a smooth shock curve $\Gsh$ with the following properties:
		\begin{itemize}
			
			\smallskip
			\item For the supersonic reflection configurations, $\Gsh$ has endpoints
			$P_1$ and $P_2$;
			
			\smallskip
			\item For the subsonic reflection configurations, $\Gsh$ has endpoints
			$P_2$ and $O$;
			
			\smallskip
			\item $\Gsh$, $\Gso^\mN$,  $\Gw$, and $\Gso^\mO$ do not have common
			points except at their end points.
		\end{itemize}
		
		\smallskip
		\item $\Omega$ is the domain bounded by the curve formed by
		$\overline{\Gsh}$, $\overline{\Gso^\mN}$,   $\overline{\Gw}$, and $\overline{\Gso^\mO}$.
		
		\smallskip
		\item For the supersonic reflection configurations,
		$\Omega_\mO$ is the region bounded by arc
		$\overline{\Gso^\mO}$ and the straight segments $\overline{OP_1}$ and $\overline{OP_4}$.
		
		\smallskip
		\item $\Omega_\mN$ is the unbounded region with the boundary consisting of
		arc $\overline{\Gso^\mN}$, and the straight half-lines $\partial W\cap\{\xi_1\ge \xi_{1P_3}\}$
		and $S_\mN\cap \{\xi_1\ge \xi_{1P_2}\}$.
		
		\smallskip
		\item $\Omega_\infty:=\Lambda\setminus\overline{\Omega_\mO\cup\Omega\cup\Omega_\mN}$
		for the supersonic case, and
		$\Omega_\infty:=\Lambda\setminus\overline{\Omega\cup\Omega_\mN}$
		for the subsonic case.
	\end{enumerate}
	
	Now we define a class of solutions of the Prandtl-Meyer reflection problem
	with the structure as in Figs.
	\ref{figure:Prandtl Meyer Reflection}--\ref{figure:Prandtl Meyer Reflection-subs}.
	
	\begin{defi}\label{Prandtl-admisSolnDef}
		Let $(\rho_{\infty}, (u_{\infty},0))$ be a supersonic state
		in $\Omega_{\infty}$, and let $\theta_{\rm w}^{\rm d}$ and $\theta_{\rm w}^{\rm s}$ be the corresponding
		detachment and sonic angles.
		Let $\theta_{\rm w}\in (\theta_{\rm w}^{\rm d},\frac{\pi}2)$.
		A function $\varphi\in C^{0,1}(\overline\Lambda)$ is an admissible solution of the Prandtl-Meyer reflection problem
		if $\varphi$ is a solution in the sense of Definition {\rm \ref{def:weak solutionPrandtlRefl}}
		and satisfies the following properties{\rm :}
		\begin{enumerate}[\rm (i)]
			\item \label{PrandtlAdmSolnProperties-2} The structure of solution is the following{\rm :}
			
			\smallskip
			\begin{itemize}
				\item
				If $\theta_{\rm w}\in(0, \theta_{\rm w}^{\rm s})$, then the solution
				is of supersonic reflection configuration as in
				Fig. {\rm \ref{figure:Prandtl Meyer Reflection}}
				and satisfies
				\begin{align}
				& \varphi\in C^{1}(\overline\Lambda\setminus\overline{S_{\mO, {\rm seg}}\cup\Gsh\cup S_{\mN, {\rm seg}}}),\label{states123-supers-Prandtl-a}\\
				&\varphi\in C^3(\Omega)\cap C^2(\overline{\Omega}\setminus\overline{\Gso^\mO\cup\Gso^\mN})
				\cap C^1(\overline{\Omega}),\label{states123-supers-Prandtl-b}\\
				\label{states123-supers-Prandtl}
				&\varphi=\begin{cases}
				\,\varphi_\infty \qquad&\mbox{in $\Omega_\infty$},\\[0.3mm]
				\,\varphi_\mO &\mbox{in $\Omega_\mO$},\\[0.3mm]
				\,\varphi_\mN &\mbox{in $\Omega_\mN$}{\rm ;}
				\end{cases}
				\end{align}
				
				\smallskip
				\item
				If $\theta_{\rm w}\in[\theta_{\rm w}^{\rm s}, \theta_{\rm w}^{\rm d})$, then the solution
				is of subsonic reflection configuration as in Fig. {\rm \ref{figure:Prandtl Meyer Reflection-subs}}
				and satisfies
				\begin{align}
				& \varphi\in C^{1}(\overline\Lambda\setminus\overline{\Gsh\cup S_{\mN, {\rm seg}}}),\label{states123-subs-Prandtl-a}\\
				&\varphi\in C^3(\Omega)\cap C^2(\overline{\Omega}\setminus(\{O\}\cup\overline{\Gso^\mN}))
				\cap C^1(\overline{\Omega}), \label{states123-subs-Prandtl}\\
				&\varphi=\begin{cases}
				\,\varphi_\infty \qquad&\mbox{in $\Omega_\infty$},\\[0.3mm]
				\,\varphi_\mO(O) &\mbox{at $O$},\\[0.3mm]
				\,\varphi_\mN &\mbox{in $\Omega_\mN$}{\rm ,}
				\end{cases}
				\label{states123-subsOO-Prandtl}\\[1mm]
				&D\varphi(O)=D\varphi_\mO(O).
				\end{align}
			\end{itemize}
			
			\smallskip
			\item\label{PrandtlAdmSolnProperties-1}
			The shock curve $\Gsh$ is $C^2$ in its relative interior{\rm .}
			
			\smallskip
			\item\label{PrandtlSolnProperties-3}
			Equation \eqref{equ:potential flow equation} is strictly elliptic
			in $\overline{\Omega}\backslash(\overline{\Gso^\mN}\cup\overline{\Gso^\mO})$ for the supersonic case
			and in $\overline{\Omega}\backslash(\overline{\Gso^\mN}\cup\{O\})$ for the subsonic case{\rm .}

			\smallskip
			\item\label{PrandtlSolnProperties-3-1}
			$\partial_{\bn}\varphi_\infty>\partial_{\bn}\varphi>0$ on $\Gsh$, where $\bn$ is the normal vector
			to $\Gsh$ pointing into $\Omega$.
			
			\smallskip
			\item\label{PrandtlAdmSolnProperties-4}
			$\partial_{\ee_{S_\mO}}(\varphi_{\infty}-\varphi)\leq0$ and
			$\partial_{\ee_{S_\mN}}(\varphi_{\infty}-\varphi)\leq0$ in $\Omega$, where
			vectors $\ee_{S_\mO}$ and $\ee_{S_\mN}$ are parallel to lines $S_\mO$ and $S_\mN$, respectively,
			oriented towards the interior of $\Gsh$ from points $P_1$ and $P_2$, respectively{\rm ;}
		\end{enumerate}
	\end{defi}

	\begin{rem}\label{eqnInOmegaPrantl-rmk}
		A version of Remark {\rm \ref{eqnInOmega-rmk}}
		holds in the present case, with the only difference that
		the potential function of the incoming state is $\varphi_\infty$ here,
		instead of $\varphi_1$.
	\end{rem}
	
	\begin{rem}\label{nonconsta-Prandtl-rmk}
		$\varphi$ in $\Omega$ is not a constant state. Indeed,
		if $\varphi$ is a constant state in $\Omega$,
		then $\varphi=\varphi_\mN$ in $\Omega$,
		which follows from
		\eqref{states123-supers-Prandtl} and \eqref{states123-subsOO-Prandtl}
		in the supersonic and subsonic cases, respectively.
		On the other hand, we obtain
		that $\varphi=\varphi_\mO$ in $\Omega$,
		which follows from both \eqref{states123-supers-Prandtl}
		for the supersonic case {\rm (}since
		$\varphi$ is $C^{1}$ across $\Gso^\mO${\rm )}
		and the property that
		$(\varphi, D\varphi)=(\varphi_2, D\varphi_2)$ at $O$
		for the subsonic case.
		However, $\varphi_\mO$ and $\varphi_\mN$ are
		two different states,
		which can be seen from their definitions, since
		line $S_\mN$ is parallel to $\partial W$ {\rm (}so that
		these lines do not coincide{\rm )}, while
		$S_\mO$ intersects $\partial W$ at point $O$.
	\end{rem}

	\begin{lem}\label{admisEquiv1-2-Prandtl}
		Definition {\rm \ref{Prandtl-admisSolnDef}} is equivalent to the definition of admissible
		solutions in {\rm \cite{Bae-Chen-Feldman-2}}{\rm ;} see Definition {\rm 2.14} there.
	\end{lem}
	The proof of Lemma \ref{admisEquiv1-2-Prandtl} follows closely the proof of Lemma \ref{admisEquiv1-2} with mostly notational changes, so we skip this proof  here.
	
	From Lemma \ref{admisEquiv1-2-Prandtl}, the results of \cite{baechenfeldman-prandtlmayerReflection,Bae-Chen-Feldman-2}
	for the existence and properties of
	admissible solutions apply to the solutions in the sense of Definition {\rm \ref{admisSolnDef}}.
	We list some of these properties in the following theorem.
	
	\begin{thm}\label{PrandtlReflPropertiesTh}
		Let $(\rho_{\infty}, (u_{\infty},0))$ be a supersonic state
		in $\Omega_{\infty}$, and let $\theta_{\rm w}^{\rm d}$ and $\theta_{\rm w}^{\rm s}$ be the corresponding
		detachment and sonic angles.
		Then  any admissible solution of the Prandtl-Meyer reflection problem
		with wedge angle $\theta_{\rm w}\in (0, \theta_{\rm w}^{\rm d})$ has the following properties{\rm :}
		\begin{enumerate}[\rm (i)]
			\item \label{PrandtlSolnProperties-2} Additional regularity{\rm :}
			
			\smallskip
			\begin{itemize}
				\item
				If $\theta_{\rm w}\in(0, \theta_{\rm w}^{\rm s})$, i.e., when the solution
				is of supersonic reflection configuration as in
				Fig. {\rm \ref{figure:Prandtl Meyer Reflection}},
				then
				$\varphi\in C^{1,1}(\overline{\Omega_\mO\cup\Omega\cup\Omega_\mN})$ and
				$\varphi\in C^\infty(\overline{\Omega}\setminus\overline{\Gso^\mO\cup\Gso^\mN})${\rm ;}
				
				\smallskip
				\item
				If $\theta_{\rm w}\in[\theta_{\rm w}^{\rm s}, \theta_{\rm w}^{\rm d})$, i.e., when the solution
				is of subsonic reflection configuration as in Fig. {\rm \ref{figure:Prandtl Meyer Reflection-subs}},
				then $\varphi\in C^{1,\alpha}(\overline{\Omega\cup\Omega_\mN})\cap C^{1,1}(\overline{\Omega\cup\Omega_\mN}\setminus\{O\})$ and
				$\varphi\in C^\infty(\overline{\Omega}\setminus(\{O\}\cup\overline{\Gso^\mN}))
				$ for some $\alpha\in (0, 1)$, depending on $(\rho_\infty, u_\infty, \theta_{\rm w})$
				and non-increasing with respect to $\theta_{\rm w}$.
			\end{itemize}
			
			\smallskip
			\item\label{PrandtlSolnProperties-1}
			The shock curve $\Gsh$ is $C^\infty$ in its relative interior.
			
			\item\label{PrandtlSolnProperties-1-1}
			$\Gsh$ has the following regularity up to the endpoints{\rm :}
			In the supersonic case, the whole shock curve $\overline{S_{\mO, {\rm seg}}\cup\Gsh\cup S_{\mN, {\rm seg}}}$ is
			$C^{2,\beta}$ for any $\beta\in(0, 1)$. In the subsonic case,
			curve
			$\overline{\Gsh\cup S_{\mN, {\rm seg}}}$
			is  $C^{1, \alpha}$ with $\alpha$ as in \eqref{PrandtlSolnProperties-2}{\rm .}

			\smallskip
			\item\label{PrandtlSolnProperties-4}
			For each $\mathbf{e}\in Con(\ee_{S_\mO}, \ee_{S_\mN})$,
			\begin{equation}\label{strictCone-regrefl-Prandtl}
			\partial_{\mathbf{e}}(\varphi_\infty-\varphi)< 0 \qquad\,\,\mbox{in $\overline\Omega$},
			\end{equation}
			where vectors   $\ee_{S_\mO}$ and $\ee_{S_\mN}$
			are
			introduced in  Definition {\rm \ref{Prandtl-admisSolnDef}}\eqref{PrandtlAdmSolnProperties-4}, and
			notation \eqref{3.5a} has been used.
			
			\smallskip
			\item\label{PrandtlSolnProperties-5}
			Denote by $\bn_{\rm w}$ the interior
			unit  normal vector to $\Gw$ pointing into $\Omega$, {\it i.e.},
			$\bn_{\rm w}=(-\sin\theta_{\rm w}, \cos\theta_{\rm w})$.
			Then
			$$
            \partial_{\bn_{\rm w}}(\varphi-\varphi_{\mathcal O})\leq0, \quad
			\partial_{\bn_{\rm w}}(\varphi-\varphi_{\mathcal N})\leq0 \qquad\mbox{in $\overline\Omega$}.
            $$
		\end{enumerate}
	\end{thm}
	
	\begin{proof}
		Properties (\ref{PrandtlSolnProperties-2})--(\ref{PrandtlSolnProperties-1-1})
		are from \cite[Theorem 2.16]{Bae-Chen-Feldman-2}.
		Properties (\ref{PrandtlSolnProperties-4})  and
		(\ref{PrandtlSolnProperties-5})
		are shown in \cite[Lemmas 3.2 and 3.6]{Bae-Chen-Feldman-2},
		where the results are stated in a rotated coordinate system, in which
		the $\xi_2$--variable is in the direction of $\bn_{\rm w}$.
	\end{proof}
	
	\begin{thm}\label{PrandtlReflExistTh}
		Let $(\rho_{\infty}, (u_{\infty},0))$ be a supersonic state
		in $\Omega_{\infty}$, and let $\theta_{\rm w}^{\rm d}$ be the corresponding
		detachment angle. Then,  for any $\theta_{\rm w}\in(0,\theta_{\rm w}^{\rm d})$,
		there exists an admissible solution of the Prandtl-Meyer reflection problem.
	\end{thm}
	The existence of solutions follows from \cite[Theorem 2.15]{Bae-Chen-Feldman-2}.
	
	Now, similar to Lemma \ref{RegReflSatisfisCond-Lemma}, we have
	
	\begin{lem}\label{RegReflSatisfisCond-Lemma-Prandtl}
	 The following statements hold{\rm :}	
     \begin{enumerate}[\rm (i)]
			\item\label{RegReflSatisfisCond-Lemma-Prandtl-i1}
			Any admissible solution in the sense of Definition {\rm \ref{Prandtl-admisSolnDef}} satisfies the conditions
			of
			Theorems {\rm \ref{thm:main theorem}} and
			{\rm \ref{thm:main theorem-strictUnif}}.

			\item\label{RegReflSatisfisCond-Lemma-Prandtl-i2}
			Any global weak solution
			of the Prandtl-Meyer reflection problem in the sense of
			Definition {\rm \ref{def:weak solutionPrandtlRefl}} with properties
			{\rm (\ref{PrandtlAdmSolnProperties-2})}--\eqref{PrandtlSolnProperties-3-1} of
			Definition {\rm \ref{Prandtl-admisSolnDef}} and with shock $\Gsh$ being
			a strictly convex graph
			in the sense of \eqref{shock-graph-inMainThm}--\eqref{strictConvexity-degenerate-graph}
			satisfies
			property
			{\rm (\ref{PrandtlAdmSolnProperties-4})} of
			Definition {\rm \ref{Prandtl-admisSolnDef}}.
		\end{enumerate}
	\end{lem}
	
	\begin{proof}
		We first discuss the proof of assertion (\ref{RegReflSatisfisCond-Lemma-Prandtl-i1}).
		
		Conditions (A1)--(A5) follow directly as in Lemma \ref{RegReflSatisfisCond-Lemma}.
		In particular, in (A5), $Con=Con(\ee_{S_\mO},\ee_{S_\mN})$,
		where we have used (\ref{3.5a}). Also, $A=P_1$ and $B=P_2$, where $P_1:=O$ in the
		subsonic/sonic case.
		
		For condition (A6), we choose $\ee=\bn_{\rm w}$, where $\bn_{\rm w}$ is
		defined in   Theorem \ref{PrandtlReflPropertiesTh}(\ref{PrandtlSolnProperties-5}).
		Then $\ee\in Con$, which can be seen from the fact that $u_\mO>0$ and $v_\mO>0$
		with $\frac{v_\mO}{u_\mO}>\tan\theta_{\rm w}$,
		and  $\ee_{\mN}=-(\cos\theta_{\rm w}, \sin\theta_{\rm w})$.
		
		\smallskip
		In the argument below, the local extrema are relative to $\overline\Omega$.
		Also, we discuss the supersonic and subsonic/sonic cases together and
		define $\overline{\Gso^\mO}:=\{O\}$ and $P_1:=O$ for the subsonic/sonic case.
		
		Recall the boundary conditions:
		$$
		\partial_{\bn}\varphi=0, \quad \partial_{\bn}\varphi_\mO=0,  \quad \partial_{\bn}\varphi_\mN=0
		\qquad\;\; \mbox{on $\overline\Gw$}.
		$$
		Also,
		$D\varphi=D\varphi_\mO$ on
		$\overline{\Gso^\mO}$ by
		Definition \ref{Prandtl-admisSolnDef}(\ref{PrandtlAdmSolnProperties-2}).
		Thus, $\partial_\ee(\varphi-\varphi_\mO)=0$ on
		$\overline\Gw\cup\overline{\Gso^\mO}$,
		which is the global maximum over
		$\overline\Omega$ by Theorem \ref{PrandtlReflPropertiesTh}(\ref{PrandtlSolnProperties-5}).
		Since $\varphi$ is not a constant state,
		arguing as in Step 5 of the
		proof of Lemma \ref{RegReflSatisfisCond-Lemma}, we find that, if
		$\phi_\ee$ attains its local minimum
		on $\overline\Gw\cup\overline{\Gso^\mO}$, then $\phi_\ee$ is constant in $\Omega$.
		Similarly,
		using that $D\varphi=D\varphi_\mN$
		on $\overline{\Gso^\mN}$ and arguing as above,
		we conclude that $\phi_\ee$ cannot attain its local minimum
		on $\overline{\Gso^\mN}$, unless $\phi_\ee$ is constant in $\Omega$.
		Combining all the facts together, we conclude that, if
		$\phi_\ee$ attains its local minimum
		on $\overline{\Gso^\mN}\cup\overline\Gw\cup\overline{\Gso^\mO}$,
		then $\phi_\ee$ is constant in $\Omega$.
		
		We now show that, if $\phi_\ee$ is constant in $\Omega$, then $\varphi$ is a constant
		state in $\Omega$. To fix notations, we consider first the
		supersonic case.
		Since conditions (A1)--(A5) have been verified, we can apply Lemma \ref{lem:shock graph}.
		We work in the $(S,T)$--coordinates with basis $\{\ee, \ee^\perp\}$ and origin $O$,
		where the orientation of $\ee^\perp$ is as in Lemma \ref{lem:shock graph}.
		Then
		$T_{P_4}<T_{P_1}<T_{P_2}<T_{P_3}$, where we have used that $A=P_1$ and $B=P_2$. Also,
		$$
		\Gso^\mO=\{S=f_{\mO}(T), \;T\in (T_{P_4},\,T_{P_1})\},\,\,
		\Gso^\mN=\{S=f_{\mN}(T), \;T\in (T_{P_2},\,T_{P_3})\},
		$$
		where $f_{\mO}\in C^\infty((T_{P_4},\,T_{P_1}))$
		and $f_{\mN}\in C^\infty((T_{P_2},\,T_{P_3}))$ are positive. With this, we obtain
		\begin{align*}
		\Omega=\{(S, T)\,:\, T\in (T_{P_4},\,T_{P_3}), \;\;
		S\in (0,\, \hat f(T) ) \},
		\end{align*}
		where $\hat f\in C(T_{P_4},\,T_{P_3})$ satisfies
		\begin{equation}\label{domainInPrandtlCoord}
		\hat f=f_{\mO}\;\mbox{on $(T_{P_4},\,T_{P_1})$},  \,\,
		\hat f=f_{\ee}\;\mbox{on $(T_{P_1},\,T_{P_2})$}, \,\,
		\hat f=f_{\mN}\;\mbox{on $(T_{P_2},\,T_{P_3})$}.
		\end{equation}
		Let $\phi_\ee\equiv a$ in $\Omega$. Then, from the structure of $\Omega$ described above,
		$\phi(S,T)=aS+g(T)$ in $\Omega$ for some $g\in C^1(\mR)$.
		Then, noting that $D\varphi=D\varphi_{\mN}$ on $\Gso^{\mN}$, we obtain
		$$
		g'(T)=\partial_{\ee^\perp}\phi(f_{\mN}(T), T)=D(\varphi_{\mN}-\varphi_\infty)\cdot\ee^\perp
		\qquad \mbox{for all $T\in (T_{P_2},\,T_{P_3})$},
		$$
		where we have used that $D(\varphi_{\mN}-\varphi_\infty)$ is a constant vector.
		Thus, $g'(T)$ is constant on $(T_{P_2},\,T_{P_3})$ so that
		$\phi(S,T)=aS+bT+c$  in $\hat\Omega:=\{(S, T)\,:\, T\in (T_{P_2},\,T_{P_3}), \; S\in (0,\, f_{\mN}(T) ) \} \subset \Omega$.
		Then, arguing as in Step 5 of the
		proof of Lemma \ref{RegReflSatisfisCond-Lemma}, we conclude that $\varphi$ is a constant state
		in $\Omega$, which is a contradiction. For the subsonic/sonic case, the argument is the same,
		except the structure of $\Omega$, where now $P_4=P_1=O$, and
		$f_{\mO}$ is not present in \eqref{domainInPrandtlCoord}.

		Therefore, Case (iii) of (A6) holds with  $\ee=\bn_{\rm w}$,
		$\Gamma_1:=\Gso^\mN\cup\overline\Gw\cup\Gso^\mO$
		for the supersonic case (resp.
		$\Gamma_1:=\Gso^\mN\cup\overline\Gw\setminus\{P_2\}$
		for the subsonic case),
		and $\Gamma_2=\emptyset$.
		
		\smallskip
		We now show that conditions (A7)--(A10) are satisfied with
		$\hat\Gamma_0=\overline{\Gso^\mN}\setminus\{P_2\}$,
		$\hat\Gamma_1=\Gw^0$,
		$\hat\Gamma_2=\emptyset$,
		and $\hat\Gamma_3=\overline{\Gso^\mO}\setminus\{P_1\}$
		for the supersonic case
		(resp. $\hat\Gamma_3=\emptyset$ for the subsonic case).
		Indeed, then (A7) clearly holds.
		Also, (A8) holds since $D\varphi=D\varphi_\mN$
		on $\overline{\Gso^\mN}$,  and
		$D\varphi=D\varphi_\mO$ on $\overline{\Gso^\mO}$ for the supersonic case.
		
		Condition (A9) on $\hat\Gamma_1=\Gw^0$ can be checked as follows:
		If $\ee \cdot \bt\ne 0$ on $\Gw^0$, then the argument of Step 5 in the proof
		of Lemma \ref{RegReflSatisfisCond-Lemma} applies here to yield that
		$\phi_\ee$ cannot attain the local minima or maxima on $\Gw^0$.
		In the other case,
		when $\ee=\pm\bn_{\rm w}$, we use the boundary condition:
		$$
		\partial_{\bn}\varphi=0\qquad \mbox{on $\Gw^0$}
		$$
		to derive that $\partial_{\bn}\phi=-u_\infty\sin\theta_{\rm w}$ on $\Gw$,
		similar to (\ref{equ:boundary-ofOmega-RegRefl-phi}).
		Also, on $\hat\Gamma_0=\overline{\Gso^\mN}\setminus\{P_2\}$, $D\varphi=D\varphi_\mN$
		so that
		$\phi_{\ee}=\partial_{\ee}(\varphi_\mN-\varphi_\infty)=const$.
		The argument on $\hat\Gamma_3=\overline{\Gso^\mO}\setminus\{P_1\}$ in the supersonic case
		is similar.
		This verifies (A9).
		Case (i) of (A10) clearly holds here.
		
		To prove assertion (\ref{RegReflSatisfisCond-Lemma-Prandtl-i2}),
		we follow directly the argument
		of Step 7 in the proof of Lemma \ref{RegReflSatisfisCond-Lemma} with mostly notational changes,
		{\it e.g.}, now $\varphi_\infty$ replaces $\varphi_1$.
	\end{proof}
	
	Therefore, we have
	
	\begin{thm}
		If $\varphi$ is an admissible solution of the Prandtl-Meyer reflection problem,
		then its shock curve $\Gsh$ is uniformly convex
		in the sense described in
		Theorem {\rm \ref{thm:main theorem-strictUnif}}.
		Moreover, for a weak solution
		of the Prandtl-Meyer reflection problem in the sense of
		Definition {\rm \ref{def:weak solutionPrandtlRefl}} with properties
		{\rm (\ref{PrandtlAdmSolnProperties-2})}--\eqref{PrandtlSolnProperties-3-1} of
		Definition {\rm \ref{Prandtl-admisSolnDef}}, the transonic shock $\Gsh$ is a strictly convex graph
		as in \eqref{shock-graph-inMainThm}--\eqref{strictConvexity-degenerate-graph}
		if and only if
		property
		{\rm (\ref{PrandtlAdmSolnProperties-4})} of
		Definition {\rm \ref{Prandtl-admisSolnDef}} holds.
	\end{thm}

    \begin{appendices}
	\section{\,Paths Connecting Endpoints of the Minimal and Maximal Chains}
	\label{Appendix1Section}
	
	For $\Lambda\subset\mR^n$,
	we denote
	\begin{equation}\label{A.1-a}
	\Lambda_r:=\{\xxi\in \Lambda \;: \;\dist(\xxi, \partial\Lambda)>r\}.
	\end{equation}

	\begin{lem}\label{lem:connected-path}
		Let $\Lambda\subset\mR^n$ be an open set such that
		$\Lambda_r$ is connected for each $r\in[0, r_0]$
		with given $r_0>0$.
		Let $P, Q\in \overline\Lambda$
		be such that $B_r(P)\cap\Lambda_\rho$ and $B_r(Q)\cap\Lambda_\rho$ are connected
		for each $0\le\rho< r\le r_0$. Then
		there exists a continuous curve
		$\mS$ with endpoints $P$ and $Q$
		such that
		$
		\mS^0\subset \Lambda$.
		More precisely, $\mS=g([0,1])$, where $g\in C([0, 1];\mR^n)$, $g$ is locally Lipschitz
		on $(0,1)$, $g(0)=P$, $g(1)=Q$, and $g(t)\in\Lambda$ for all $t\in (0,1)$.
	\end{lem}
	
	\begin{proof}
		We first note that, after points $P$ and $Q$ are fixed,
		we can assume that $\Lambda$ is a bounded set;
		otherwise, we replace $\Lambda$ by $\Lambda\cap B$, where $B$ is an open ball and $P, Q\in B$.
		
		We divide the proof into three steps.
		
		\smallskip
		{\bf 1.} We notice that, if $P, Q\in \Lambda_r$ for some $r\in[0, r_0)$, then
		there exists a piecewise-linear path $\mS$ with a finite number of corner points connecting
		$P$ to $Q$ such that $\mS\subset\Lambda_{r/2}$.
		This is obtained via covering $\overline{\Lambda_r}$ by balls $B_{r/2}(\xi_i)$, $i=1, \dots, N$,
		with each $\xi_i\in \overline\Lambda_r$
		and via noting that, since $\Lambda_r$ is connected, then any
		$\xi_i$ and $\xi_j$ can be connected by a piecewise-linear path with at most $N$ corners,
		each section of which is a straight segment connecting centers of two intersecting balls
		of the cover.
		
		Thus, the path connecting $\xi_i$ to $\xi_j$ lies in
		$\cup_{k=1}^N B_{r/2}(\xi_k)\subset \Lambda_{r/2}$.
		Then we connect $P$ to $Q$ by first connecting $P$ (resp. $Q$)
		to the nearest center of ball
		$\xi_i$ (resp. $\xi_j$) via a straight segment that
		lies in $B_{r_2}(\xi_i)$ (resp. $B_{r_2}(\xi_j)$),
		and next connect $\xi_i$ to $\xi_j$ as above.
		In this way, the whole path $\mS$ between $P$ and $Q$ is Lipschitz up
		to the endpoints and lies in $\Lambda_{r/2}$.
		Clearly, there exists $g\in C^{0,1}([0,1]; \mR^n)$ with
		$g(0)=P$, $g(1)=Q$, and $g(t)\in\Lambda_{r/2}$ for all $t\in[0,1]$
		such that $\mS=g([0, 1]$.
		Therefore, this lemma is proved for any $P, Q\in\Lambda$.
		
		\medskip
		{\bf 2.} Now we consider the case when $P\in\partial \Lambda$ and $Q\in \Lambda$.
		Since $\Lambda$
		is open, there exists a sequence $P_m\to P$ with $P_m\in \Lambda$ for
		$m=1,2,\dots$.
		Then $P_m\in \Lambda_{r_m}$ with $r_m>0$ and $r_m\to 0$.
		Thus, taking a subsequence, we can assume
		without loss of generality that $0<r_m< \frac{r_0}m$ for all $m$.
		
		As proved in Step 1,
		$P_1$ can be connected to $Q$ by a Lipschitz curve that is
		parameterized by $g\in C^{0,1}([\frac 12,1]; \mR^n)$ with
        $$
        g(\frac 12)=P_1,\,\,\,\, g(1)=Q, \quad\,\,\, g(t)\in\Lambda_{\tilde r} \,\, \mbox{for all $t\in[0,1]$},
		$$
		where $\tilde r>0$.
		Since $(B_r(P)\cap\Lambda)_\eps= B_{r-\eps}(P)\cap\Lambda_\eps$
		for all $\eps\in [0, \frac{r}{2})$, then
		the assumptions of this lemma allow to apply the result
		of Step 1 to sets
		$B_{r_0/m}(P)\cap \Lambda$ for $m=1, 2, \dots$.
		Thus, for each $m=1, 2, \dots$, we
		obtain a Lipschitz path
		between $P_m$ and $P_{m+1}$ which lies in $B_{r_0/m}(P)\cap \Lambda$
		and is parameterized by
		$\displaystyle g\in C^{0,1}([\frac 1{m+2},\frac 1{m+1}]; \mR^n)$ with
		\begin{equation*}
		\begin{split}
		&g(\frac 1{m+1})=P_m,\;\;\;\; g(\frac1{m+2})=P_{m+1},\\
        &g(t)\in B_{\frac{r_0}{m}}(P)\cap \Lambda_{\tilde r_m}
		\,\,\,\,\mbox{ for all $t\in[\frac 1{m+2},\frac 1{m+1}]$}.
		\end{split}
		\end{equation*}
		Combining the above together, we obtain a function $g: (0, 1]\to \mR^n$ such that
		$\displaystyle g\in C([0,1]; \mR^n)\cap C^{0,1}_{\rm loc}((0,1]; \mR^n)$ with
		\begin{equation*}
		\begin{split}
		\lim_{t\to 0+} g(t)=P,\quad g(1)=Q,\;\quad\; g(t)\in\Lambda
		\;\mbox{ for all $t\in(0,1]$}.
		\end{split}
		\end{equation*}
		This completes the proof for the case when $P\in \partial\Lambda$ and $Q\in \Lambda$.
		
		\medskip
		{\bf 3.} The remaining case for both $P, Q\in \partial\Lambda$ now readily follows,
		by connecting each
		of $P$ and $Q$ to some $C\in\Lambda$ and taking the union of the paths.
	\end{proof}
	
	\begin{lem}\label{lem:connected-path-2d}
		Let $\Omega\subset \mR^2$ satisfy the conditions stated at the beginning of
		\S {\rm \ref{ExistMinMaxChainsSection}},
		and let $r^*$ be the constant from Lemma {\rm \ref{lem:connectBall}}.
		Let $\Omega_{\rho}$ be defined as in \eqref{A.1-a}.
		Then there exists $r_0\in (0, \frac{r^*}{10}]$
		such that sets $\Omega_\rho$
		are connected for all $\rho\in[0, r_0]$,
		and sets $B_r(E)\cap \Omega_\rho$ are connected for any $E\in\overline\Omega$
		and $0\le\rho<r\le r_0$.
		Moreover, if $0\le\rho< r\le 2r_0$, $P\in\overline\Omega$, and $\dist(P, \partial\Omega)<r$, then
		\begin{equation}\label{bondariesClose}
		\dist(E, \partial\Omega\cap B_r(P))\le C\rho\qquad\;\mbox{ for each $E\in \partial\Omega_\rho\cap B_r(P)$},
		\end{equation}
		where $C$ depends only on the constants in the assumptions on $\Omega$.
	\end{lem}
	
	\begin{proof}
		Throughout this proof, $C$ denotes a universal constant, depending only on $\Omega$.
		We divide the proof into two steps.
		
		\smallskip
		{\bf 1.} We first describe the structure of $\partial\Omega_\rho$ for sufficiently small $\rho>0$
		and show that $\Omega_\rho$ is connected for such $\rho$ and (\ref{bondariesClose}) holds.
		
		Denote by $\Gamma_i$, $i=1,\dots, m$, the smooth regions of $\partial\Omega$ up to the corner points.
		Then, for $P\in\Omega$, we have
		$$
		{\dist}(P, \partial\Omega)=\min_{i=1,\dots, m}{\dist}(P, \Gamma_i).
		$$
		Denote
		$$
		\Omega_i=\{P\in \Omega\;:\; \dist(P, \partial\Omega)=\dist(P, \Gamma_i)\}.
		$$
		Using that each $\Gamma_i$ is $C^{1,\alpha}$ up to the corner points,
		and  the angles at the corner points are between $(0, \pi)$,
		we now show that there exists
		$r_0>0$ such that, for any $\rho\in (0, r_0)$ and $i=1,\dots, m$,
		the set:
		$$
		\Gamma_i^\rho:=\{P\in \Omega_i\;:\; \dist(P, \partial\Omega)=\rho\}
		$$
		is a  Lipschitz curve.
		In addition, $\Gamma_i^\rho$ is close to $\Gamma_i$ in the Lipschitz norm
		in the sense described bellow.
		
		\smallskip
		Consider first a curve $\Gamma=\{(s,t)\in \mr^2\;:\; s=g(t)\}$
		for some $g\in C^{1,\alpha}(\mR)$. Let $\rho>0$ and
		$\Gamma^\rho=\{(s,t)\in \mr^2\;:\; s>g(t), \;\dist((s,t), \Gamma)=\rho\}$.
		Fix $t_0\in \mR$  and $r>10\rho$, and
		denote $L:=\|g\|_{C^{0,1}([t_0-2r, t_0+2r])}$. Then
		we find that, for any $t_1\in[t_0-r, t_0+r]$,
		there exists $s_1\in[g(t_1)+\rho, \;g(t_1)+\rho\sqrt{L^2+1}]$
		such that $(s_1, t_1)\in\Gamma^\rho$ and
		$$
		\Gamma^\rho
		\cap\{(s,t)\in \mr^2\;:\;|t-t_0|\le r,\; s>s_1+L|t-t_1|\}=\emptyset
		$$
		by noting that $B_\rho(s_1, t_1)\cap \Gamma=\emptyset$.
		From this,
		$$
		\Gamma^\rho=\{(s,t)\in \mr^2\;:\; s=g^\rho(t)\}
		$$
        with $g^\rho\in C^{0,1}_{\rm loc}(\mr)$ and $\|g-g^\rho\|_{L^\infty([-r,r])}\le \rho\sqrt{L^2+1}$.
		Moreover,
		fix $P\in \Gamma^\rho$. Then there exists $Q\in\Gamma$ such that $\dist(P, Q)=\rho$.
		It follows that
		$$
		B_\rho(P)\cap\Gamma=\emptyset,\quad\;B_\rho(Q)\cap\Gamma^\rho=\emptyset.
		$$
		From this,  for any $r\in(0,1)$,
		we find that  there exists $r_0\in(0,\frac{r}{10}]$
		depending only on $r$, $\alpha$, and  $\hat L:=\|g\|_{C^{1,\alpha}([3r, 3r])}$
		such that, if
		$\rho\in(0, r_0]$, then, for any $P=(g^\rho(t_P), t_P)\in  \Gamma^\rho\cap\{t\in[-r, r]\}$,
		we have
		\begin{align*}
		g^\rho(t)&\ge  g^\rho(t_P)+ g'(t_Q)(t-t_P)-\hat L r^\alpha|t-t_P|\\
		&\ge  g^\rho(t_P)+ g'(t_P)(t-t_P)-\hat L (r^\alpha+\rho^\alpha)|t-t_P|
		\end{align*}
        for any $t\in[-2r, 2r]$,
        where $Q:=(g(t_Q), t_Q)$ a point such that $\dist(P, Q)=\rho$.
		
		Then, noting that
		\begin{align*}
		|g(t)- g(t_P)- g'(t_P)(t-t_P)|\le L r^\alpha|t-t_P|\qquad\,\mbox{for any $t\in[-2r, 2r]$},
		\end{align*}
		and $\|g-g^\rho\|_{L^\infty([-r,r])}\le \rho\sqrt{\hat L^2+1}$,
		we have
		$$
		\|g-g^\rho\|_{C^{0,1}([-r,r])}\le \hat L \rho^\alpha+\rho\sqrt{\hat L^2+1}.
		$$
		Thus, for any $\varepsilon\in (0,1)$, reducing $r_0$,
		we obtain
		\begin{equation}\label{gDifFg0Lip}
		\|g-g^\rho\|_{C^{0,1}([-r,r])}\le \varepsilon\qquad \; \mbox{ if $\rho\le r_0$}.
		\end{equation}

		From this,  under the conditions of Case (\ref{domainSmooth-case})
		in the proof of Lemma \ref{lem:connectBall},
		when (\ref{domainSmoothCase}) holds, we follow the argument in the proof of Lemma \ref{lem:connectBall}
		and  choosing sufficiently small $r_0$ and $\varepsilon$ in (\ref{gDifFg0Lip}) to obtain that,
		for any positive $\rho<\min\{r, r_0\}$,
		\begin{equation}\label{domainSmoothCase-r}
		\begin{split}
		&\Omega_\rho\cap Q_{\frac{3r}2}=\{(s,t)\in Q_{\frac{3r}2}\;:\; s>g^\rho(t)\},\\
		&\partial\Omega^\rho\cap Q_{\frac{3r}2}=\{(s,t)\in Q_{\frac{3r}2}\;:\; s=g^\rho(t)\}.
		\end{split}
		\end{equation}
		Furthermore, under the conditions of Case (\ref{domainNearCorner-case})
		in the proof of Lemma \ref{lem:connectBall},
		when (\ref{domainNearCorner})--(\ref{domainNearCorner-funct}) hold,
		we repeat the argument there by choosing small $r_0$ and $\varepsilon$, and conclude that, for any positive $\rho<\min\{r, r_0\}$,
		\begin{equation}\label{domainNearCorner-r}
		\begin{split}
		&\Omega^\rho\cap Q_{3N r}=\{(s,t)\in Q_{3N r}\;:\; s>\max(g_1^\rho(t), g_2^\rho(t))\},\\[0.5mm]
		&\partial\Omega^\rho\cap Q_{3N r}=\{(s,t)\in Q_{3Nr}\;:\; s=\max(g_1^\rho(t), g_2^\rho(t))\},
		\end{split}
		\end{equation}
		where $g_1^\rho$ and $g_2^\rho$ satisfy
		(\ref{gDifFg0Lip}) with $g_1$ and $g_2$, respectively,
		and that there exists $t_\rho\in (-C\rho, C\rho)$ such that
		\begin{equation}\label{domainNearCorner-funct-r}
		\begin{split}
		&g_1^\rho(t)>g_2^\rho(t)\;\mbox{  for $t<t_\rho$},\qquad
		g_1^\rho(t)<g_2^\rho(t)\;\mbox{  for $t>t_\rho$}.
		\end{split}
		\end{equation}
		
		We adjust $r_0$ so that $r_0\le \frac{r^*}{10}$.
		Then,  from  (\ref{domainSmoothCase-r})--(\ref{domainNearCorner-funct-r})
		with $r=r^*$,
		we obtain that,
		for each $\rho\in(0, r_0]$,
		$\partial\Omega_\rho$ is a Lipschitz curve without self-intersection.
		It follows that $\Omega_\rho$ is
		simply-connected.
		
		Also, combining  (\ref{domainSmoothCase-r}) with (\ref{domainSmoothCase}) and
		(\ref{domainNearCorner-r})--(\ref{domainNearCorner-funct-r})
		with (\ref{domainNearCorner})--(\ref{domainNearCorner-funct})
		for $r=r_0$,
		choosing $\varepsilon$ small in (\ref{gDifFg0Lip}) for $g, g_1$, and $g_2$,
		and adjusting $r_0$,
		we have
		$$
		\dist(\partial\Omega_\rho, \partial\Omega)\le C\rho \qquad\mbox{for each $\rho\in (0, r^0)$}.
		$$
		Then we conclude (\ref{bondariesClose}).

		\medskip
		{\bf 2.} Now we show that $B_r(E)\cap\Omega_\rho$ is connected for any $E\in \Omega$
		and $0\le\rho<r\le r_0$.
		
		\smallskip
		Assume that ${\dist}(E, \partial\Omega)<2r$ (otherwise, the result already holds).
		Since $r_0\le \frac{r^*}{10}$, we obtain (\ref{domainSmoothCase})--(\ref{domainNearCorner-funct})
		for $2r$ instead of $r$,
		so that (\ref{domainSmoothCase-r})--(\ref{domainNearCorner-funct-r}) hold for $2r$ instead of $r$.
		Then, arguing as in the proof of Lemma \ref{lem:connectBall} and possibly reducing $r_0$,
		we obtain the following:
		
		\begin{itemize}
			
			\smallskip
			\item  If $B_r(E)\cap\Omega$ has expression
			(\ref{connectedRnbhd-smooth}), then
			$$
			\Omega_\rho\cap B_{r}(E)=\{(s,t)\;:\;t\in(t^-_\rho, t^+_\rho),\;
			\max(f^-(t), g^\rho(t))<s<f^+(t)\},
			$$
			where $t^+_\rho\in (\frac{9r}{10}, r]$
			and $t^-_\rho\in [-r, -\frac{9r}{10})$
			with $|t^\pm_\rho-t^\pm|\le C\rho$,
			$f^+>g^\rho$ on $(t^-_\rho, t^+_\rho)$, and $f^+<g^\rho$
			on $[-r, r]\setminus [t^-_\rho, t^+_\rho]$;
			
			\smallskip
			\item If  $B_r(E)\cap\Omega$ has expression
			(\ref{connectedRnbhd-corner}), then
			$$
			\Omega_\rho\cap B_{r}(E)=\{(s,t)\;:\;t\in(t^-_\rho, t^+_\rho),\;
			\max(f^-(t), g_1^\rho(t), g_2^\rho(t))<s<f^+(t)\},
			$$
			where  $t^-_\rho\in [t^*- r, t^*)$ and $t^+_\rho\in (t^*, t^*+r]$
			with $|t^\pm_\rho-t^\pm|\le C\rho$, and
			$f^+(t)>\max(g_1^\rho(t), g_2^\rho(t))$ on $(t^-_\rho, t^+_\rho)$.
		\end{itemize}
		The facts above imply that sets $B_r(E)\cap\Omega_\rho$ are connected.
	\end{proof}
	
	In the next lemma, we use the minimal and maximal chains in the sense of
	Definition \ref{def:minimal chain}.
	
	\begin{lem}\label{lem:connected-path-inChain}
		Let $\Omega\subset \mR^2$ satisfy the conditions stated at the beginning of
		\S {\rm \ref{ExistMinMaxChainsSection}},
		and let $r_0$ be the constant
		from Lemma {\rm \ref{lem:connected-path-2d}}.
		Let $E_1, E_2\in \overline\Omega$, and let there exist
		a minimal or maximal chain $\{E^i\}_{i=1}^N$ of radius $r_1\in(0, r_0]$ connecting $E_1$ to $E_2$
		in $\Omega$,
		{\it i.e.}, $E^0=E_1$ and $E^N=E_2$. Denote
		$$
		\Lambda=\bigcup_{i=0}^N B_{r_1}(E^i)\cap\Omega
		$$
		so that $E_1, E_2\in \partial\Lambda$.
		Then there exists $\hat r_0>0$ such that
		set $\Lambda$ and points $\{E_1, E_2\}$
		satisfy the conditions of Lemma {\rm \ref{lem:connected-path}}
		with radius $\hat r_0$.
	\end{lem}
	
	\begin{proof} We divide the proof into two steps.
		
		\smallskip
		{\bf 1.} We first show the existence of $\hat r_0\in (0, r_1)$ such that $\Lambda_\rho$ is connected
		for each $\rho\in (0, \hat r_0]$.
		We recall that $r_1\le r_0\le r^*$ so that the conclusions
		of Lemma \ref{lem:connectBall} hold for $B_{r_1}(E^i)$.
		
		Since, for each $P\in\Lambda$,
		$$
		{\dist}(P, \partial\Lambda)=\min\big\{ {\dist}(P, \; \partial(\bigcup_{i=0}^N B_{r_1}(E^i))),\;
		{\dist}(P, \;\partial\Omega)
		\big\},
		$$
		then
		\begin{equation}\label{LambdaRhoForChain}
		\Lambda_\rho=  \bigcup_{i=0}^N B_{r_1-\rho}(E^i) \cap \Omega_\rho.
		\end{equation}

		By Lemma \ref{lem:connectBall}(\ref{lem:connectBall-i2}) and property
		(\ref{def:minimal chain-Ib}) of
		Definition \ref{def:minimal chain}, we see that,
		if $r_1\le r^*$, then $B_{r_1}(E^i) \cap B_{r_1}(E^{i+1}) \cap \Omega\ne\emptyset$
		for $i=0, \dots, N-1$.
		Note that all the sets in the last intersection are open.
		Then, recalling that $r_1\le r_0$ and using (\ref{bondariesClose}) in
		Lemma \ref{lem:connected-path-2d},
		we obtain that there exists $\hat r^0\in (0, r_1)$ such that,
		for any $\rho\in (0, \hat r_0)$,
		$$
		B_{r_1-\rho}(E^i)\cap  B_{r_1-\rho}(E^{i+1}) \cap \Omega_\rho\ne\emptyset\qquad\,
		\mbox{for $i=0, \dots, N-1$}.
		$$
		Also, from Lemma \ref{lem:connected-path-2d},
		sets $B_{r_1-\rho}(E^i)\cap\Omega_\rho$ are connected, since $r_1\le r_0$.
		Then we obtain that
		$\displaystyle  \bigcup_{i=0}^N B_{r_1-\rho}(E^i) \cap \Omega_\rho$
		is connected by using the argument in the proof of
		Lemma \ref{lem:connected-chain}(\ref{lem:connected-chain-i1}).
		Thus, by (\ref{LambdaRhoForChain}),
		we conclude that
		$\Lambda_\rho$ is connected for all $\rho\in (0, \hat r_0)$.
		
		\medskip
		{\bf 2.} Since $B_{r_1}(E^0)\cap\Omega\subset\Lambda$, then we use (\ref{LambdaRhoForChain})
		to obtain
		$$
		B_{r}(E^0)\cap\Lambda_\rho
		=B_{r}(E^0)\cap\Omega_\rho
		\qquad \mbox{ for all $r\in(0, \frac {r_1} {10}]$ and $\rho\in (0, r)$}.
		$$
		Sets $B_{r}(E^0)\cap\Omega_\rho$ with $r$ and $\rho$ as above are connected by
		Lemma \ref{lem:connected-path-2d}.
		Thus, the assumptions of Lemma \ref{lem:connected-path}
		with radius $\frac {r_1} {10}$ hold for point $E_1=E^0$.
		For point $E_2=E^N$, the argument is similar.
	\end{proof}
\end{appendices}

 \section*{Conflict of interest}

 The authors declare that they have no conflict of interest.

\end{document}